\documentclass[11pt]{article}

\usepackage{epsfig,epsf,fancybox}
\usepackage{amsmath}
\usepackage{mathrsfs}
\usepackage{amssymb}
\usepackage{graphicx}
\usepackage{color}
\usepackage{multirow}
\usepackage{paralist}
\usepackage{verbatim}
\usepackage{galois}
\usepackage{algorithm}
\usepackage{algorithmic}
\usepackage{bbm}
\usepackage[left=1in,top=1in,right=1in,bottom=1in,letterpaper]{geometry}
\usepackage[linktocpage,colorlinks,linkcolor=blue,anchorcolor=blue,citecolor=blue,urlcolor=blue]{hyperref}

\newtheorem{theorem}{Theorem}[section]
\newtheorem{corollary}[theorem]{Corollary}
\newtheorem{lemma}[theorem]{Lemma}

\newtheorem{remark}[theorem]{Remark}

\newcommand{\st}{\textnormal{s.t.}}

\newcommand{\diag}{\textnormal{diag}}
\newcommand{\Diag}{\textnormal{Diag}\,}

\newcommand{\argmin}{\mathop{\rm argmin}}
\newcommand{\LCal}{\mathcal{L}}

\newcommand{\QCal}{\mathcal{Q}}
\newcommand{\SCal}{\mathcal{S}}

\newcommand{\pres}{\textsf{pres}}
\newcommand{\dres}{\textsf{dres}}
\newcommand{\dgap}{\textsf{dgap}}
\newcommand{\pinfeas}{\textsf{pinfeas}}
\newcommand{\dinfeas}{\textsf{dinfeas}}

\newcommand{\x}{\mathbf x}
\newcommand{\y}{\mathbf y}
\newcommand{\s}{\mathbf s}
\newcommand{\sr}{\mathbf r}

\newcommand{\bc}{\mathbf c}
\newcommand{\bb}{\mathbf b}
\newcommand{\h}{\mathbf h}
\newcommand{\e}{\mathbf e}
\newcommand{\z}{\mathbf z}
\newcommand{\w}{\mathbf w}
\newcommand{\bl}{\mathbf l}
\newcommand{\su}{\mathbf u}
\newcommand{\sv}{\mathbf v}
\newcommand{\p}{\mathbf p}

\newcommand{\q}{\mathbf q}

\newcommand{\br}{\mathbb{R}}

\newcommand{\ba}{\begin{array}}
\newcommand{\ea}{\end{array}}

\newcommand{\etal}{{\it et al.\ }}

\newcommand{\indfunc}{{\mathbbm{1}}}
\newcommand{\indmat}{{\mathbbm{I}}}
\newcommand{\nn}{\nonumber}

\title{An ADMM-Based Interior-Point Method \\ for Large-Scale Linear Programming}
\author{
 Tianyi Lin
 \thanks{Department of Industrial Engineering and Operations Research, UC Berkeley, Berkeley, CA 94720, USA.} \and
 Shiqian Ma
 \thanks{Department of Mathematics, University of California, Davis, CA 95616, USA.} \and
 Yinyu Ye
 \thanks{Department of Management Science and Engineering, Stanford University, Stanford, CA 94305, USA. } \and
 Shuzhong Zhang
 \thanks{Department of Industrial and Systems Engineering, University of Minnesota, Minneapolis, MN 55455, USA, and Institute of Data and Decision Analytics, The Chinese University of Hong Kong, Shenzhen, 518172, China. }
}
\date{\today} 

\begin{document}
\maketitle

\begin{abstract}
In this paper, we propose a new framework to implement interior point method (IPM) in order to solve some very large scale linear programs (LP). Traditional IPMs typically use Newton's method to approximately solve a subproblem that aims to minimize a log-barrier penalty function at each iteration. Due its connection to Newton's method, IPM is often classified as {\it second-order method} -- a genre that is attached with stability and accuracy at the expense of scalability. Indeed, computing a Newton step amounts to solving a large system of linear equations, which can be efficiently implemented if the input data are reasonably-sized and/or sparse and/or well-structured. However, in case the above premises fail, then the challenge still stands on the way for a traditional IPM. To deal with this challenge, one approach is to apply the iterative procedure, such as preconditioned conjugate gradient method, to solve the system of linear equations. Since the linear system is different each iteration, it is difficult to find good pre-conditioner to achieve the overall solution efficiency. In this paper, an alternative approach is proposed. Instead of applying Newton's method, we resort to the alternating direction method of multipliers (ADMM) to approximately minimize the log-barrier penalty function at each iteration, under the framework of primal-dual path-following for a homogeneous self-dual embedded LP model. The resulting algorithm is an ADMM-Based Interior Point Method, abbreviated as \textsf{ABIP} in this paper. The new method inherits stability from IPM, and scalability from ADMM. Because of its self-dual embedding structure, \textsf{ABIP} is set to solve any LP without requiring prior knowledge about its feasibility. We conduct extensive numerical experiments testing \textsf{ABIP} with large-scale LPs from NETLIB and machine learning applications. The results demonstrate that \textsf{ABIP} compares favorably with other LP solvers including \textsf{SDPT3}, \textsf{MOSEK}, \textsf{DSDP-CG} and \textsf{SCS}.

\vspace{0.5cm}

\end{abstract}

\vspace{0.3cm}

\begin{center}

{\large \sf This paper is dedicated to the special issue in memory of Professor Masao Iri}

\end{center}

\maketitle

\newpage

\section{Introduction}
{By and large, traditional interior point method (IPM) for linear program (LP) is based on solving a sequence of log-barrier penalty subproblems using Newton's method \cite{Iri-1986-multiplicative,Renegar-1988-polynomial,NN-1994,Ye-IPM-1997}.} It turns out that with a suitable penalty-parameter choice scheme, one step of Newton's method usually yields a very good initial solution for the next log-barrier penalty subproblem. {As a result, the crux of IPMs boils down to the computation of Newton steps, which requires to solve system of linear equations (e.g., Iri and Imai~\cite{Iri-1986-multiplicative}, Gill {\it et al.}~\cite{Gill-1986-Projected}, Renegar~\cite{Renegar-1988-polynomial} and Ye and Kojima~\cite{Ye-1987-recovering}).} On the surface of it, computing Newton's direction then amounts to the Cholesky decomposition of a large (often ill-conditioned) matrix or even computing its inverse, which can be done when the dimensions are not exceedingly high, or the input data are sparse and/or well-structured. In general however, computing Newton's direction is computationally expensive. Because of this connection to Newton's method, IPM is often classified as a second-order approach. Typical to the second-order methods, IPM is known to be stable and accurate, but computing Newton's direction remains a challenge when the problem is dense and large scale. On the other hand, recently there has been considerable research attention paid on the so-called first-order methods, in that no Newton direction is needed. One very successful first-order method is called alternating direction method of multipliers (ADMM), which is essentially a gradient-based algorithm for the dual of a linearly constrained optimization model; see, for example, \cite{Glowinski1975b, Lions1979, Gabay-83, Glowinski1989, Eckstein-1992-douglas} and the recent survey papers \cite{Boyd-2011-distributed,Eckstein2012}. It turns out that the ADMM is highly scalable; however, it may suffer from numerical instability and it may take overly many iterations to compute an accurate solution. A natural question thus arises: Can we combine the benefits from both campuses? This paper aims to provide an affirmative answer to the afore question.


Before moving further to that question, let us first mention a beautiful technique which resolved a difficulty that baffled researchers in the early days of the IPM.
The difficulty is that an IPM requires an interior feasible solution to begin with, but an LP may not even be feasible let alone availability of an interior feasible solution.
To tackle this, Ye {\em et al.}~\cite{Ye-1994-hsd} proposed 
a homogeneous self-dual (HSD) reformulation of the original LP, which contains all the information that one may possibly care to obtain: an optimal solution if it exists; or, if the problem is infeasible or unbounded, a certificate that proves the case. Interestingly, the HSD has a ready interior feasible solution, while an optimal solution is guaranteed to exist. A central path following algorithm was subsequently proposed in \cite{Ye-1994-hsd} to solve the HSD.
Later, Xu {\em et al.}~\cite{Xu-1996-simplified} proposed a simplified homogeneous self-dual model.
However, in either cases computing the Newton directions remains to be the chore.
In this paper, we propose to apply the ADMM to approximately solve the log-barrier penalty subproblems from the path-following scheme for the HSD. The resulting algorithm, ADMM-based IPM (abbreviated as \textsf{ABIP} henceforth), inherits advantages of both IPM and ADMM. It turns out that \textsf{ABIP} is robust, highly scalable, and achieves high accuracy. Moreover, as a benefit inherited from HSD, \textsf{ABIP} finds both primal and dual optimal solutions if they exist; otherwise it finds a certificate proving primal or dual infeasibility.

There have been lots of efforts to efficiently compute Newton's step in IPMs. Existing approaches adopted in IPM for computing Newton's step include sparse matrix factorization, Krylov subspace method and preconditioned conjugate gradient method (cf.~\cite{Davis-2006-direct, Gondzio-2012-interior, Gondzio2012, Fountoulakis2014, Benson2000, Cui-2019-Implementation}). {Very recently, Cui {\it et al.}~\cite{Cui-2019-Implementation} proposed a novel inner-iteration preconditioned Krylov subspace method which overcomes the severe ill-conditioning of linear equations solved in the final phase of interior-point iterations. Despite several progresses along this direction, the performance of these methods still highly depends on the structure of the problem and the input data and suffers from the curse of dimensionality.} For example, the system of linear equation data of the IPM is changing every iteration and becomes more and more ill-conditioned, so that it is typically hard to find a good pre-conditioner for the CG method. Moreover, often we have to solve more than one system of linear equations with different right-hand-side vectors. By investigating the structure of the problem in HSD, we discover that ADMM can be used to approximately solve the log-barrier penalty problems very efficiently. Though connected with the classical operator splitting methods in \cite{Gabay-83,Fortin-Glowinski-1983,Eckstein-1992-douglas}, renaissance of ADMM in recent years was due to its success in signal processing \cite{Combettes2007b,Yang2011}, image processing \cite{Goldstein2009a,Yang2010}, and machine learning problems; see a recent survey paper \cite{Boyd-2011-distributed} and the references therein.
An interesting and positive discovery through our study reported in this paper is twofold: 1) we find that the overall IPM strategy such as central-path following greatly improve the stability and robustness of the variable-splitting approach; 2) it turns out that the log-barrier penalty subproblem under the central path following scheme for the HSD is well-structured and can be efficiently solved by ADMM.

Recently, there are several attempts on designing first-order methods for solving LPs (and conic programs); see \cite{Wen-2010-alternating, Yang-2015-sdpnal+, Yen-2015-sparse, Wang-2017-New}. Without an HSD framework, such methods are not suited for primal or dual infeasible problems. 
More recently, the ADMM-based solvers for LP have been explored prior to our method, e.g., a split conic solver (\textsf{SCS}) developed by O'Donoghue \etal \cite{Donoghue-2016-conic}, which applies ADMM to solve the homogeneous self-dual embedding of the original LP. Therefore, \textsf{ABIP} is similar to \textsf{SCS} in that they are both based on ADMM and HSD. However, conceptually \textsf{ABIP} is built within the IPM framework, so that it can be used as an optional solver when the data is large and dense for any existing IPM solver. It is our hope that, by combining ADMM with interior-point strength, the solver would have an additional machinery to improve its solution robustness. For example, the scheme of \textsf{ABIP} can easily incorporate the notion of step-sizes and wide neighborhoods of the central path (cf.~\cite{Sturm-Zhang-1997}). The efficacy of \textsf{ABIP} is confirmed by our extensive numerical tests on large-scale LPs from NETLIB and machine learning applications. Finally, we remark that extending \textsf{ABIP} to a more general convex conic optimization setting is straightforward.

\subsection{Notation and Organization}
Throughout this paper, we denote vectors by bold lower case letters, e.g., $\x$, and matrices by regular upper case letters, e.g., $X$. The transpose of a real vector $\x$ is denoted as $\x^\top$. For a vector $\x$, and a matrix $X$, $\left\|\x\right\|$ and $\left\|X\right\|$ denote the $\ell_2$ norm and the matrix spectral norm, respectively. For two symmetric matrices $A$ and $B$, $A \succeq B$ indicates that $A-B$ is symmetric positive semi-definite. {The superscript, e.g., $\x^t$,} denotes iteration counter. $\log(\x)$ denotes the natural logarithm of $\x$. $\e$ denotes the vector of all ones. $e_j$ denotes the coordinate vector with $j$-th entry being 1. {$\indmat$ is an identity matrix with appropriate dimension.} For two vectors $\x$ and $\y$, the Hadamard product is denoted as $\x\circ\y = \left(x_1y_1, \ldots, x_n y_n\right)$.

The rest of the paper is organized as follows. In Section \ref{section2:background}, we discuss some background of homogeneous self-dual embedding. In Section \ref{Section3:Approach}, we propose our \textsf{ABIP} method for solving the homogeneous self-dual embedding with log barrier functions. We also discuss how \textsf{ABIP} can be simplified and reduced to a matrix-free algorithm. The iteration complexity of \textsf{ABIP} is also analyzed.
In Section \ref{section:implementation}, we propose several techniques that help improve the performance of \textsf{ABIP} in practice. In Section \ref{Section5:Experiment}, we present extensive numerical results on large-scale LPs from NETLIB and machine learning applications and compare with several existing LP solvers. We make some concluding remarks in Section \ref{Section6:Conclusion}.

\section{Homogeneous and Self-Dual Linear Programming}\label{section2:background}
We are interested in solving the following \textit{primal-dual} pair of linear programs (LP):
\begin{equation}\label{prob:LP}
\ba{llllll}
& \min & \bc^\top\x & & \max & \bb^\top\y \\
(P) \quad & \st  & A\x = \bb, & \quad\quad\quad (D) \quad & \st  & A^\top \y + \s = \bc, \\
\quad & & \x\geq 0,   & \quad\quad\quad\quad & & \s\geq 0,
\ea\end{equation}
where $\x\in\br^n$ is the primal variable, $\y\in\br^m$ and $\s\in\br^n$ are the dual variables, the problem data are $A\in\br^{m\times n}$, $\bb\in\br^m$ and $\bc\in\br^n$ with $m\leq n$, and without loss of generality, we assume that $A$ is of full row rank. The primal and dual optimal objective values are denoted as $p^*$ and $d^*$ respectively. 

In addition to the celebrated simplex method of Dantzig~\cite{Dantzig-1963} in 1940's, the interior point method (IPM), which was pioneered by Karmarkar~\cite{Karmarkar-1984} and intensively developed by many researchers in the 1980's and 1990's, has been a standard approach to solve linear program \eqref{prob:LP}.
In the early years of IPM, initial feasible interior solutions were assumed to be available at hand. Clearly, this assumption can be restrictive.
To address this issue specifically, Ye \etal \cite{Ye-1994-hsd} proposed to solve the following homogeneous and self-dual linear programming with arbitrary initial points $\x^0>0$, $\s^0>0$ and $\y^0$:
\begin{eqnarray}\label{HLP}
(HSD) \qquad
\begin{array}{crrrrrrrl}
\min & & & & & & ((\x^0)^\top\s^0+1)\theta & & \\
\st  & & A\x & - \bb\tau & & & + \bar{\bb}\theta & = & 0, \\
& -A^\top \y & & + \bc\tau & -\s & & -\bar{\bc}\theta & = & 0, \\
& \bb^\top \y & - \bc^\top \x &  &  & -\kappa & + \bar{\z}\theta & = & 0, \\
& -\bar{\bb}^\top \y & + \bar{\bc}^\top \x & -\bar{\z}\tau & & & & = & -(\x^0)^\top \s^0-1, \\
& \y \ \mbox{free}, & \x \geq 0, & \tau \geq 0, & \s \geq 0, & \kappa \geq 0, & \theta \mbox{ free}, & &
\end{array}
\end{eqnarray}
where
\begin{equation}\label{bar-b-c-z}
\bar{\bb}=\bb-A\x^0, \quad \bar{\bc}=\bc-A^\top \y^0-\s^0, \quad \bar{\z}=\bc^\top \x^0+1-\bb^\top \y^0.
\end{equation}
The HSD \eqref{HLP} has many nice properties. In the following we give a partial list (cf.~\cite{Ye-1994-hsd}).

\begin{theorem}[Theorem 2 in \cite{Ye-1994-hsd}]
The following holds for \eqref{HLP}:
\begin{enumerate}[(i)]
\item The optimal value of \eqref{HLP} is zero, and for any feasible point $(\y,\x,\tau,\theta,\s,\kappa)$, it holds
\begin{equation*}
((\x^0)^\top \s^0+1)\cdot\theta = \x^\top \s+\tau\kappa .
\end{equation*}
\item There is a feasible solution $(\y, \x, \tau, \theta, \s,\kappa)$ to \eqref{HLP} such that
\begin{equation*}
\y = \y^0, \ \x = \x^0, \ \tau = 1, \ \theta = 1, \ \s = \s^0, \ \kappa = 1.
\end{equation*}
\item There is an optimal solution $(\y^*, \x^*, \tau^*,\theta^*=0, \s^*, \kappa^*)$ such that $\x^*+\s^*>0$ and $\tau^*+\kappa^*>0$, which is called a strictly complementary solution.
\end{enumerate}
\end{theorem}

\begin{theorem}[Theorem 3 in \cite{Ye-1994-hsd}]
Let $(\y^*,\x^*,\tau^*,\theta^*=0,\s^*,\kappa^*)$ be a strictly complementary solution for \eqref{HLP}. Then:
\begin{enumerate}[(i)]
\item (P) has a solution (feasible and bounded) if and only if $\tau^*>0$. In this case, $\x^*/\tau^*$ is an optimal solution for (P) and $(\y^*/\tau^*,\s^*/\tau^*)$ is an optimal solution for (D).
\item If $\tau^*=0$, then $\kappa^*>0$, which implies that $\bc^\top \x^*-\bb^\top \y^*<0$, i.e., at least one of $\bc^\top \x^*$ and $-\bb^\top \y^*$ is strictly less than zero. If $\bc^\top \x^*<0$ then (D) is infeasible; if $-\bb^\top \y^*<0$ then (P) is infeasible; and if both $\bc^\top \x^*<0$ and $-\bb^\top \y^*$ then both (P) and (D) are infeasible.
\end{enumerate}
\end{theorem}

\begin{theorem}[Corollary 4 in \cite{Ye-1994-hsd}]
Let $(\bar{\y},\bar{\x},\bar{\tau},\bar{\theta}=0,\bar{\s},\bar{\kappa})$ be any optimal solution for \eqref{HLP}. If $\bar{\kappa}>0$, then either (P) or (D) is infeasible.
\end{theorem}

\section{An ADMM-based Interior-Point Method}\label{Section3:Approach}

In \cite{Ye-1994-hsd}, Ye \etal proposed an $O\left(\sqrt{n}\log\left(\frac{1}{\varepsilon}\right)\right)$-iteration and $O\left(n^3\log\left(\frac{1}{\varepsilon}\right)\right)$-arithmetic operation interior point algorithm to solve \eqref{HLP}. However, like all interior-point methods, it requires solving a linear system at each iteration and therefore does not scale well for dense data. In this section, we propose our \textsf{ABIP} method which uses ADMM to solve the log-barrier penalty subproblems for HSD, and we show that the procedures of \textsf{ABIP} can be simplified. 
In this section, we also provide an iteration complexity analysis for \textsf{ABIP}.

\subsection{The ABIP Method}
For ease of presentation, we choose $\y^0=0$, $\x^0=\e$, and $\s^0=\e$, where $\e$ denotes the vector of all ones. By introducing a constant parameter $\beta>0$ and constant variables $\sr=0$ and $\xi=-(\x^0)^\top \s^0-1=-n-1$, \eqref{HLP} can be rewritten as
\begin{eqnarray}\label{HLP-re}
\begin{array}{ll}
\min & \beta(n+1)\theta + \indfunc(\sr=0) + \indfunc(\xi=-n-1) \\
\st  & Q\su = \sv, \\
& \y \mbox{ free}, \ \x \geq 0, \ \tau \geq 0,  \ \theta \mbox{ free}, \ \s\geq 0, \ \kappa\geq 0,
\end{array}
\end{eqnarray}
where
\[Q = \begin{bmatrix} 0 & A & -\bb & \bar{\bb} \\ -A^\top & 0 & \bc & -\bar{\bc} \\ \bb^\top & -\bc^\top & 0 & \bar{\z} \\ -\bar{\bb}^\top & \bar{\bc}^\top & -\bar{\z} & 0 \end{bmatrix}, \quad \su = \begin{bmatrix} \y \\ \x \\ \tau \\ \theta \end{bmatrix}, \quad \sv = \begin{bmatrix} \sr \\ \s \\ \kappa \\ \xi \end{bmatrix}, \]
\begin{equation}\label{bar-b-c-z-re}
\bar{\bb}=\bb-A\e, \quad \bar{\bc}=\bc-\e, \quad \bar{\z}=\bc^\top\e+1,
\end{equation}
and the indicator function $\indfunc({\cal C})$ equals zero if the constraint ${\cal C}$ is satisfied, and equals $+\infty$ otherwise.
The reason that we introduce a parameter $\beta$ in the objective is completely for ease of presentation. It does not change the solution of the problem.

One classical way to solve \eqref{HLP-re} is to use log-barrier penalty to penalize the variables with non-negativity constraints, which results in a primal-dual interior-point method. The new formulation with log-barrier penalty is:
\begin{equation}\label{prob:HSD-log}
\ba{ll}
\min & B(\su, \sv, \mu), \\
\st  & Q\su = \sv,
\ea
\end{equation}
where $B(\su, \sv, \mu)$ is a barrier function defined as follows:
{\begin{equation}\label{prob:barrier}
B(\su, \sv, \mu) = \beta(n+1)\theta + \indfunc(\sr=0) + \indfunc(\xi=-n-1) -\mu\sum_i\log(\x_i) - \mu\sum_i\log(\s_i) - \mu\log(\tau) - \mu\log(\kappa),
\end{equation}}and $\mu>0$ is the penalty parameter.
In the $k$-th iteration of IPM, one uses Newton's method to solve the KKT system of \eqref{prob:HSD-log} with $\mu=\mu^k$. One then reduces $\mu^k$ to $\mu^{k+1}$ for the next iteration. When $\mu^k\rightarrow 0$, the solution of \eqref{prob:HSD-log} approaches that of \eqref{HLP-re}. The computational bottleneck of IPM is that
one has to assemble a
Newton's direction, which can be expensive when the problem is large and data are dense.

Observing the structure of \eqref{prob:HSD-log}, we propose to use the Alternating Direction Method of Multipliers (ADMM) to solve it inexactly. 
To do so, we first rewrite \eqref{prob:HSD-log} as the following problem by introducing auxiliary variables $(\tilde{\su},\tilde{\sv})$:
\begin{equation}\label{prob:HSD-admm}
\ba{ll}
\min & \indfunc(Q\tilde{\su}=\tilde{\sv}) + B(\su, \sv, \mu^k), \\
\st  & (\tilde{\su},\tilde{\sv}) = (\su,\sv).
\ea
\end{equation}
By associating (scaled) Lagrange multipliers $\p$ to constraint $\tilde{\su}=\su$ and $\q$ to constraint $\tilde{\sv}=\sv$, the augmented Lagrangian function for \eqref{prob:HSD-admm} can be written as
\[
\LCal_\beta(\tilde{\su},\tilde{\sv},\su,\sv,\mu^k,\p,\q):=\indfunc(Q\tilde{\su}=\tilde{\sv})+ B(\su,\sv,\mu^k)-\langle \beta(\p,\q),(\tilde{\su},\tilde{\sv})-(\su,\sv)\rangle + \frac{\beta}{2}\|(\tilde{\su},\tilde{\sv})-(\su,\sv)\|^2,
\]
where $\beta>0$ is the same parameter as in \eqref{HLP-re}.
The $i$-th iteration of ADMM for solving \eqref{prob:HSD-admm} is as follows:
\begin{align}
(\tilde{\su}_{i+1}^k,\tilde{\sv}_{i+1}^k) & = \argmin_{\tilde{\su},\tilde{\sv}} \ \LCal_\beta(\tilde{\su},\tilde{\sv},\su_i^k,\sv_i^k,\mu^k,\p_i^k,\q_i^k) = \prod_{Q\su=\sv} (\su_i^k+\p_i^k, \sv_i^k+\q_i^k), \label{admm-typical-1} \\
({\su}_{i+1}^k,{\sv}_{i+1}^k) & = \argmin_{{\su,\sv}} \ \LCal_\beta(\tilde{\su}_{i+1}^k,\tilde{\sv}_{i+1}^k,\su,\sv,\mu^k,\p_i^k,\q_i^k)   \label{admm-typical-2}, \\
(\p_{i+1}^k,\q_{i+1}^k) & = (\p_i^k,\q_i^k) - (\tilde{\su}_{i+1}^k,\tilde{\sv}_{i+1}^k) + (\su_{i+1}^k,\sv_{i+1}^k), \label{admm-typical-3}
\end{align}
where $\prod_{\SCal}(\x)$ denotes the Euclidean projection of $\x$ onto the set $\SCal$. {A generic description of our proposed 
approach -- \textsf{ABIP} -- is sketched as Algorithm \ref{Algorithm:Basic} below.}

\begin{algorithm}[ht]
\caption{The Basic Algorithmic Framework of \textsf{ABIP} for Linear Programming} \label{Algorithm:Basic}
\begin{algorithmic}[1]
\STATE Given parameters $\beta>0$ and $\gamma \in (0, 1)$. Set initial points $(\su_0^0, \sv_0^0)$, $(\p_0^0, \q_0^0)$ and $\mu^0>0$.
\FOR {$k = 0,1,2,\ldots$}
\FOR {$i = 0,1,2,\ldots$}
\IF{the termination criterion is satisfied}
\STATE \textbf{break}.
\ENDIF
\STATE Update $(\tilde{\su}_{i+1}^k, \tilde{\sv}_{i+1}^k)$ by \eqref{admm-typical-1};
\STATE Update $(\su_{i+1}^k, \sv_{i+1}^k)$ by \eqref{admm-typical-2};
\STATE Update $(\p_{i+1}^k, \q_{i+1}^k)$ by \eqref{admm-typical-3}.
\ENDFOR
\STATE Set $(\su_0^{k+1}, \sv_0^{k+1}) = (\su_{i+1}^k, \sv_{i+1}^k)$ and $(\p_0^{k+1}, \q_0^{k+1}) = (\p_{i+1}^k, \q_{i+1}^k)$;
\STATE Set $\mu^{k+1} = \gamma\mu^k$.
\ENDFOR
\end{algorithmic}
\end{algorithm}

\subsection{Implementing ABIP}
In this subsection we discuss the detailed implementation of \textsf{ABIP}. In particular, we show that the dual variables $\p$ and $\q$ in \eqref{admm-typical-1},  \eqref{admm-typical-2} and \eqref{admm-typical-3} can be eliminated using a proper initialization. The framework of our analysis is similar to the one in \cite{Donoghue-2016-conic}, but the techniques we use are quite different because the \textit{Moreau decomposition} cannot be directly applied to $(\su, \sv)$ when the log-barrier penalty function is used. The main technical result is summarized in the following theorem.

\begin{theorem}\label{Theorem:Simplified}
For the $k$-th outer iteration of Algorithm~\ref{Algorithm:Basic}, we initialize $\p_0^k=\sv_0^k$ and $\q_0^k=\su_0^k$ with
\begin{equation*}
\x_0^k\circ\s_0^k=\frac{\mu^k}{\beta}\e, \quad \tau_0^k\kappa_0^k=\frac{\mu^k}{\beta}, \quad \sr_0^k = 0, \quad \xi_0^k = -n-1.
\end{equation*}
It then holds, for all iterations $i\geq 0$, that
\begin{equation}\label{p=v,q=u}
\p_i^k=\sv_i^k, \quad \q_i^k=\su_i^k, \quad \x_i^k\circ\s_i^k=\frac{\mu^k}{\beta}\e, \quad \tau_i^k\kappa_i^k=\frac{\mu^k}{\beta}, \quad \sr_i^k = 0, \quad \xi_i^k = -n-1.
\end{equation}
\end{theorem}

\begin{proof}
We shall prove the result by induction. Indeed, the proof is based on the following steps:
(i) Iteration $j=0$: the result holds true since we can initialize the variables accordingly.
(ii) The result holds true for iteration $j=i+1$ given that it holds true for iteration $j=i$. 
We prove the desired result in two steps:

{\bf Step 1:} We claim that
\begin{equation}\label{Skew-Symmetric-Result}
\su_i^k + \sv_i^k = \tilde{\su}_{i+1}^k + \tilde{\sv}_{i+1}^k.
\end{equation}
Indeed, we rewrite \eqref{admm-typical-1} as
\begin{equation}\label{Skew-Symmetric}
\left(\tilde{\su}_{i+1}^{k}, \tilde{\sv}_{i+1}^{k}\right) = \prod_\QCal\left(\su_i^k+\sv_i^k, \su_i^k+\sv_i^k\right),
\end{equation}
where $\QCal=\left\{(\su,\sv): Q\su=\sv\right\}$. Moreover, it follows from $Q$ being skew-symmetric that the orthogonal complement of $\QCal$ is $\QCal^\perp=\left\{(\sv,\su): Q\su=\sv\right\}$.
Therefore, we conclude that,
\begin{equation*}
(\sv,\su)=\prod_{\QCal^\perp}(\z,\z), \quad \text{if} \ (\su,\sv)=\prod_\QCal(\z,\z),
\end{equation*}
because the two projections are identical for reversed output arguments. This implies that
\begin{equation}\label{Skew-Symmetric-Reversed}
\left(\tilde{\sv}_{i+1}^{k}, \tilde{\su}_{i+1}^{k}\right) = \prod_{\QCal^\perp}\left(\su_i^k+\sv_i^k, \su_i^k+\sv_i^k\right).
\end{equation}
Therefore, combining \eqref{Skew-Symmetric} and \eqref{Skew-Symmetric-Reversed} yields the desired result.

{\bf Step 2:} We proceed to proving that
\begin{equation*}
\p_{i+1}^k=\sv_{i+1}^k, \quad \q_{i+1}^k=\su_{i+1}^k, \quad \x_{i+1}^k\circ\s_{i+1}^k=\frac{\mu^k}{\beta}\e, \quad \tau_{i+1}^k\kappa_{i+1}^k=\frac{\mu^k}{\beta}, \quad \xi_{i+1}^k = -n-1,
\end{equation*}
given
\begin{equation}\label{Assumption:Previous-Equality}
\p_i^k=\sv_i^k, \quad \q_i^k=\su_i^k
\end{equation}
and
\begin{equation}\label{Assumption:Previous-Complementary}
\x_i^k\circ\s_i^k=\frac{\mu^k}{\beta}\e, \quad \tau_i^k\kappa_i^k=\frac{\mu^k}{\beta}, \quad \xi_i^k = -n-1.
\end{equation}

Indeed, we partition $\p$ and $\q$ as
\begin{equation*}
\p = \left[\begin{array}{c} \p_\y \\ \p_\x \\ \p_\tau \\ \p_\theta \end{array}\right], \quad \q = \left[\begin{array}{c} \q_\sr \\ \q_\s \\ \q_\kappa \\ \q_\xi \end{array}\right],
\end{equation*}
and the optimality conditions of \eqref{admm-typical-2} are given by
\begin{align}
0 & = \y_{i+1}^k - \tilde{\y}_{i+1}^k + (\p_\y)_i^k, \label{opt:y}   \\
0 & = -\frac{\mu^k}{\beta} \cdot \frac{1}{\x_{i+1}^k} + \x_{i+1}^k - \tilde{\x}_{i+1}^k + (\p_\x)_i^k, \label{opt:x}  \\
0 & = -\frac{\mu^k}{\beta} \cdot \frac{1}{\tau_{i+1}^k} + \tau_{i+1}^k - \tilde{\tau}_{i+1}^k + (\p_\tau)_i^k, \label{opt:tau}  \\
0 & = (n+1) + \theta_{i+1}^k-\tilde{\theta}_{i+1}^k+(\p_\theta)_i^k, \label{opt:theta} \\
0 & = \sr_{i+1}^k, \label{opt:r}  \\
0 & = -\frac{\mu^k}{\beta} \cdot \frac{1}{\s_{i+1}^k} + \s_{i+1}^k - \tilde{\s}_{i+1}^k + (\q_\s)_i^k, \label{opt:s}  \\
0 & = -\frac{\mu^k}{\beta} \cdot \frac{1}{\kappa_{i+1}^k} + \kappa_{i+1}^k - \tilde{\kappa}_{i+1}^k + (\q_\kappa)_i^k, \label{opt:kappa} \\
0 & = \xi_{i+1}^k + n+1. \label{opt:xi}
\end{align}
First, we show that
\begin{equation*}
\sr_{i+1}^k = (\p_\y)_{i+1}^k = 0, \qquad \y_{i+1}^k = (\q_\sr)_{i+1}^k.
\end{equation*}
Indeed, from \eqref{admm-typical-3}, \eqref{opt:y} and \eqref{opt:r} we have that
\begin{equation*}
(\p_\y)_{i+1}^k = (\p_\y)_i^k - \tilde{\y}_{i+1}^k + \y_{i+1}^k = 0 = \sr_{i+1}^k.
\end{equation*}
Furthermore, we have
\begin{align*}
(\q_\sr)_{i+1}^k & \overset{\eqref{admm-typical-3}}{=} (\q_\sr)_i^k - \tilde{\sr}_{i+1}^k + \sr_{i+1}^k \overset{\eqref{opt:r}}{=} (\q_\sr)_i^k - \tilde{\sr}_{i+1}^k \overset{\eqref{Skew-Symmetric-Result}}{=} (\q_\sr)_i^k - \left(\sr_i^k + \y_i^k - \tilde{\y}_{i+1}^k\right) \\
& \overset{\eqref{opt:y}}{=} (\q_\sr)_i^k - \left(\sr_i^k + \y_i^k - \y_{i+1}^k - (\p_\y)_i^k\right) \overset{\eqref{Assumption:Previous-Equality}}{=} \y_i^k - \left(\sr_i^k + \y_i^k - \y_{i+1}^k - \sr_i^k\right) = \y_{i+1}^k.
\end{align*}
Second, we shall prove
\begin{equation*}
\x_{i+1}^k = (\q_\s)_{i+1}^k, \qquad \s_{i+1}^k = (\p_\x)_{i+1}^k, \qquad \x_{i+1}^k \circ\s_{i+1}^k = \frac{\mu^k}{\beta}\cdot \e.
\end{equation*}
Indeed, from \eqref{admm-typical-3} and \eqref{Assumption:Previous-Equality} we have
\begin{equation*}
(\p_\x)_{i+1}^k = (\p_\x)_i^k - \tilde{\x}_{i+1}^k + \x_{i+1}^k = \s_i^k - \tilde{\x}_{i+1}^k + \x_{i+1}^k,
\end{equation*}
and
\begin{equation*}
(\q_\s)_{i+1}^k = (\q_\s)_i^k - \tilde{\s}_{i+1}^k + \s_{i+1}^k = \x_i^k - \tilde{\s}_{i+1}^k + \s_{i+1}^k.
\end{equation*}
Combining the above two equations with \eqref{Skew-Symmetric-Result} yields
\begin{equation}\label{Primal-Dual-Equality}
(\p_\x)_{i+1}^k + (\q_\s)_{i+1}^k = \x_{i+1}^k + \s_{i+1}^k.
\end{equation}
Besides, from \eqref{opt:x} and \eqref{admm-typical-3} we have
\begin{equation}\label{Result:x-px}
\frac{\mu^k}{\beta} \cdot \e = \x_{i+1}^k \circ \left(\x_{i+1}^k - \tilde{\x}_{i+1}^k + (\p_\x)_i^k\right) = \x_{i+1}^k \circ (\p_\x)_{i+1}^k,
\end{equation}
and from \eqref{opt:s} and \eqref{admm-typical-3} we have
\begin{equation}\label{Result:s-qs}
\frac{\mu^k}{\beta} \cdot \e = \s_{i+1}^k \circ \left(\s_{i+1}^k - \tilde{\s}_{i+1}^k + (\q_\s)_i^k\right) = \s_{i+1}^k \circ (\q_\s)_{i+1}^k.
\end{equation}
Therefore, we obtain
\begin{eqnarray*}
0 & \overset{\eqref{Result:x-px},\eqref{Result:s-qs}}{=} & \x_{i+1}^k \circ (\p_\x)_{i+1}^k - \s_{i+1}^k \circ (\q_\s)_{i+1}^k \\
& \overset{\eqref{Primal-Dual-Equality}}{=} & \x_{i+1}^k \circ \left( \x_{i+1}^k + \s_{i+1}^k - (\q_\s)_{i+1}^k\right) - \s_{i+1}^k \circ (\q_\s)_{i+1}^k \\
& = & \left(\x_{i+1}^k - (\q_\s)_{i+1}^k\right) \circ \left( \x_{i+1}^k + \s_{i+1}^k\right).
\end{eqnarray*}
Since $\x_{i+1}^k + \s_{i+1}^k>0$, we conclude that
$\x_{i+1}^k=(\q_\s)_{i+1}^k$
which, combining with~\eqref{Result:s-qs}, leads to
\begin{equation*}
\frac{\mu^k}{\beta} \cdot \e = \x_{i+1}^k\circ\s_{i+1}^k.
\end{equation*}
It also directly follows from~\eqref{Primal-Dual-Equality} that
$\s_{i+1}^k=(\p_\x)_{i+1}^k$.
We use the same arguments to conclude
\begin{equation*}
\tau_{i+1}^k = (\q_\kappa)_{i+1}^k, \qquad \kappa_{i+1}^k = (\p_\tau)_{i+1}^k, \qquad \tau_{i+1}^k \kappa_{i+1}^k = \frac{\mu^k}{\beta}.
\end{equation*}
Finally, we show that
\begin{equation*}
\xi_{i+1}^k = (\p_\theta)_{i+1}^k = -n-1, \qquad \theta_{i+1}^k = (\q_\xi)_{i+1}^k = \frac{\mu^k}{\beta}.
\end{equation*}
Indeed, from \eqref{admm-typical-3}, \eqref{opt:theta} and \eqref{opt:xi} we have
\begin{equation*}
(\p_\theta)_{i+1}^k = (\p_\theta)_i^k - \tilde{\theta}_{i+1}^k + \theta_{i+1}^k = -n-1 = \xi_{i+1}^k.
\end{equation*}
Furthermore, combining~\eqref{Assumption:Previous-Equality} and \eqref{Assumption:Previous-Complementary} we have
\begin{equation*}
(\p_\theta)_i^k = \xi_i^k = -n-1,
\end{equation*}
which implies that
\begin{equation}\label{Theta-Constant-Equality}
\theta_{i+1}^k - \tilde{\theta}_{i+1}^k = (\p_\theta)_{i+1}^k - (\p_\theta)_i^k = 0.
\end{equation}
Therefore, we conclude that
\begin{align*}
(\q_\xi)_{i+1}^k & \overset{\eqref{admm-typical-3}}{=} (\q_\xi)_i^k - \tilde{\xi}_{i+1}^k + \xi_{i+1}^k \overset{\eqref{Skew-Symmetric-Result}}{=}  (\q_\xi)_i^k - \theta_i^k - \xi_i^k + \tilde{\theta}_{i+1}^k + \xi_{i+1}^k \\
& \overset{\eqref{Assumption:Previous-Complementary},\eqref{opt:xi}}{=} (\q_\xi)_i^k - \theta_i^k + \tilde{\theta}_{i+1}^k  \overset{\eqref{Theta-Constant-Equality}}{=} (\q_\xi)_i^k - \theta_i^k + \theta_{i+1}^k \overset{\eqref{Assumption:Previous-Equality}}{=} \theta_{i+1}^k.
\end{align*}
This completes the proof.
\end{proof}
Observe that Theorem~\ref{Theorem:Simplified} simplifies Algorithm~\ref{Algorithm:Basic} by eliminating the dual variables $\p$ and $\q$. They are replaced by $\sv$ and $\su$, respectively. Moreover, note that $\su$ and $\sv$ are separable in \eqref{admm-typical-2}. As a result, we can update $\su$ by
\begin{equation}\label{Simple-Update:Barrier}
\su_{i+1}^{k} = \argmin_{\su} \left[\bar{B}(\su, \mu^k) + \frac{\beta}{2}\left\| \su - \tilde{\su}_{i+1}^k + \sv_i^k\right\|^2\right],
\end{equation}
where
\begin{equation}\label{B-u-mu}
\bar{B}(\su, \mu) := \beta(n+1)\theta-\mu\log(\x) - \mu\log(\tau),
\end{equation}
and update $\sv$ by
\begin{equation}\label{Simple-Update:Dual}
\sv_{i+1}^k = \sv_i^k - \tilde{\su}_{i+1}^k + \su_{i+1}^k,
\end{equation}
which follows from the update for $\p_{i+1}^k$. Note that $\tilde{\sv}_i^k$ can now be eliminated from the algorithm. Problem \eqref{Simple-Update:Barrier} admits closed-form solutions given by
\begin{align}
\y_{i+1}^k & = \argmin_{\y} \left[\frac{1}{2}\left\| \y - \tilde{\y}_{i+1}^k + \sr_i^k\right\|^2\right] = \tilde{\y}_{i+1}^k, \label{Simple-Update:Barrier-Y} \\
\x_{i+1}^k & = \argmin_{\x} \left[-\frac{\mu^k}{\beta}\log(\x) + \frac{1}{2}\left\| \x - \tilde{\x}_{i+1}^k + \s_i^k\right\|^2\right] \nonumber \\
& = \frac{1}{2}\left[\left(\tilde{\x}_{i+1}^k - \s_i^k\right) + \sqrt{\left(\tilde{\x}_{i+1}^k - \s_i^k\right)\circ\left(\tilde{\x}_{i+1}^k - \s_i^k\right)+\frac{4\mu^k}{\beta}}\right], \label{Simple-Update:Barrier-X} \\
\tau_{i+1}^k & = \argmin_{\tau} \left[-\frac{\mu^k}{\beta}\log(\tau) + \frac{1}{2}\left\|\tau - \tilde{\tau}_{i+1}^k + \kappa_i^k\right\|^2\right] \nonumber \\
& = \frac{1}{2}\left[\left(\tilde{\tau}_{i+1}^k - \kappa_i^k\right) + \sqrt{\left(\tilde{\tau}_{i+1}^k - \kappa_i^k\right)\circ\left(\tilde{\tau}_{i+1}^k - \kappa_i^k\right)+\frac{4\mu^k}{\beta}}\right], \label{Simple-Update:Barrier-tau} \\
\theta_{i+1}^k & = \tilde{\theta}_{i+1}^k, \label{Simple-Update:Barrier-theta}
\end{align}
where the last step is from~\eqref{Theta-Constant-Equality}. By eliminating $\p_i^k$ and $\q_i^k$, \eqref{admm-typical-1} reduces to
\begin{equation*}
\left(\tilde{\su}_{i+1}^k, \tilde{\sv}_{i+1}^k\right) = \prod_{Q\su=\sv} \left(\su_i^k+\sv_i^k, \su_i^k+\sv_i^k\right).
\end{equation*}
It is easy to show (by the KKT condition) that the solution is given by
\begin{equation}\label{Simple-Update:Affine}
\tilde{\su}_{i+1}^k = \left(I+Q^\top Q\right)^{-1}\left(I-Q\right)\left(\su_i^k + \sv_i^k\right)= \left(I+Q\right)^{-1}\left(\su_i^k + \sv_i^k\right),
\end{equation}
because matrix $Q$ is skew-symmetric. Moreover, we only need to invert (or factorize) $I+Q$ once at the beginning of the algorithm. In this sense, \textsf{ABIP} is matrix inversion free. 

Therefore, we have shown that \eqref{admm-typical-1}, \eqref{admm-typical-2} and \eqref{admm-typical-3} can be simply implemented by means of \eqref{Simple-Update:Affine}, \eqref{Simple-Update:Barrier} and \eqref{Simple-Update:Dual} respectively, and the solutions of \eqref{Simple-Update:Barrier} are given by \eqref{Simple-Update:Barrier-Y}, \eqref{Simple-Update:Barrier-X}, \eqref{Simple-Update:Barrier-tau} and
\eqref{Simple-Update:Barrier-theta}.

{We use the following criterion to terminate the inner loop of Algorithm \ref{Algorithm:Basic}:
\begin{equation}\label{inner-stop}
\left\|Q\su_i^k - \sv_i^k\right\|^2 \leq \mu^k.
\end{equation}
Finally, we present this specific implementation \textsf{ABIP} as Algorithm~\ref{Algorithm:Simplified}.}
\begin{algorithm}[ht]
\caption{The Detailed Implementation of \textsf{ABIP}} \label{Algorithm:Simplified}
\begin{algorithmic}[1]
\STATE Set $\mu^0=\beta>0$ and $\gamma\in\left(0,1\right)$.
\STATE Set $\sr_0^0=\y_0^0=0$, $\left(\x_0^0, \tau_0^0, \s_0^0, \kappa_0^0\right)=\left(\e,1,\e,1\right)>0$, $\theta_0^0=1$, and $\xi_0^0=-n-1$ with $\x_0^0\circ\s_0^0=\frac{\mu^0}{\beta}\e$, and $\tau_0^0\kappa_0^0=\frac{\mu^0}{\beta}$.
\FOR {$k = 0,1,2,\ldots$}
\FOR {$i = 0,1,2,\ldots$}
\IF{the \textbf{inner} termination criterion \eqref{inner-stop} is satisfied}
\STATE \textbf{break}.
\ENDIF
\STATE Update $\tilde{\su}_{i+1}^k$ by using \eqref{Simple-Update:Affine};
\STATE Update $\su_{i+1}^k$ by using \eqref{Simple-Update:Barrier-Y}, \eqref{Simple-Update:Barrier-X}, \eqref{Simple-Update:Barrier-tau} and \eqref{Simple-Update:Barrier-theta};
\STATE Update $\sv_{i+1}^k$ by using \eqref{Simple-Update:Dual}.
\IF{the \textbf{final} termination criterion is satisfied}
\STATE \textbf{return}.
\ENDIF
\ENDFOR
\STATE Set $\mu^{k+1}=\gamma\cdot\mu^k$;
\STATE Set $\sr_0^{k+1} = 0$, $\xi_0^{k+1}=-n-1$ and
\begin{equation}\label{update-next-iterate}
(\y_0^{k+1}, \x_0^{k+1}, \s_0^{k+1}, \tau_0^{k+1}, \kappa_0^{k+1}, \theta_0^{k+1}) = \sqrt{\gamma}\cdot(\y_{i+1}^k, \x_{i+1}^k, \s_{i+1}^k, \tau_{i+1}^k, \kappa_{i+1}^k, \theta_{i+1}^{k}).
\end{equation}
\ENDFOR
\end{algorithmic}
\end{algorithm}

\begin{remark}
{We denote $(\su_k^*,\sv_k^*)$ as the optimal solution to \eqref{prob:HSD-log} when $\mu=\mu^k$, which also satisfies the following optimality conditions of \eqref{prob:HSD-log}:
\begin{eqnarray}\label{HSD-log-kkt}
\left\{
\begin{array}{rcl}
Q\su - \sv & = & 0, \\
\x \circ \s & = & \frac{\mu^k}{\beta} \e, \\
\tau\kappa & = & \frac{\mu^k}{\beta}, \\
\theta & = & \frac{\mu^k}{\beta}, \\
\sr & = & 0, \\
\xi & = & -n-1, \\
\left(\x, \s, \tau, \kappa\right) & > & 0 .
\end{array}
\right.
\end{eqnarray}
Moreover, $(\su_k^*,\sv_k^*)$ is uniquely defined. In fact, \eqref{HSD-log-kkt} defines a central path (cf.~\cite{Sonnevend-1986-analytical, Bayer-1989-nonlinear, Megiddo-1989-progress, Renegar-1988-polynomial}) of the homogeneous self-dual embedded model (\cite{Ye-1994-hsd}).}

From Theorem \ref{Theorem:Simplified} we have
\[
\left\{
\begin{array}{rcl}
Q\tilde{\su}_i^k - \tilde{\sv}_i^k &=& 0, \\
\x_i^k \circ \s_i^k & = & \frac{\mu^k}{\beta} \e, \\
\tau_i^k \kappa_i^k & = & \frac{\mu^k}{\beta}, \\
\sr_i^k & = & 0, \\
\xi_i^k & = & -n-1, \\
\left(\x_i^k, \s_i^k, \tau_i^k, \kappa_i^k\right) & > & 0,
\end{array}
\right.
\]
and
\begin{equation*}
\theta_i^k = \tilde{\theta}_i^k = \frac{\left(\tilde{\x}_i^k\right)^\top\tilde{\s}_i^k + \tilde{\tau}_i^k \tilde{\kappa}_i^k + \left(\tilde{\y}_i^k\right)^\top\tilde{\sr}_i^k}{-\tilde{\xi}_i^k}.
\end{equation*}
Together with the feasibility condition that $\left\|\left(\tilde{\su}_i^k, \tilde{\sv}_i^k\right) - \left(\su_i^k, \sv_i^k\right)\right\|\rightarrow 0$ when $i\rightarrow +\infty$ (See Lemma \ref{Lemma:Descent-Inner-Loop}), we conclude that the optimal solution to problem~\eqref{prob:HSD-admm} is on the central path. This implies that \textsf{ABIP} is indeed a central path following algorithm, in view of the classical primal-dual central path following scheme.
\end{remark}

\begin{corollary}\label{cor:pk=vk}
Following a similar argument as in Theorem \ref{Theorem:Simplified}, it is easy to prove:
    \begin{equation}\label{pstar=vstar} \left(\p_k^*, \q_k^*\right)=\left(\sv_k^*, \su_k^*\right), \end{equation}
    where $\left(\p_k^*, \q_k^*\right)$ denotes the optimal dual solution of \eqref{prob:HSD-admm}.
\end{corollary}

\subsection{Iteration Complexity Analysis}
In this subsection we analyze the iteration complexity of \textsf{ABIP}. 
The following identity will be frequently used in our analysis:
\begin{equation}\label{identity}
\left(a_1-a_2\right)^\top\left(a_3-a_4\right) = \frac{1}{2}\left(\left\|a_4+a_2\right\|^2-\left\|a_4+a_1\right\|^2+\left\|a_3+a_1\right\|^2-\left\|a_3+a_2\right\|^2\right), \forall a_1,a_2,a_3,a_4.
\end{equation}

To prove the main result, we need several technical lemmas.
\begin{lemma}\label{Lemma:Descent-Inner-Loop}
Given $k\geq 1$, the sequence $\{\|\su_i^k - \su_k^*\|^2 + \|\sv_i^k - \sv_k^*\|^2\}_{i\geq 0}$ is monotonically decreasing and converges to $0$.
\end{lemma}
\begin{proof}
We observe from the optimality condition of problem~\eqref{prob:HSD-admm} that
\[
\beta(\p_k^*, \q_k^*) \in \partial \indfunc (Q\su=\sv) [\su_k^*, \sv_k^*], \quad
-\beta(\p_k^*, \q_k^*) \in \partial B(\su_k^*, \sv_k^*, \mu^k).
\]
By using the convexity of $\indfunc (Q\su=\sv)$ and \eqref{pstar=vstar} we have
\[
0 \leq \beta(\su_k^*-\tilde{\su}_{i+1}^k, \sv_k^*-\tilde{\sv}_{i+1}^k)^\top(\p_k^*, \q_k^*) = \beta(\su_k^*-\tilde{\su}_{i+1}^k, \sv_k^*-\tilde{\sv}_{i+1}^k)^\top(\sv_k^*, \su_k^*),
\]
and using the convexity of $B(\su,\sv,\mu)$ with respect to $(\su,\sv)$ and \eqref{pstar=vstar} we have
\[
\beta(n+1)(\theta_k^* - \theta_{i+1}^k) \leq -\beta(\su_k^*-\su_{i+1}^k, \sv_k^*-\sv_{i+1}^k)^\top(\p_k^*, \q_k^*) = -\beta(\su_k^*-\su_{i+1}^k, \sv_k^*-\sv_{i+1}^k)^\top(\sv_k^*, \su_k^*),
\]
{where we have used the fact $\beta (n+1) \left(\theta_{i+1}^k - \theta_k^*\right)  = B(\su_{i+1}^k, \sv_{i+1}^k, \mu^k) - B(\su_k^*, \sv_k^*, \mu^k)$ that follows from the complementarity conditions $\x_i^k\circ\s_i^k = \frac{\mu^k}{\beta}\e$ and $\tau_i^k\kappa_i^k = \frac{\mu^k}{\beta}$.}
Summing up the above two inequalities leads to
\begin{equation}\label{lemma-descent:inequality-1}
\beta (n+1) (\theta_k^* - \theta_{i+1}^k) \leq \beta(\su_{i+1}^k - \tilde{\su}_{i+1}^k, \sv_{i+1}^k - \tilde{\sv}_{i+1}^k)^\top(\sv_k^*, \su_k^*).
\end{equation}
Combining \eqref{admm-typical-3} and the optimality conditions of \eqref{admm-typical-1} and \eqref{admm-typical-2} yields
\begin{equation}\label{opt-1}
(\p_{i+1}^k, \q_{i+1}^k) - (\su_{i+1}^k, \sv_{i+1}^k) + (\su_i^k, \sv_i^k) \in \partial \indfunc (Q\su=\sv) [\tilde{\su}_{i+1}^k, \tilde{\sv}_{i+1}^k],
\end{equation}
\begin{equation}\label{opt-2}
-\beta(\p_{i+1}^k, \q_{i+1}^k) \in \partial B(\su_{i+1}^k, \sv_{i+1}^k, \mu^k).
\end{equation}
From the convexity of $\indfunc (Q\su=\sv)$, we have
\begin{eqnarray}\label{37.5}
0 & \overset{\eqref{opt-1}}{\leq} & (\tilde{\su}_{i+1}^k - \su_k^*, \tilde{\sv}_{i+1}^k - \sv_k^*)^\top(\p_{i+1}^k - \su_{i+1}^k + \su_i^k, \q_{i+1}^k - \sv_{i+1}^k + \sv_i^k) \\
& \overset{\eqref{identity}}{=} & (\tilde{\su}_{i+1}^k - \su_k^*, \tilde{\sv}_{i+1}^k - \sv_k^*)^\top(\p_{i+1}^k, \q_{i+1}^k) + \frac{1}{2}\left(\|\su_i^k - \su_k^*\|^2 + \|\sv_i^k - \sv_k^*\|^2 - \|\su_i^k - \tilde{\su}_{i+1}^k\|^2\right. \nonumber\\
& & - \left. \|\sv_i^k - \tilde{\sv}_{i+1}^k\|^2 \right) + \frac{1}{2}\left(\|\su_{i+1}^k - \tilde{\su}_{i+1}^k\|^2 + \|\sv_{i+1}^k - \tilde{\sv}_{i+1}^k\|^2 - \|\su_{i+1}^k - \su_k^*\|^2 - \|\sv_{i+1}^k - \sv_k^*\|^2 \right).\nonumber
\end{eqnarray}
From the convexity of $B(\su,\sv,\mu^k)$ with respect to $(\su,\sv)$ we have
\begin{eqnarray}\nonumber
\beta (n+1) \left(\theta_{i+1}^k - \theta_k^*\right) & = & B(\su_{i+1}^k, \sv_{i+1}^k, \mu^k) - B(\su_k^*, \sv_k^*, \mu^k) \\ & \overset{\eqref{opt-2}}{\leq} & -\beta(\su_{i+1}^k - \su_k^*, \sv_{i+1}^k - \sv_k^*)^\top(\p_{i+1}^k, \q_{i+1}^k),\label{37.6}
\end{eqnarray}
{where the equality again follows from the complementarity conditions $\x_i^k\circ\s_i^k = \frac{\mu^k}{\beta}\e$ and $\tau_i^k\kappa_i^k = \frac{\mu^k}{\beta}$.}
Adding \eqref{37.5} and \eqref{37.6} and using \eqref{p=v,q=u} yields
\begin{eqnarray}\label{lemma-descent:inequality-2}
& & \beta (n+1)(\theta_{i+1}^k - \theta_k^*) \\
& \leq & \beta(\tilde{\su}_{i+1}^k - \su_{i+1}^k, \tilde{\sv}_{i+1}^k - \sv_{i+1}^k)^\top(\sv_{i+1}^k, \su_{i+1}^k) \nonumber \\
& & + \frac{\beta}{2}\left(\|\su_i^k - \su_k^*\|^2 + \|\sv_i^k - \sv_k^*\|^2 - \|\su_i^k - \tilde{\su}_{i+1}^k\|^2 - \|\sv_i^k - \tilde{\sv}_{i+1}^k\|^2\right) \nonumber \\
& & + \frac{\beta}{2}\left(\|\su_{i+1}^k - \tilde{\su}_{i+1}^k\|^2 + \|\sv_{i+1}^k - \tilde{\sv}_{i+1}^k\|^2 - \|\su_{i+1}^k - \su_k^*\|^2 - \|\sv_{i+1}^k - \sv_k^*\|^2 \right).  \nonumber
\end{eqnarray}
Further adding \eqref{lemma-descent:inequality-1} and \eqref{lemma-descent:inequality-2} we have
\begin{eqnarray}
0 & \leq & (\tilde{\su}_{i+1}^k - \su_{i+1}^k, \tilde{\sv}_{i+1}^k - \sv_{i+1}^k)^\top(\sv_{i+1}^k - \sv_k^*, \su_{i+1}^k-\su_k^*) \nonumber\\
& & + \frac{1}{2}\left(\|\su_i^k - \su_k^*\|^2 + \|\sv_i^k - \sv_k^*\|^2 - \|\su_i^k - \tilde{\su}_{i+1}^k\|^2 - \|\sv_i^k - \tilde{\sv}_{i+1}^k\|^2 \right) \nn\\
& & + \frac{1}{2}\left(\|\su_{i+1}^k - \tilde{\su}_{i+1}^k\|^2 + \|\sv_{i+1}^k - \tilde{\sv}_{i+1}^k\|^2 - \|\su_{i+1}^k - \su_k^*\|^2 - \|\sv_{i+1}^k - \sv_k^*\|^2 \right) \nn\\
& \overset{\eqref{admm-typical-3},\eqref{p=v,q=u}}{=} & (\sv_i^k - \sv_{i+1}^k, \su_i^k - \su_{i+1}^k)^\top(\sv_{i+1}^k - \sv_k^*, \su_{i+1}^k-\su_k^*) \nn \\
& & + \frac{1}{2}\left(\|\su_i^k - \su_k^*\|^2 + \|\sv_i^k - \sv_k^*\|^2 - \|\su_i^k - \tilde{\su}_{i+1}^k\|^2 - \|\sv_i^k - \tilde{\sv}_{i+1}^k\|^2 \right) \nn \\
& & + \frac{1}{2}\left(\|\su_{i+1}^k - \tilde{\su}_{i+1}^k\|^2 + \|\sv_{i+1}^k - \tilde{\sv}_{i+1}^k\|^2 - \|\su_{i+1}^k - \su_k^*\|^2 - \|\sv_{i+1}^k - \sv_k^*\|^2 \right) \nn \\
& \overset{\eqref{identity}}{=} & \frac{1}{2}\left(\| \su_i^k - \su_k^*\|^2 + \| \sv_i^k - \sv_k^*\|^2 - \| \su_{i+1}^k - \su_k^*\|^2 - \| \sv_{i+1}^k - \sv_k^*\|^2 \right) - \frac{1}{2}\left(\| \su_i^k - \su_{i+1}^k \|^2 \right. \nn \\
&& \left.+\| \sv_i^k - \sv_{i+1}^k\|^2\right) + \frac{1}{2}\left(\|\su_i^k - \su_k^*\|^2 + \|\sv_i^k - \sv_k^*\|^2 - \|\su_i^k - \tilde{\su}_{i+1}^k\|^2 - \|\sv_i^k - \tilde{\sv}_{i+1}^k\|^2 \right) \nn \\
& & + \frac{1}{2}\left(\|\su_{i+1}^k - \tilde{\su}_{i+1}^k\|^2 + \|\sv_{i+1}^k - \tilde{\sv}_{i+1}^k\|^2 - \|\su_{i+1}^k - \su_k^*\|^2 - \|\sv_{i+1}^k - \sv_k^*\|^2 \right).\label{long}
\end{eqnarray}
{Recall that \eqref{admm-typical-3} and \eqref{p=v,q=u} imply}
\begin{equation}\label{norm-identity}
\|\su_{i+1}^k - \tilde{\su}_{i+1}^k\|^2 + \|\sv_{i+1}^k - \tilde{\sv}_{i+1}^k\|^2 = \|\p_i^k - \p_{i+1}^k\|^2 + \|\q_i^k - \q_{i+1}^k\|^2 = \|\su_i^k - \su_{i+1}^k\|^2 + \|\sv_i^k - \sv_{i+1}^k\|^2.
\end{equation}
Now, we use \eqref{norm-identity} and
\eqref{long} to obtain
\begin{equation}\label{decrease}
\frac{1}{2}\left(\|\su_i^k - \tilde{\su}_{i+1}^k\|^2 + \|\sv_i^k - \tilde{\sv}_{i+1}^k \|^2\right) \leq \|\su_i^k - \su_k^*\|^2 + \|\sv_i^k - \sv_k^*\|^2 - \|\su_{i+1}^k - \su_k^*\|^2 - \|\sv_{i+1}^k - \sv_k^*\|^2.
\end{equation}
Therefore, we conclude that $\{\|\su_i^k - \su_k^*\|^2 + \|\sv_i^k - \sv_k^*\|^2\}_{i\geq 0}$ is a monotonically decreasing sequence. {Note} that $\{(\su_i^k, \sv_i^k)\}_{i\geq 0}$ is a bounded sequence, there must exist a subsequence $\{(\su_{i_j}^k, \sv_{i_j}^k)\}_{j\geq 0}$ that converges to a limit point $(\bar{\su}, \bar{\sv})$. Since
\begin{equation*}
-\beta(\sv_{i_j}^k, \su_{i_j}^k) = -\beta(\p_{i_j}^k, \q_{i_j}^k) \in \partial B(\su_{i_j}^k, \sv_{i_j}^k, \mu^k),
\end{equation*}
by letting $j\rightarrow+\infty$ we have
\[
\left\{
\begin{array}{rcl}
\bar\x \circ \bar\s & = & \left(\mu^k/\beta\right) \cdot \e, \\
\bar\tau \bar\kappa & = & \mu^k/\beta, \\
\bar\sr & = & 0, \\
\bar\xi & = & -n-1, \\
\left(\bar\x, \bar\s, \bar\tau, \bar\kappa\right) & > & 0.
\end{array}
\right.
\]
Furthermore, from \eqref{decrease} we have
\begin{equation*}
\| \su_{i_j}^k - \tilde{\su}_{i_j+1}^k \|^2 + \| \sv_{i_j}^k - \tilde{\sv}_{i_j+1}^k\|^2 \rightarrow 0,
\end{equation*}
which implies that $(\tilde{\su}_{i_j+1}^k, \tilde{\sv}_{i_j+1}^k)$ converges to $\left(\bar{\su}, \bar{\sv}\right)$. Therefore, we have $Q\bar{\su} - \bar{\sv} = 0$ and
\begin{equation*}
\bar{\theta} = \frac{\left(\bar{\x}\right)^\top\bar{\s} + \bar{\tau}\bar{\kappa} + \left(\bar{\y}\right)^\top\bar{\sr}}{-\bar{\xi}} = \frac{\mu^k}{\beta}.
\end{equation*}
Due to the uniqueness of the central path solution, we have $(\bar{\su}, \bar{\sv})=(\su_k^*, \sv_k^*)$, which implies that
\begin{equation*}
\| \su_{i_j}^k - \su_k^* \|^2 + \| \sv_{i_j}^k - \sv_k^*\|^2 \longrightarrow 0, \quad \text{as} \ j\rightarrow+\infty.
\end{equation*}
Therefore, we conclude that
\begin{equation*}
\| \su_i^k - \su_k^* \|^2 + \| \sv_i^k - \sv_k^*\|^2 \longrightarrow 0, \quad \text{as} \ i\rightarrow+\infty.
\end{equation*}
This completes the proof.
\end{proof}
\begin{lemma}\label{Lemma:Bound-Inner-Loop}
The sequence $\{\|\su_0^k - \su_k^*\|^2 + \|\sv_0^k - \sv_k^*\|^2\}_{k\geq 0}$ is uniformly bounded, i.e., there exists a constant $C>0$ that does not depend on $k$ or $\mu^k$ such that
\begin{equation}\label{lemma-bound-1}
\|\su_0^k - \su_k^*\|^2 + \|\sv_0^k - \sv_k^*\|^2 \leq C.
\end{equation}
Moreover, the iterates $\{(\su_i^k, \sv_i^k)\}$, for $k \geq 1$, $i = 0,1,\ldots,N_k$, are uniformly bounded, i.e., there exists a constant $D>0$ that does not depend on $k$ or $\mu^k$ such that
\begin{equation}\label{lemma-bound-2}
\|\su_i^k\|^2 + \|\sv_i^k\|^2 \leq D, \quad \forall i = 0, 1, \ldots, N_k,
\end{equation}
where $N_k$ denotes the number of inner iterations in the $k$-th outer loop.
\end{lemma}
\begin{proof}
We recall an important fact {(see, e.g., \cite{Ye-1994-hsd})} that the set of the central path points $\{(\su_k^*, \sv_k^*)\}$ with $\mu^k/\beta$, i.e., the solution of \eqref{HSD-log-kkt}, is uniformly bounded, where $0 < \mu^k \leq \mu^0$. That is, there exists a constant $C_1$ that does not depend on $k$ or $\mu^k$ such that
\begin{equation}\label{center-path-bound}
\|\su_k^*\|^2 + \|\sv_k^*\|^2 \leq C_1, \quad \forall k\geq 1.
\end{equation}
{This also implies that the following inequality holds for the limiting point $(\su^*, \sv^*)$ of the central path points $\{(\su_k^*, \sv_k^*)\}$, namely $\|\su^*\|^2 + \|\sv^*\|^2 \leq C_1$. Using Lemma~\ref{Lemma:Descent-Inner-Loop} and \eqref{inner-stop} with $\mu^k > 0$, we know that $N_k < +\infty$ is well-defined for any fixed $k \geq 0$. 
Now we claim that
\begin{equation}\label{lemma-3.5-claim}
\|\su_{N_k}^k - \su^*\|^2 + \|\sv_{N_k}^k - \sv^*\|^2 \rightarrow 0, \quad \text{as} \ k \rightarrow +\infty.
\end{equation}
Indeed, if the claim does not hold, then there exists $\delta > 0$, and a subsequence $\{ k_\ell \mid \ell=1,2,...\}$ with $k_\ell\uparrow \infty$ as $\ell\rightarrow \infty$ and $N' > 0$ such that
\begin{equation*}
\|\su_{N_{k_\ell}}^{k_\ell} - \su^*\|^2 + \|\sv_{N_{k_\ell}}^{k_\ell} - \sv^*\|^2 \geq \delta, \quad \text{for all } \ell > N'.
\end{equation*}
Using the property of central path, we have $\|\su_k^* - \su^*\|^2 + \|\sv_k^* - \sv^*\|^2 \rightarrow 0$ as $k \rightarrow +\infty$. Thus, there exists $N'' > 0$ such that
\begin{equation*}
\|\su_{N_{k_\ell}}^{k_\ell} - \su_{k_\ell}^*\|^2 + \|\sv_{N_{k_\ell}}^{k_\ell} - \sv_{k_\ell}^*\|^2 \geq \delta/2, \quad \text{for all } \ell > N''.
\end{equation*}
This contradicts with \eqref{inner-stop}.
Thus, \eqref{lemma-3.5-claim} holds and there exists a constant $D_1>0$ that does not depend on $k$ or $\mu^k$ such that
\begin{equation}\label{center-path-bound-aux}
\|\su_{N_k}^k\|^2 + \|\sv_{N_k}^k\|^2 \leq D_1, \quad \text{for all } k \geq 0.
\end{equation}

{Now we prove \eqref{lemma-bound-1}.} For $k = 0$, since we choose $\mu^0=\beta$, the initial point we choose in Algorithm~\ref{Algorithm:Simplified} satisfies \eqref{HSD-log-kkt} automatically, i.e., $(\su_0^*,\sv_0^*) = (\su_0^0,\sv_0^0)$, and \eqref{inner-stop}, i.e., $\|Q\su_0^0 - \sv_0^0\|^2 \leq \mu^0$, and $N_0 = 0$. Thus, we have
\begin{equation*}
\left\| \su_0^0 - \su_0^* \right\|^2 + \left\| \sv_0^0 - \sv_0^*\right\|^2 = 0.
\end{equation*}
For $k \geq 0$, by the definition of $\su$ and $\sv$, we have
\begin{eqnarray*}
\|\su_0^{k+1} - \su_{k+1}^*\|^2 + \|\sv_0^{k+1} - \sv_{k+1}^*\|^2 & = & \|\y_0^{k+1} - \y_{k+1}^*\|^2 + \|\x_0^{k+1} - \x_{k+1}^*\|^2 + (\tau_0^{k+1} - \tau_{k+1}^*)^2 \\
& & + (\theta_0^{k+1} - \theta_{k+1}^*)^2 + \|\s_0^{k+1} - \s_{k+1}^*\|^2 + (\kappa_0^{k+1} - \kappa_{k+1}^*)^2.
\end{eqnarray*}
Recall that (cf.\ \eqref{update-next-iterate}) the following equation holds true,
\begin{equation*}
(\y_0^{k+1}, \x_0^{k+1}, \s_0^{k+1}, \tau_0^{k+1}, \kappa_0^{k+1}, \theta_0^{k+1}) = \sqrt{\gamma}(\y_{N_k}^k, \x_{N_k}^k, \s_{N_k}^k, \tau_{N_k}^k, \kappa_{N_k}^k, \theta_{N_k}^k).
\end{equation*}
This implies that
\begin{eqnarray*}
& & \|\y_0^{k+1}\|^2 + \|\x_0^{k+1}\|^2 + (\tau_0^{k+1})^2 + (\theta_0^{k+1})^2 + \|\s_0^{k+1}\|^2 + (\kappa_0^{k+1})^2 \\
& = & \gamma\left(\|\y_{N_k}^k\|^2 + \|\x_{N_k}^k\|^2 + (\tau_{N_k}^k)^2 + (\theta_{N_k}^k)^2 + \|\s_{N_k}^k\|^2 + (\kappa_{N_k}^k)^2\right) \ \overset{\eqref{center-path-bound-aux}}{\leq} \ \gamma D_1.
\end{eqnarray*}
Putting these pieces together yields that
\begin{equation*}
\|\su_0^{k+1} - \su_{k+1}^*\|^2 + \|\sv_0^{k+1} - \sv_{k+1}^*\|^2 \leq 2\gamma D_1 + 2C_1.
\end{equation*}
Therefore, letting $C = 2\gamma D_1 + 2C_1$ proves \eqref{lemma-bound-1}.

Now we prove \eqref{lemma-bound-2}. Note that \eqref{lemma-bound-1} leads to \eqref{lemma-bound-2} immediately. To see this, note that combining \eqref{lemma-bound-1} and Lemma~\ref{Lemma:Descent-Inner-Loop} yields
\begin{equation*}
\|\su_i^k - \su_k^*\|^2 + \|\sv_i^k - \sv_k^*\|^2 \leq C, \quad \forall i=0,1,\ldots,N_k
\end{equation*}
which together with \eqref{center-path-bound} implies
\begin{equation*}
\|\su_i^k\|^2 + \|\sv_i^k\|^2 \leq 2C + 2C_1, \quad \forall i = 0, 1, \ldots, N_k
\end{equation*}
proving \eqref{lemma-bound-2} with $D=2C+2C_1$. }
\end{proof}

\begin{lemma}\label{Lemma:Complexity-Inner-Loop}
The number of iterations (denoted by $N_k$) needed in the inner loop of Algorithm \ref{Algorithm:Simplified} is
\begin{equation}\label{bound-N_k}
N_k \leq \log\left(\frac{2C\left(1+\left\|Q\right\|^2\right)}{{\mu^k}}\right) \left[\log\left(1+\min\left\{\frac{1}{C_3}, \frac{\mu^k}{4DC_3\beta}\right\}\right)\right]^{-1},
\end{equation}
where $C_3$ is defined as
\begin{eqnarray}
C_3 = \left[1+\frac{12\lambda_{\max}\left(A^\top A\right)}{\lambda_{\min}^2\left(AA^\top\right)}\cdot\max\left\{1, \|\bc\|^2, \|\bar{\bc}\|^2\right\}\right] \cdot \left[1 + \frac{6}{\|\bar{\bb}\|^2}\cdot\max\left\{\|\bb\|^2, \|A\|^2\right\} \right] > 1.\label{def-C3}
\end{eqnarray}
\end{lemma}
\begin{proof}
It follows from Lemma \ref{Lemma:Bound-Inner-Loop} that $B\left(\su_i^k, \sv_i^k, \mu^k\right)$ is strongly convex with respect to $\left(\x, \s, \tau, \kappa\right)$. More specifically, we have
\begin{eqnarray*}
\nabla_{\x}^2 B(\su_i^k, \sv_i^k, \mu^k) = \Diag\left(\frac{\mu^k}{\x_i^k.*\x_i^k}\right) & \succeq & \frac{\mu^k}{D} \cdot \indmat, \\
\nabla_{\s}^2 B(\su_i^k, \sv_i^k, \mu^k) = \Diag\left(\frac{\mu^k}{\s_i^k.*\s_i^k}\right) & \succeq & \frac{\mu^k}{D} \cdot \indmat, \\
\nabla_{\tau}^2 B(\su_i^k, \sv_i^k, \mu^k) = \frac{\mu^k}{\left(\tau_i^k\right)^2} & \succeq & \frac{\mu^k}{D}, \\
\nabla_{\kappa}^2 B(\su_i^k, \sv_i^k, \mu^k) = \frac{\mu^k}{\left(\kappa_i^k\right)^2} & \succeq & \frac{\mu^k}{D}.
\end{eqnarray*}
Therefore, \eqref{lemma-descent:inequality-2} is changed to
\begin{eqnarray}\label{lemma-complexity:inequality-2}
& & \beta (n+1) (\theta_{i+1}^k - \theta_k^*) + \frac{\mu^k}{2D}\left(\|\x_{i+1}^k - \x_k^*\|^2 + \|\s_{i+1}^k - \s_k^*\|^2 + (\tau_{i+1}^k - \tau_k^*)^2 + (\kappa_{i+1}^k - \kappa_k^*)^2 \right) \nonumber \\
& \leq & \beta(\tilde{\su}_{i+1}^k - \su_{i+1}^k, \tilde{\sv}_{i+1}^k - \sv_{i+1}^k)^\top(\sv_{i+1}^k, \su_{i+1}^k) + \frac{\beta}{2}\left(\|\su_i^k - \su_k^*\|^2 + \|\sv_i^k - \sv_k^*\|^2 - \|\su_i^k - \tilde{\su}_{i+1}^k\|^2 - \|\sv_i^k - \tilde{\sv}_{i+1}^k\|^2\right) \nonumber \\
& & + \frac{\beta}{2}\left(\|\su_{i+1}^k - \tilde{\su}_{i+1}^k\|^2 + \|\sv_{i+1}^k - \tilde{\sv}_{i+1}^k\|^2 - \|\su_{i+1}^k - \su_k^*\|^2 - \|\sv_{i+1}^k - \sv_k^*\|^2 \right).
\end{eqnarray}
By summing \eqref{lemma-descent:inequality-1} and \eqref{lemma-complexity:inequality-2}, using \eqref{norm-identity} and observing
\begin{eqnarray*}
& & (\su_{i+1}^k - \tilde{\su}_{i+1}^k, \sv_{i+1}^k - \tilde{\sv}_{i+1}^k)^\top(\sv_k^*, \su_k^*) + (\tilde{\su}_{i+1}^k - \su_{i+1}^k, \tilde{\sv}_{i+1}^k - \sv_{i+1}^k)^\top(\sv_{i+1}^k, \su_{i+1}^k) \\
& = & (\tilde{\su}_{i+1}^k - \su_{i+1}^k, \tilde{\sv}_{i+1}^k - \sv_{i+1}^k)^\top(\sv_{i+1}^k - \sv_k^*, \su_{i+1}^k - \su_k^*) \\
& = & (\p_i^k - \p_{i+1}^k, \q_i^k - \q_{i+1}^k)^\top(\sv_{i+1}^k - \sv_k^*, \su_{i+1}^k - \su_k^*) \\
& = & (\sv_i^k - \sv_{i+1}^k, \su_i^k - \su_{i+1}^k)^\top(\sv_{i+1}^k - \sv_k^*, \su_{i+1}^k - \su_k^*) \\
& = & \frac{1}{2}\left(\|\su_i^k - \su_k^*\|^2 + \|\sv_i^k - \sv_k^*\|^2 - \|\su_{i+1}^k - \su_k^*\|^2 - \|\sv_{i+1}^k - \sv_k^*\|^2 \right) - \frac{1}{2}\left(\|\su_i^k - \su_{i+1}^k\|^2 + \|\sv_i^k - \sv_{i+1}^k\|^2\right),
\end{eqnarray*}
we obtain
\begin{eqnarray}
& & \frac{\mu^k}{2D\beta}\left(\|\x_{i+1}^k - \x_k^*\|^2 + \|\s_{i+1}^k - \s_k^*\|^2 + (\tau_{i+1}^k - \tau_k^*)^2 + (\kappa_{i+1}^k - \kappa_k^*)^2 \right) \nn \\
& & + \frac{1}{2}\left(\|\su_i^k - \tilde{\su}_{i+1}^k\|^2 + \|\sv_i^k - \tilde{\sv}_{i+1}^k\|^2\right) \nn \\
& \leq & \|\su_i^k - \su_k^*\|^2 + \|\sv_i^k - \sv_k^*\|^2 - \|\su_{i+1}^k - \su_k^*\|^2 - \|\sv_{i+1}^k - \sv_k^*\|^2.\label{strong-convexity-1}
\end{eqnarray}
Moreover, from \eqref{opt-2} we have
\begin{eqnarray}\label{opt-2-inequality-1}
0 \leq B(\su, \sv, \mu^k) - B(\su_{i+1}^k, \sv_{i+1}^k, \mu^k) +\beta(\su-\su_{i+1}^k, \sv-\sv_{i+1}^k)^\top(\p_{i+1}^k, \q_{i+1}^k),
\end{eqnarray}
and
\begin{eqnarray}\label{opt-2-inequality-2}
0 \leq B(\su, \sv, \mu^k) - B(\su_i^k, \sv_i^k, \mu^k) +\beta(\su-\su_i^k, \sv-\sv_i^k)^\top(\p_i^k, \q_i^k) .
\end{eqnarray}
Letting $(\su, \sv) = (\su_i^k, \sv_i^k)$ in \eqref{opt-2-inequality-1} and $(\su, \sv) = (\su_{i+1}^k, \sv_{i+1}^k)$ in \eqref{opt-2-inequality-2}, and adding them up lead to
\begin{eqnarray*}
0 & \leq & -(\su_i^k-\su_{i+1}^k, \sv_i^k-\sv_{i+1}^k)^\top(\p_i^k - \p_{i+1}^k, \q_i^k - \q_{i+1}^k) \\
& = & (\su_i^k-\su_{i+1}^k, \sv_i^k-\sv_{i+1}^k)^\top(\su_{i+1}^k - \tilde{\su}_{i+1}^k, \sv_{i+1}^k - \tilde{\sv}_{i+1}^k) \\
& = & \frac{1}{2}\left(\|\su_i^k- \tilde{\su}_{i+1}^k\|^2 - \|\su_{i+1}^k-\tilde{\su}_{i+1}^k\|^2 - \|\su_i^k - \su_{i+1}^k\|^2\right) \\ & & + \frac{1}{2}\left(\|\sv_i^k- \tilde{\sv}_{i+1}^k\|^2 - \|\sv_{i+1}^k-\tilde{\sv}_{i+1}^k\|^2 - \|\sv_i^k - \sv_{i+1}^k\|^2\right),
\end{eqnarray*}
which implies that
\begin{equation}\label{strong-convexity-2}
\|\su_{i+1}^k-\tilde{\su}_{i+1}^k\|^2 + \|\sv_{i+1}^k-\tilde{\sv}_{i+1}^k\|^2 \leq \|\su_i^k - \tilde{\su}_{i+1}^k\|^2 + \|\sv_i^k - \tilde{\sv}_{i+1}^k\|^2.
\end{equation}
Combining \eqref{strong-convexity-1} and \eqref{strong-convexity-2} yields
\begin{eqnarray}
& & \frac{1}{2}\left(\|\su_{i+1}^k - \tilde{\su}_{i+1}^k\|^2 + \|\sv_{i+1}^k - \tilde{\sv}_{i+1}^k\|^2 \right) \nn \\
&& + \frac{\mu^k}{2D\beta}\left(\|\x_{i+1}^k - \x_k^*\|^2 + \|\s_{i+1}^k - \s_k^*\|^2 + (\tau_{i+1}^k - \tau_k^*)^2 + (\kappa_{i+1}^k - \kappa_k^*)^2 \right) \nn \\
& \leq & \|\su_i^k - \su_k^*\|^2 + \|\sv_i^k - \sv_k^*\|^2 - \|\su_{i+1}^k - \su_k^*\|^2 - \|\sv_{i+1}^k - \sv_k^*\|^2.\label{complexity-1/2+mu/2Dbeta}
\end{eqnarray}
Furthermore, we have (by denoting $C_4 := 6\lambda_{\max}(A^\top A) \max \{1, \|\bc\|^2, \|\bar{\bc}\|^2\} / \lambda_{\min}^2 (AA^\top)$)
\begin{eqnarray}
& & \|\y_{i+1}^k - \y_k^*\|^2 \nn \\
& \leq & 2\left(\|\y_{i+1}^k - \tilde{\y}_{i+1}^k\|^2 + \|\tilde{\y}_{i+1}^k - \y_k^*\|^2\right) \nn\\
& = & 2\|\y_{i+1}^k - \tilde{\y}_{i+1}^k\|^2 + 2\|(AA^\top)^{-1}A(A^\top\tilde{\y}_{i+1}^k - A^\top\y_k^*)\|^2 \nn\\
& = & 2\|\y_{i+1}^k - \tilde{\y}_{i+1}^k\|^2 + 2\|(AA^\top)^{-1}A(\bc\tilde{\tau}_{i+1}^k - \tilde{\s}_{i+1}^k-\bar{\bc}\tilde{\theta}_{i+1}^k-\bc\tau_k^*+\s_k^*+\bar{\bc}\theta_k^*)\|^2 \nn\\
& \leq & C_4\left(\|\tilde{\s}_{i+1}^k - \s_k^*\|^2 + (\tilde{\tau}_{i+1}^k - \tau_k^*)^2 + (\tilde{\theta}_{i+1}^k - \theta_k^*)^2\right) + 2\|\y_{i+1}^k - \tilde{\y}_{i+1}^k\|^2 \nn\\
& \leq & 2C_4\left(\|\s_{i+1}^k - \tilde{\s}_{i+1}^k\|^2 + (\tau_{i+1}^k - \tilde{\tau}_{i+1}^k)^2 + \|\s_{i+1}^k - \s_k^*\|^2 + (\tau_{i+1}^k - \tau_k^*)^2\right) \nn\\
& & + C_4(\theta_{i+1}^k - \theta_k^*)^2 + 2\|\y_{i+1}^k - \tilde{\y}_{i+1}^k\|^2,\label{complexity-final-y}
\end{eqnarray}
and (by denoting $C_5 := 3\max\{\|\bb\|^2, \|A\|^2\}/\|\bar{\bb}\|^2$)
\begin{eqnarray}\label{complexity-final-theta}
&& (\theta_{i+1}^k - \theta_k^*)^2 \ = \ (\tilde{\theta}_{i+1}^k - \theta_k^*)^2 \\
& = & \frac{1}{\|\bar{\bb}\|^2}\|\bb\tilde{\tau}_{i+1}^k - A\tilde{\x}_{i+1}^k + \tilde{\sr}_{i+1}^k -\bb\tau_k^*+A\x_k^*-\sr_k^*\|^2 \nn\\
& \leq & C_5 \left(\|\tilde{\x}_{i+1}^k - \x_k^*\|^2 + (\tilde{\tau}_{i+1}^k - \tau_k^*)^2 + \|\tilde{\sr}_{i+1}^k - \sr_k^*\|^2\right) \nn\\
& = & C_5\left(\|\tilde{\x}_{i+1}^k - \x_k^*\|^2 + (\tilde{\tau}_{i+1}^k - \tau_k^*)^2 + \|\tilde{\sr}_{i+1}^k - \sr_{i+1}^k\|^2\right) \nn\\
& \leq & 2C_5\left(\|\x_{i+1}^k - \tilde{\x}_{i+1}^k\|^2 + (\tau_{i+1}^k - \tilde{\tau}_{i+1}^k)^2 + \|\x_{i+1}^k - \x_k^*\|^2 + (\tau_{i+1}^k - \tau_k^*)^2\right) + C_5\|\tilde{\sr}_{i+1}^k - \sr_{i+1}^k\|^2. \nn
\end{eqnarray}
By summing up \eqref{complexity-final-y} and \eqref{complexity-final-theta}, we have
\begin{eqnarray}\label{complexity-final-y-plus-theta}
& & \|\y_{i+1}^k - \y_k^*\|^2 + (\theta_{i+1}^k - \theta_k^*)^2 \\
& \leq & 2(C_4+C_5) \left(\|\su_{i+1}^k - \tilde{\su}_{i+1}^k\|^2 + \|\sv_{i+1}^k - \tilde{\sv}_{i+1}^k\|^2\right) \nn\\
& & + 2C_4\left(\|\s_{i+1}^k - \s_k^*\|^2 + (\tau_{i+1}^k - \tau_k^*)^2 + (\theta_{i+1}^k - \theta_k^*)^2\right) + 2C_5\left(\|\x_{i+1}^k - \x_k^*\|^2 + (\tau_{i+1}^k - \tau_k^*)^2\right) \nn\\
& \overset{\eqref{complexity-final-theta}}{\leq} & 2(C_4+C_5) \left(\|\su_{i+1}^k - \tilde{\su}_{i+1}^k\|^2 + \|\sv_{i+1}^k - \tilde{\sv}_{i+1}^k\|^2\right)+ 2C_4\left(\|\s_{i+1}^k - \s_k^*\|^2 + (\tau_{i+1}^k - \tau_k^*)^2 \right) \nn\\
& & + 2C_5\left(\|\x_{i+1}^k - \x_k^*\|^2 + (\tau_{i+1}^k - \tau_k^*)^2\right) + 4C_4C_5\left(\|\x_{i+1}^k - \tilde{\x}_{i+1}^k\|^2 + (\tau_{i+1}^k - \tilde{\tau}_{i+1}^k)^2\right) \nn\\
& & + 4C_4C_5\left(\|\x_{i+1}^k - \x_k^*\|^2 + (\tau_{i+1}^k - \tau_k^*)^2\right) + 2C_4C_5\|\tilde{\sr}_{i+1}^k - \sr_{i+1}^k\|^2 \nn\\
& \overset{\eqref{def-C3}}{\leq} & C_3\left(\|\su_{i+1}^k - \tilde{\su}_{i+1}^k\|^2 + \|\sv_{i+1}^k - \tilde{\sv}_{i+1}^k\|^2\right) + C_3\left(\|\x_{i+1}^k - \x_k^*\|^2 + \|\s_{i+1}^k - \s_k^*\|^2 + (\tau_{i+1}^k - \tau_k^*)^2\right).\nn
\end{eqnarray}
Noting that $\|\sr_{i+1}^k - \sr_k^*\|^2 = (\xi_{i+1}^k - \xi_k^*)^2 = 0$, we have
\begin{eqnarray*}
& & \min\left\{\frac{1}{2C_3}, \frac{\mu^k}{4DC_3\beta}\right\}\left(\|\su_{i+1}^k - \su_k^*\|^2 + \|\sv_{i+1}^k - \sv_k^*\|^2\right) \\
& \leq & \frac{\mu^k}{4DC_3\beta}\left(\|\x_{i+1}^k - \x_k^*\|^2 + \|\s_{i+1}^k - \s_k^*\|^2 + (\tau_{i+1}^k - \tau_k^*)^2 + (\kappa_{i+1}^k - \kappa_k^*)^2 \right) \\
& & + \min\left\{\frac{1}{2C_3}, \frac{\mu^k}{4DC_3\beta}\right\}\left(\|\y_{i+1}^k - \y_k^*\|^2 + (\theta_{i+1}^k - \theta_k^*)^2\right) \\
& \leq & \frac{\mu^k}{4D\beta}\left(\|\x_{i+1}^k - \x_k^*\|^2 + \|\s_{i+1}^k - \s_k^*\|^2 + (\tau_{i+1}^k - \tau_k^*)^2 + (\kappa_{i+1}^k - \kappa_k^*)^2 \right) \\
&& + \frac{1}{2}\left(\|\su_{i+1}^k - \tilde{\su}_{i+1}^k\|^2 + \|\sv_{i+1}^k - \tilde{\sv}_{i+1}^k\|^2\right) + \frac{\mu^k}{4D\beta}\cdot\left(\|\x_{i+1}^k - \x_k^*\|^2 + \|\s_{i+1}^k - \s_k^*\|^2 + (\tau_{i+1}^k - \tau_k^*)^2\right) \\
& \leq & \frac{1}{2}\left(\|\su_{i+1}^k - \tilde{\su}_{i+1}^k\|^2 + \|\sv_{i+1}^k - \tilde{\sv}_{i+1}^k\|^2\right) \\
&& + \frac{\mu^k}{2D\beta}\left(\|\x_{i+1}^k - \x_k^*\|^2 + \|\s_{i+1}^k - \s_k^*\|^2 + (\tau_{i+1}^k - \tau_k^*)^2 + (\kappa_{i+1}^k - \kappa_k^*)^2 \right)  \\
& \overset{\eqref{complexity-1/2+mu/2Dbeta}}{\leq} & \|\su_i^k - \su_k^*\|^2 + \|\sv_i^k - \sv_k^*\|^2 - \|\su_{i+1}^k - \su_k^*\|^2 - \|\sv_{i+1}^k - \sv_k^*\|^2,
\end{eqnarray*}
where the second inequality is due to $C_3>1$ and \eqref{complexity-final-y-plus-theta}.
Therefore, we obtain that
\begin{eqnarray*}
\|\su_{N_k}^k - \su_k^*\|^2 + \|\sv_{N_k}^k - \sv_k^*\|^2 & \leq & \left(1+\min\left\{\frac{1}{2C_3}, \frac{\mu^k}{4DC_3\beta}\right\}\right)^{-N_k}\left(\|\su_0^k - \su_k^*\|^2 + \|\sv_0^k - \sv_k^*\|^2\right)  \\
& \leq & C\left(1+\min\left\{\frac{1}{2C_3}, \frac{\mu^k}{4DC_3\beta}\right\}\right)^{-N_k}.
\end{eqnarray*}
On the other hand, we have
\[\|Q\su_{N_k}^k-\sv_{N_k}^k\|^2 = \|Q(\su_{N_k}^k-\su_k^*)-(\sv_{N_k}^k-\sv_k^*)\|^2 \leq 2\|Q\|^2\|\su_{N_k}^k-\su_k^*\|^2+2\|\sv_{N_k}^k-\sv_k^*\|^2.\]
Therefore, The number of iterations (denoted by $N_k$) needed in the inner loop of Algorithm \ref{Algorithm:Simplified} should satisfy that
\begin{equation*}
2C\left(1+\left\|Q\right\|^2\right)\left(1+\min\left\{\frac{1}{2C_3}, \frac{\mu^k}{4DC_3\beta}\right\}\right)^{-N_k} \leq \mu^k,
\end{equation*}
which proves \eqref{bound-N_k}.
\end{proof}

Now we are ready to present the main result of the iteration complexity of Algorithm \ref{Algorithm:Simplified}.
\begin{theorem}\label{Theorem:Complexity-Total-Loop}
{Suppose that the \textsf{ABIP} is terminated when $\mu^k < \varepsilon$, where $\varepsilon>0$ is a pre-given tolerance.} The total \textsf{IPM} and \textsf{ADMM} iteration complexities of \textsf{ABIP} are respectively
\begin{equation*}
T_{\textsf{IPM}} = O\left(\log\left(\frac{1}{\varepsilon}\right)\right), \qquad T_{\textsf{ADMM}} = O\left(\frac{\kappa_A^2\left\|Q\right\|^2}{\varepsilon}\log\left(\frac{1}{\varepsilon}\right)\right).
\end{equation*}
where 
$\kappa_A := \lambda_{\max}(A^\top A)/\lambda_{\min}(AA^\top)$.
\end{theorem}
\begin{proof}
Note that \textsf{ABIP} consists of two types of loops: inner loops and outer loops. In the outer loop, a log-barrier penalty problem is formed with a penalty parameter $\mu^k$, and in the inner loop, this log-barrier penalty problem is solved by a two-block ADMM. The outer loop is terminated when $\mu^k < \varepsilon$, with a pre-given tolerance $\varepsilon>0$. It is then easy to see that the number of outer loops, i.e., the number of interior point iterations, is
\begin{equation*}
T_{\textsf{IPM}} = \left\lceil \frac{\log(\mu^0/\varepsilon)}{\log(1/\gamma)}\right\rceil.
\end{equation*}
For the {\it total}\/ number of ADMM steps {(rather than the ADMM steps between two IPM outer loops)}, we have the following estimate:
\begin{eqnarray*}
T_{\textsf{ADMM}} = \sum_{k=1}^{T_{\textsf{IPM}}} N_k & = & \sum_{k=1}^{T_{\textsf{IPM}}} \log\left(\frac{2C\left(1+\|Q\|^2\right)}{{\mu^k}}\right) \left[\log\left(1+\min\left\{\frac{1}{C_3}, \frac{\mu^k}{4DC_3\beta}\right\}\right)\right]^{-1} \\
& = & \sum_{k=1}^{T_{\textsf{IPM}}} \log\left(\frac{2C\left(1+\|Q\|^2\right)/\mu^0}{{\gamma^k}}\right)\left[\log\left(1+\min\left\{\frac{1}{C_3}, \frac{\mu^0\gamma^k}{4DC_3\beta}\right\}\right)\right]^{-1} \\
& = & O\left(\frac{\kappa_A^2\|Q\|^2}{\varepsilon}\log\left(\frac{1}{\varepsilon}\right)\right).
\end{eqnarray*}
This completes the proof.
\end{proof}

\begin{corollary}
Assume that $m=O(n)$, the total arithmetic operation complexity of \textsf{ABIP} is
\begin{equation*}
T = O\left(n^3 + \frac{n^2 \kappa_A^2\|Q\|^2}{\varepsilon} \log\left(\frac{1}{\varepsilon}\right)\right).
\end{equation*}
\end{corollary}
\begin{proof}
\textsf{ABIP} can be divided into two parts. In the first part, we decompose the matrix which requires $O(n^3)$ arithmetic operations. In the second part, we perform several matrix-vector and vector-vector operations where each ADMM iteration requires $O(n^2)$ arithmetic operations. So Theorem~\ref{Theorem:Complexity-Total-Loop} implies the desired result. This completes the proof.
\end{proof}

\section{Implementation Details}\label{section:implementation}
\subsection{Termination Criteria}\label{Section3:Termination}
So far we have not discussed how to terminate the outer loop of \textsf{ABIP} yet. In our implementation, we chose the one that is used in \textsf{SCS} \cite{Donoghue-2016-conic}.
Specifically, we run the algorithm until it finds a primal-dual optimal solution or a certificate of primal or dual infeasibility of the original LP pair \eqref{prob:LP}, up to some tolerance. The detailed procedure is as follows. If $\tau_i^k>0$ in $\su_i^k$, then let
\begin{equation*}
\left(\frac{\x_i^k}{\tau_i^k}, \frac{\s_i^k}{\tau_i^k}, \frac{\y_i^k}{\tau_i^k}\right)
\end{equation*}
be the candidate solution which is guaranteed to satisfy the feasibility condition. It thus suffices to check if the residuals
\begin{eqnarray*}
\pres_i^k = \frac{A\x_i^k}{\tau_i^k} - \bb, \ \dres_i^k = \frac{A^\top\y_i^k}{\tau_i^k} + \frac{\s_i^k}{\tau_i^k} - \bc, \ \dgap_i^k  =  \frac{\bc^\top\x_i^k}{\tau_i^k} - \frac{\bb^\top\y_i^k}{\tau_i^k},
\end{eqnarray*}
are small. More specifically, we terminate the algorithm if
\begin{align*}
\|\pres_i^k\| \leq \varepsilon_{\pres}(1+\|\bb\|), \ \|\dres_i^k\|  \leq \varepsilon_{\dres}(1+\|\bc\|), \ \|\dgap_i^k\| \leq \varepsilon_{\dgap}(1+|\bc^\top\x|+|\bb^\top\y|),
\end{align*}
are met. The quantities $\varepsilon_{\pres}$, $\varepsilon_{\dres}$ and $\varepsilon_{\dgap}$ are the primal residual, dual {residual} and duality gap tolerances, respectively.

On the other hand, if the current iterates satisfy that
\begin{equation}\label{Criteria:Primal-Infeasibility}
\|A^\top\y_i^k + \s_i^k\| \leq \varepsilon_{\pinfeas}\left(\frac{\bb^\top\y_i^k}{\|\bb\|}\right),
\end{equation}
then $\frac{\y_i^k}{\bb^\top\y_i^k}$ is an approximate certificate for the primal infeasibility with the tolerance $\varepsilon_{\pinfeas}$; or if the current iterates satisfy that
\begin{equation}\label{Criteria:Dual-Infeasibility}
\|A\x_i^k\| \leq \varepsilon_{\dinfeas}\left(\frac{-\bc^\top\x_i^k}{\|\bc\|}\right),
\end{equation}
then $-\frac{\x_i^k}{\bc^\top\x_i^k}$ is an approximate certificate for the dual infeasibility with the tolerance $\varepsilon_{\dinfeas}$.

\subsection{Over Relaxation}
In practice, we implemented some techniques that can accelerate the algorithm. One of them is to incorporate a relaxation parameter in the ADMM \cite{Eckstein-1992-douglas}. Specifically, when applied to Algorithm~\ref{Algorithm:Simplified}, we replace all $\tilde{\su}_{i+1}^k$ in the $\su$- and $\sv$-updates with
\begin{equation*}
\alpha\tilde{\su}_{i+1}^k + (1-\alpha)\su_i^k,
\end{equation*}
where $\alpha\in\left[0,2\right]$ is a relaxation parameter. {In that case, \eqref{Simple-Update:Dual}, \eqref{Simple-Update:Barrier-Y}, \eqref{Simple-Update:Barrier-X}, \eqref{Simple-Update:Barrier-tau} and \eqref{Simple-Update:Barrier-theta} become}
\begin{equation*}
\sv_{i+1}^k = \sv_i^k - \alpha\tilde{\su}_{i+1}^k - (1-\alpha)\su_i^k + \su_{i+1}^k,
\end{equation*}
and
{ \begin{eqnarray*}
\y_{i+1}^k & = & \alpha\tilde{\y}_{i+1}^k + (1-\alpha)\y_i^k,   \\
\x_{i+1}^k & = & \frac{1}{2}\left[\left(\alpha\tilde{\x}_{i+1}^k + (1-\alpha)\x_i^k - \s_i^k\right) + \sqrt{\left(\alpha\tilde{\x}_{i+1}^k + (1-\alpha)\x_i^k - \s_i^k\right)\circ\left(\alpha\tilde{\x}_{i+1}^k + (1-\alpha)\x_i^k - \s_i^k\right)+\frac{4\mu^k}{\beta}}\right], \\
\tau_{i+1}^k & = & \frac{1}{2}\left[\left(\alpha\tilde{\tau}_{i+1}^k + (1-\alpha)\tau_i^k - \kappa_i^k\right) + \sqrt{\left(\alpha\tilde{\tau}_{i+1}^k + (1-\alpha)\tau_i^k - \kappa_i^k\right)\circ\left(\alpha\tilde{\tau}_{i+1}^k + (1-\alpha)\tau_i^k - \kappa_i^k\right)+\frac{4\mu^k}{\beta}}\right], \\
\theta_{i+1}^k & = & \alpha\tilde{\theta}_{i+1} + (1-\alpha) \theta_i^k.
\end{eqnarray*}}
When $\alpha=1$, this reduces to the corresponding update in Algorithm~\ref{Algorithm:Simplified} given above; when $\alpha>1$, this is known as \textit{over-relaxation}; when $\alpha<1$, this is known as \textit{under-relaxation}. Previous works \cite{Eckstein-1994-parallel, Donoghue-2013-splitting} suggest that the performance of ADMM can be improved significantly if one sets $\alpha\approx 1.5$.

\subsection{Barzilai-Borwein Spectral Method for Selecting $\beta$}\label{Section3:Selection}

The performance of ADMM highly depends on the choice of $\beta$. One way to accelerate ADMM is to adaptively adjust $\beta$ (see also \cite{He-2000-alternating, Wen-2010-alternating}). In practice, we generated a sequence of $\{\beta^k\}_{k\geq 0}$, where $\beta^k$ is only used in the $k$-th outer iteration. Intuitively, the speed of traveling along the central path is determined by $\mu^k$ and $\beta$, implying that adjusting $\beta$ in each outer iteration based on the iterates is equivalent to a predictor-corrector method \cite{Mehrotra-1992-implementation}.
The way we adaptively adjust $\beta$ is based on the Barzilai-Borwein spectral method \cite{Barzilai-1988-two,Xu-2017-adaptive}, which is proven to be superior than the residue balancing approach \cite{Wen-2010-alternating}. Indeed, for each $k\geq 0$, we select $\beta^k$ using \textbf{spectral stepsize estimation} and \textbf{safeguarding} at the beginning of the $k$-th outer iteration.

\textbf{Spectral stepsize estimation:} We calculate the first three iterates, i.e., $\left(\tilde{\su}_i^k, \su_i^k, \sv_i^k\right)_{i=0}^{i=2}$, {using a fixed $\beta^k>0$ and an initial point $\left(\y_0^k, \x_0^k, \s_0^k, \tau_0^k, \kappa_0^k, \theta_0^k\right)$ obtained by~\eqref{update-next-iterate} and $\sr_0^k = 0$, $\xi_0^k=-n-1$.} Then we estimate the curvature, i.e.,
\begin{eqnarray*}
\hat{\sv}_1^k = \sv_0^k - \alpha\tilde{\su}_1^k - (1-\alpha) \su_0^k +\alpha\su_1^k, \quad
\hat{\sv}_2^k = \sv_1^k - \alpha\tilde{\su}_2^k - (1-\alpha) \su_1^k + \alpha\su_2^k,
\end{eqnarray*}
and
\begin{eqnarray*}
\Delta\hat{\sv}^k  =   \hat{\sv}_2^k - \hat{\sv}_1^k, \quad
\Delta\tilde{\su}^k  =   \alpha\left(\tilde{\su}_2^k - \tilde{\su}_1^k\right) + (1-\alpha)\left(\su_1^k - \su_0^k\right).
\end{eqnarray*}
As is typical in Barzilai-Borwein stepsize gradient methods \cite{Barzilai-1988-two}, the spectral stepsizes $\varphi_{\textsf{SD}}^k$ and $\varphi_{\textsf{MG}}^k$ have the closed-form solutions as
\begin{eqnarray*}
\varphi_{\textsf{SD}}^k = \frac{\left\langle\Delta\hat{\sv}^k, \Delta\hat{\sv}^k\right\rangle}{\left\langle\Delta\hat{\sv}^k, \Delta\tilde{\su}^k\right\rangle}, \quad
\varphi_{\textsf{MG}}^k = \frac{\left\langle\Delta\hat{\sv}^k, \Delta\tilde{\su}^k\right\rangle}{\left\langle\Delta\tilde{\su}^k, \Delta\tilde{\su}^k\right\rangle},
\end{eqnarray*}
where \textsf{SD} and \textsf{MG} refer to \textit{steepest descent} and \textit{minimum gradient}, respectively, and $\varphi_{\textsf{SD}}^k \geq \varphi_{\textsf{MG}}^k$ due to the Cauchy-Schwarz inequality. We then consider the hybrid stepsize rule proposed in \cite{Zhou-2006-gradient},
\begin{equation}
\varphi^k = \left\{\begin{array}{ll} \varphi_{\textsf{MG}}^k, & \text{if} \ 2\varphi_{\textsf{MG}}^k >\varphi_{\textsf{SD}}^k, \\ \varphi_{\textsf{SD}}^k - \varphi_{\textsf{MG}}^k/2, & \text{otherwise.} \end{array}\right.
\end{equation}
Similarly, we calculate
\begin{equation*}
\Delta\su^k = -\left(\su_2^k - \su_1^k\right),
\end{equation*}
and the spectral stepsizes $\psi_{\textsf{SD}}^k$ and $\psi_{\textsf{MG}}^k$ as
\begin{eqnarray*}
\psi_{\textsf{SD}}^k = \frac{\left\langle\Delta\hat{\sv}^k, \Delta\hat{\sv}^k\right\rangle}{\left\langle\Delta\hat{\sv}^k, \Delta\su^k\right\rangle}, \quad
\psi_{\textsf{MG}}^k = \frac{\left\langle\Delta\hat{\sv}^k, \Delta\su^k\right\rangle}{\left\langle\Delta\su^k, \Delta\su^k\right\rangle},
\end{eqnarray*}
and consider the hybrid stepsize rule again,
\begin{equation}
\psi^k = \left\{ \begin{array}{ll} \psi_{\textsf{MG}}^k, & \text{if} \ 2\psi_{\textsf{MG}}^k > \psi_{\textsf{SD}}^k, \\ \psi_{\textsf{SD}}^k - \psi_{\textsf{MG}}^k/2, & \text{otherwise.} \end{array}\right.
\end{equation}

\textbf{Safeguarding:} 
We suggest a safeguarding heuristic by accessing the quality of the curvature estimates, i.e., only update the stepsize if the curvature estimates satisfy a reliability criterion. More specifically, we consider the following quantities defined in \cite{Xu-2017-adaptive}:
\begin{eqnarray*}
\varphi_{\textsf{cor}}^k = \frac{\left\langle\Delta\hat{\sv}^k, \Delta\tilde{\su}^k\right\rangle}{\left\|\Delta\hat{\sv}^k\right\|\left\|\Delta\tilde{\su}^k\right\|}, \quad
\psi_{\textsf{cor}}^k = \frac{\left\langle\Delta\hat{\sv}^k, \Delta\su^k\right\rangle}{\left\|\Delta\hat{\sv}^k\right\|\left\|\Delta\su^k\right\|}.
\end{eqnarray*}
The spectral stepsizes are updated only if the estimation is sufficiently reliable, i.e.,
\begin{equation}
\hat{\beta}^k = \left\{\begin{array}{ll}
\sqrt{\varphi^k\psi^k}, & \text{if} \ \varphi_{\textsf{cor}}^k>\varepsilon_{\textsf{cor}} \ \text{and} \ \psi_{\textsf{cor}}^k>\varepsilon_{\textsf{cor}}, \\
\varphi^k, & \text{if} \ \varphi_{\textsf{cor}}^k>\varepsilon_{\textsf{cor}} \ \text{and} \ \psi_{\textsf{cor}}^k\leq \varepsilon_{\textsf{cor}}, \\
\psi^k, & \text{if} \ \varphi_{\textsf{cor}}^k\leq\varepsilon_{\textsf{cor}} \ \text{and} \ \psi_{\textsf{cor}}^k>\varepsilon_{\textsf{cor}}, \\
\beta^k, & \text{otherwise}, \end{array}\right.
\end{equation}
where $\varepsilon_{\textsf{cor}}>0$ is a quality threshold for the curvature estimates. {Notice that $\hat{\beta}^k=\beta^k$ when both curvature estimates are deemed inaccurate while $\hat{\beta}^k\neq\beta^k$ but $\hat{\beta}^k\approx\beta^k$ implies that $\hat{\beta}^k$ and $\beta^k$ are both suitable to be used in the $k$-th outer iteration. }

{{\textbf{A heuristic selection:} We select a threshold $\varepsilon_{\textsf{penalty}}>0$ and set}
\begin{enumerate}
\item If $\hat{\beta}^k\neq\beta^k$ and $\left|\hat{\beta}^k-\beta^k\right|\leq\varepsilon_{\textsf{penalty}}$, then $\frac{\beta^k+\hat{\beta}^k}{2}$ will be used in the $k$-th outer iteration.
\item If $\hat{\beta}^k\neq\beta^k$ and $\left|\hat{\beta}^k-\beta^k\right|>\varepsilon_{\textsf{penalty}}$, then $\beta^k=\hat{\beta}^k$ and we redo the spectral step-size estimation and safeguarding with the same initial point.
\item If $\hat{\beta}^k=\beta^k$, then the spectral step-size estimation and safeguarding will be continued based on the subsequent iterates.
\end{enumerate}}


\subsection{Linear System Solver}\label{Section4:Factorization}
In this subsection, we discuss how to efficiently compute the projection onto the subspace $\QCal=\{(\su, \sv): \ Q\su=\sv\}$ in \eqref{Simple-Update:Affine}. This requires solving the linear system $\left(I+Q\right)\su = \w$ for some $\w\in\br^{m+n+2}$, i.e.,
\begin{equation*}
\left[\begin{array}{c} \w_{\y} \\ \w_{\x} \\ \w_{\tau} \end{array}\right] = \left[\begin{array}{cc} M & \h \\ -\h^\top & 1\end{array}\right]\left[\begin{array}{c} \su_{\y} \\ \su_{\x} \\ \su_{\tau} \end{array}\right],
\end{equation*}
where
\begin{equation*}
M = \left[\begin{array}{cc} I & A \\ -A^\top & I \end{array}\right], \quad \h = \left[\begin{array}{c} -\bb \\ \bc \end{array}\right].
\end{equation*}
This implies that
\begin{equation*}
\su_\tau = \w_\tau + \h^\top\left[\begin{array}{c} \su_{\y} \\ \su_{\x} \end{array}\right], \qquad \left[\begin{array}{c} \su_{\y} \\ \su_{\x} \end{array}\right] = \left(M+\h\h^\top\right)^{-1}\left(\left[\begin{array}{c} \w_{\y} \\ \w_{\x} \end{array}\right] - \w_{\tau}\h\right),
\end{equation*}
where $M+\h\h^\top$ is the Schur complement of the lower right block $1$ in $I+Q$. Therefore, we can apply the Sherman-Morrison-Woodbury formula \cite{Press-2007-umerical} to $\left(M+\h\h^\top\right)^{-1}$ and obtain
\begin{equation*}
\left[\begin{array}{c} \su_{\y} \\ \su_{\x} \end{array}\right] = \left(M^{-1}-\frac{M^{-1}\h\h^\top M^{-1}}{1+\h^\top M^{-1}\h}\right)\left(\left[\begin{array}{c} \w_{\y} \\ \w_{\x} \end{array}\right] - \w_{\tau}\h\right),
\end{equation*}
and
\begin{equation*}
\su_\tau = \w_\tau + \bc^\top\su_{\x} - \bb^\top\su_{\y}.
\end{equation*}
Therefore, \eqref{Simple-Update:Affine} reduces to
\begin{equation}\label{Practical-Update:Affine-YX}
\left[\begin{array}{c} \tilde{\y}_{i+1}^k \\ \tilde{\x}_{i+1}^k \end{array}\right] = \left(M^{-1}-\frac{M^{-1}\h\h^\top M^{-1}}{1+\h^\top M^{-1}\h}\right)\left(\left[\begin{array}{c} \y_i^k + \sr_i^k \\ \x_i^k + \s_i^k \end{array}\right] - \left(\tau_i^k+\kappa_i^k\right)\h\right),
\end{equation}
and
\begin{equation}\label{Practical-Update:Affine-tau}
\tilde{\tau}_{i+1}^k= \tau_i^k+\kappa_i^k + \bc^\top\tilde{\x}_{i+1}^k - \bb^\top\tilde{\y}_{i+1}^k.
\end{equation}
\begin{remark}
In view of practical implementation, we can calculate and cache $M^{-1}\h$ before the first iteration. In the subsequent iterations, the main computational cost lies in the calculation of
\begin{equation*}
\left[\begin{array}{c} \z_{\y} \\ \z_{\x} \end{array}\right] = M^{-1}\left[\begin{array}{c} \y_i^k + \sr_i^k \\ \x_i^k + \s_i^k \end{array}\right],
\end{equation*}
while the other simple vector operations using cached quantities are cheap.
\end{remark}
To solve the linear system of the form
\begin{equation}\label{Factorization:System}
\left[\begin{array}{cc} I & -A \\ -A^\top & -I \end{array}\right] \left[\begin{array}{c} \z_{\y} \\ -\z_{\x} \end{array}\right]= \left[\begin{array}{c} \y_i^k + \sr_i^k \\ \x_i^k + \s_i^k \end{array}\right],
\end{equation}
we follow~\cite{Donoghue-2016-conic} and propose two approaches. The first approach is to computes a sparse permuted \textsf{LDL$^\top$} factorization \cite{Davis-2006-direct} of the matrix in \eqref{Factorization:System} before the first iteration and use this cached factorization to solve the system in the subsequent steps. When the factorization cost dominates, this technique is very effective since only one factorization is necessary and all iterations after that can be carried out quickly. The second approach is to approximately solve the system, involves first reformulate~\eqref{Factorization:System} and then solve it with the conjugate gradient method (CG)~\cite{Saad-2003-Iterative, Watkins-2004-Fundamentals}. Each iteration of conjugate gradient requires multiplying once by $A$ and once by $A^\top$, each of which
can be parallelized. If $A$ is very sparse, then these multiplications can be performed especially quickly; when $A$ is dense, it may be better to first form $G = I + AA^\top$ in the setup phase. We terminate the CG iterations using the standard same criterion; see~\cite{Donoghue-2016-conic} for the details.


\subsection{Presolving and Postsolving}\label{Section4:Preprocess}

Now we discuss issues in analyzing large and sparse LPs prior to solving them with \textsf{ABIP}. Firstly, we remove several computational expensive sub-procedures, e.g., forcing, dominated and duplicate rows and columns procedures, as used in \cite{Andersen-1995-presolving, Gondzio-1997-presolve, Meszaros-2003-advanced, Gould-2004-preprocessing}. Secondly, we use \textsf{Dulmage-Mendelsohn decomposition} \cite{Pothen-1990-computing} to remove all the dependent rows in $A$ and reformulate the original problem in the form of problem~\eqref{prob:LP}.

We consider LPs formulated in the following form:
\begin{equation*}
\ba{lll}
& \min & \bc^\top\x \\
\quad & \st  & A\x = \bb, \\
\quad & & \bl \leq \x \leq \w,
\ea
\end{equation*}
where $A$ has some linearly dependent rows. Before solving them with \textsf{ABIP}, we run a simple presolve procedure. More specifically, we detect and remove empty rows, singleton rows, fixed variables and empty columns, together with removing all the linearly dependent rows, i.e.,
\begin{enumerate}
\item[(P1)] \textsf{An empty row: $\exists i: \ a_{ij}=0, \ \forall j$.} Either the $i$-th constraint is redundant or infeasible.
\item[(P2)] \textsf{An empty column: $\exists j: \ a_{ij}=0, \ \forall i$.} $x_j$ is fixed at one of its bounds or the problem is unbounded.
\item[(P3)] \textsf{An infeasible variable: $\exists j: \ l_j > w_j$.} The problem is trivially infeasible.
\item[(P4)] \textsf{A fixed variable: $\exists j: \ l_j=w_j$.} $x_j$ can be substituted out of the problem.
\item[(P5)] \textsf{A free variable: $\exists j: \ l_j=-\infty, \ w_j=\infty$.} $x_j$ can be substituted by two nonnegative variables $x_j^+$ and $x_j^-$.
\item[(P6)] \textsf{A singleton row: $\exists\left(i,k\right): \ a_{ij}=0, \ \forall j\neq k, \ a_{jk}\neq 0$.} The $i$-th constraint fixes variable $x_j=\frac{b_i}{a_{ik}}$.
\item[(P7)] \textsf{Dulmage-Mendelsohn decomposition.} All the linearly dependent rows are detected and removed.
\end{enumerate}
Then we have a reduced LP problem as follows,
\begin{equation*}
\ba{lll}
& \min & \tilde{\bc}^\top\tilde{\x} \\
\quad & \st  & \tilde{A}\tilde{\x} = \tilde{\bb}, \\
\quad & & \tilde{\bl} \leq \tilde{\x} \leq \tilde{\w},
\ea
\end{equation*}
where $\tilde{\x}, \tilde{\bl}, \tilde{\w}\in\br^{\tilde{n}}$ and $\tilde{A}\in\br^{\tilde{m}+\tilde{n}}$ has full row rank. The last step is to reformulate the above problem as in the form of problem~\eqref{prob:LP}. After the presolving, it is guaranteed that $\tilde{\bl}>-\infty$. Now we can define $U=\{j: \tilde{\w}_j<+\infty\}$ and $\bar{\x}=\tilde{\x}-\tilde{\bl}$, and obtain the desired LP problem,
\begin{equation*}
\ba{lll}
& \min & \tilde{\bc}^\top\bar{\x} \\
\quad & \st  & \tilde{A}\bar{\x} = \tilde{\bb} - \tilde{A}\tilde{\bl}, \\
\quad & & \tilde{\x}_U + \tilde{\s} = \tilde{\w}_U-\tilde{\bl}_U, \\
\quad & & \tilde{\x} \geq 0, \ \tilde{\s}\geq 0.
\ea
\end{equation*}
After the presolve procedure, the reduced problem is ready to be solved by \textsf{ABIP}. A postsolve procedure is used to convert the solution to the reduced problem back to the solution to the original problem. 

\begin{remark}
Since our algorithm is a purely first-order algorithm, we also conduct the scale procedure after the presolve procedure to make the problems more well-conditioned. We refer to \cite{Donoghue-2016-conic} for more details.
\end{remark}

\section{Numerical Experiments}\label{Section5:Experiment}
In this section, we report experimental results of \textsf{ABIP} on solving randomly generated LPs, 6 instances from UCI collection\footnote{https://archive.ics.uci.edu/ml/datasets.html} and 114 instances from NETLIB collection\footnote{http://users.clas.ufl.edu/hager/coap/format.html}. To make the comparison fair, we compare the direct and indirect \textsf{ABIP} with different group of state-of-the-art solvers. More specifically, we compare the direct \textsf{ABIP} with \textsf{SDPT3} \cite{Tutuncu-2003-solving}, \textsf{MOSEK} \cite{Andersen-2009-homogeneous} and the direct \textsf{SCS} \cite{Donoghue-2016-conic} and the indirect \textsf{ABIP} with \textsf{DSDP-CG} \cite{Benson2000} and the indirect \textsf{SCS} \cite{Donoghue-2016-conic}.

Our \textsf{ABIP} code is written in C with a MATLAB interface. It has multi-threaded and single-threaded versions. The direct \textsf{ABIP} computes the (approximate) projections onto the subspace using a direct method based on a single-threaded sparse permuted LDL$^\top$ decomposition from the SuiteSparse package \cite{Davis-2006-direct, Davis-2005-algorithm, Amestoy-2004-algorithm}, and the indirect \textsf{ABIP} using a preconditioned conjugate gradient method (denoted by \textsf{ABIP}-CG). {For a fair comparison, we compare \textsf{ABIP} with \textsf{SCS (v2.0.0)} where the direct solver is single-threaded.} 

{In the experimental results reported below, we use the termination criteria for \textsf{ABIP} in Sections~\ref{Section3:Termination} with default values
{\begin{equation}\label{target-accuracy}
\epsilon=\epsilon_{\pres} = \epsilon_{\dres} = \epsilon_{\dgap} = \epsilon_{\pinfeas} = \epsilon_{\dinfeas} = 10^{-3} \textnormal{ or } 10^{-5},
\end{equation}}
and that for Barzilai-Borwein spectral method in Section~\ref{Section3:Selection} with default values
\begin{equation*}
\epsilon_{\textsf{cor}} = 0.2, \quad \epsilon_{\textsf{penalty}} = 0.1.
\end{equation*}}

{
For the scaling update in~\eqref{update-next-iterate}, a better heuristic update for $\x$, $\s$, $\tau$ and $\kappa$ is available in practice and we used it in our software package\footnote{{We sincerely thank an anonymous reviewer for suggesting this scaling strategy, which significantly improved the efficiency of the algorithm.}}. Indeed, we used
\begin{equation*}
((\x_0^{k+1})_j, (\s_0^{k+1})_j) = \left\{\begin{array}{rl}
((\x_{i+1}^k)_j, \gamma \cdot (\s_{i+1}^k)_j), & \textnormal{if }  (\x_{i+1}^k)_j \geq (\s_{i+1}^k)_j, \\
(\gamma \cdot (\x_{i+1}^k)_j, (\s_{i+1}^k)_j), & \textnormal{if } (\s_{i+1}^k)_j \geq (\x_{i+1}^k)_j.
\end{array}\right. \quad j = 1, 2, \ldots, n,
\end{equation*}
and
\begin{equation*}
((\tau_0^{k+1}), (\kappa_0^{k+1})) = \left\{\begin{array}{rl}
(\tau_{i+1}^k, \gamma \cdot \kappa_{i+1}^k), & \textnormal{if } \tau_{i+1}^k \geq \kappa_{i+1}^k, \\
(\gamma \cdot \tau_{i+1}^k, \kappa_{i+1}^k), & \textnormal{if } \kappa_{i+1}^k \geq \tau_{i+1}^k.
\end{array}\right.
\end{equation*}
}

The over-relaxation parameter is set to $\alpha=1.8$. Moreover, we set the maximum number of ADMM steps of \textsf{ABIP} to $10^6$. If the target accuracy in \eqref{target-accuracy} is not achieved in $10^6$ ADMM steps, we terminate the code and claim that \textsf{ABIP} fails to solve this instance and use ``---'' in the table to indicate the failure.

The decreasing ratio $\gamma$ is adjusted according to the value of the barrier parameter $\mu$, primal/dual feasibility violations and the duality gap. More specifically, we increase $\gamma$ as the iterate approaches the optimal solution. The objective value reported for all methods in the experiments below is the one after postsolving. The time required to do any preprocessing, i.e., presolving, postsolving, do/undo scaling and matrix factorization are included in the total solve times. All the experiments were carried out on a laptop with Linux system and 8 2.60GHz cores and 16Gb of RAM.

For other four solvers, we use the following stopping criteria. For \textsf{SCS}, we change the default $\alpha$ from $1.5$ to $1.8$, {use the tolerance in the stopping tolerance as $10^{-3}$ or $10^{-5}$}, and set the maximum number ADMM steps to $10^6$. These changes are made to ensure a more fair comparison, because we found that the default parameter setting needs much longer time to converge. Moreover, we use \textsf{SCS-CG} to denote \textsf{SCS} with a preconditiond conjugate gradient method computing the projections onto the subspaces. For \textsf{SDPT3} and \textsf{DSDP-CG}, the maximum number of interior point steps are both set as the default value 100. For \textsf{MOSEK}, we use the default settings. For all these solvers, we claim that they fail to solve an instance (denoted by ``---'' in the tables) if after the codes terminate, the target accuracy in \eqref{target-accuracy} is not achieved.

\subsection{Random LP Instances}
In this subsection we test the five solvers on randomly generated dense LPs.
First, we generate a Gaussian random vector $\x\in\br^n$ and split its entries randomly into three groups. The first group consists $60\%$ of entries and their values are set to zero. The second group consists of $10\%$ entries and their values are zoom in for ten times larger. The rest of the entries are in the third group and their values are zoomed out for ten times smaller. We then generate vector $\s\in\br^n$ such that $\x_i\s_i=0$ for all $i$, and nonzero entries of $\s$ follow standard normal distribution.
This ensures the complementary slackness and zero duality gap. We generate the data matrix $A = U^\top DV \in\br^{m \times n}$ where $D=\diag(\textnormal{randn}(m, 1))$, $U = \textnormal{sprandn}(m, m, 0.2)$ and $V = \textnormal{sprandn}(m, n, 0.2)$. We generate the dual solution $\y\in\br^m$ with entries following standard normal distribution. Finally, we set $\bb=A\x$ and $\bc=A^\top\y+\s$, which ensures primal and dual feasibility. The solution to the problem is not necessarily unique, but the optimal value is given by $\bc^\top\x=\bb^\top\y$.

\textit{Results:} We report the comparison results of the four direct solvers in Table \ref{Table:random-LP-direct} and the three indirect solvers in Table \ref{Table:random-LP-indirect}. {The results in Table \ref{Table:random-LP-direct} show that \textsf{ABIP} compares favorably to \textsf{SCS} and \textsf{SDPT3}, though it is inferior to the commercial solver \textsf{MOSEK}. For results in Table \ref{Table:random-LP-indirect}, our \textsf{ABIP-CG} compares favorably to \textsf{SCS-CG} and \textsf{DSDP-CG}.}

\subsection{Sparse Inverse Covariance Estimation}
In this subsection, we compare the five solvers on solving the following problem which arises from machine learning applications: \begin{eqnarray}\label{Prob:SICE}
& \min\limits_{\Omega\in\br^{d\times d}} & \left\|\Omega\right\|_{1} \\
& \st & \left\| \Sigma\Omega - I_d\right\|_{\infty} \leq \lambda,  \nonumber
\end{eqnarray}
where $\Sigma\in\br^{d\times d}$ denotes a sample covariance matrix, and $\lambda>0$ denotes some noisy tolerance. This problem, known as sparse inverse covariance estimation (SICE), is widely studied in high-dimensional statistics and machine learning \cite{Cai2011}. For given $\Sigma$, SICE \eqref{Prob:SICE} aims to find a perturbed inverse convariance matrix which is also sparse. 
Note that SICE \eqref{Prob:SICE} is separable for columns of $\Omega = \left(\beta_1, \beta_2, \cdots, \beta_d\right)$, and thus can be decomposed to $d$ problems as follows:
\begin{eqnarray}\label{Prob:SICE-Subproblem}
& \min\limits_{\beta_j\in\br^d} & \left\|\beta_j\right\|_{1} \\
& \st & \left\| \Sigma\beta_j - e_j\right\|_{\infty} \leq \lambda. \nonumber
\end{eqnarray}
Problem~\eqref{Prob:SICE-Subproblem} can be written as a standard LP as follows: 
\begin{eqnarray}\label{Prob:SICE-Subproblem-LP}
\begin{array}{ll}
\min & \e^\top\beta^{+} + \e^\top\beta^{-} \\
\st & {\Sigma\beta^{+} - \Sigma\beta^{-}} + w^{+} = \lambda\e + e_j , \\
& w^{+} + w^{-} = 2\lambda\e, \\
& \beta^{+}, \beta^{-}, w^{+}, w^{-} \geq 0,
\end{array}
\end{eqnarray}
where the number of variables is $n=4d$ and the number of constraints is $m=2d$. In our experiment, we set $\lambda=\frac{3}{2}\sqrt{\frac{\log(d)}{N}}$ where $N$ is the number of sampled data in the original data.

\textit{Problem instances:} We obtained $\Sigma$ from the UCI Machine Learning Repository\footnote{https://archive.ics.uci.edu/ml/datasets.html}. The statistics of the 6 selected instances is summarized in Table~\ref{Table:LP-UCI}.

\textit{Results:} Detailed numerical results are reported in Tables~\ref{Table:LP-UCI-Direct} and~\ref{Table:LP-UCI-Indirect}. From Table~\ref{Table:LP-UCI-Direct}, we see that \textsf{MOSEK} is the best among all the direct solvers possibly because of its preprocessing procedure of detecting dependent columns. \textsf{ABIP} and \textsf{SCS} are comparable and the speedup over \textsf{SDPT3} is more significant as the problem size increases. \textsf{ABIP} is more robust than \textsf{SCS} and \textsf{SDPT3} as \textsf{SCS} fails on ucihapt and \textsf{SDPT3} fail on gisette. From Table~\ref{Table:LP-UCI-Indirect}, \textsf{ABIP-CG} is more robust than \textsf{SCS-CG} and \textsf{DSDP-CG} as \textsf{SCS-CG} fails on ucihapt and ucihar and \textsf{DSDP-CG} fail on gisette.

{ 
\subsection{NETLIB LP Collections}

In this subsection, we report the performance of all five solvers on 114 feasible instances from NETLIB collection\footnote{http://www.netlib.org/lp/}. 

\textit{Problem instances:} NETLIB is a collection of LPs from real applications. It has been recognized as the standard testing data set for LP. The problem statistics of the 114 feasible instances before and after presolving is summarized in Tables \ref{Table:LP-NETLIB-Presolve-1} and \ref{Table:LP-NETLIB-Presolve-2}.



\textit{Results:} A summary of the numerical resutls on NETLIB is given in Tables~\ref{Table:LP-NETLIB-Main-Direct-1},~\ref{Table:LP-NETLIB-Main-Direct-2} and~\ref{Table:LP-NETLIB-Main-Indirect}.
Detailed results for different tolerance are given in Tables \ref{Table:LP-NETLIB-Direct-1}-\ref{Table:LP-NETLIB-Indirect-2}.
We observe that neither \textsf{SCS} nor \textsf{SCS-CG} is as robust as other solvers. {To be more specific, \textsf{SCS} only successfully solved 89 and 63 problems when the target accuracy is set to $10^{-3}$ and $10^{-5}$ respectively, while \textsf{SCS-CG} only successfully solved 86 problems. \textsf{MOSEK} is the most robust one while \textsf{ABIP}, \textsf{ABIP-CG}, \textsf{SDPT3} and \textsf{DSDP-CG} are comparable, and they all significantly outperform \textsf{SCS} and \textsf{SCS-CG}. This phenomenon can be explained by the superior robustness of the interior-point methods over the pure first-order methods. Furthermore, \textsf{ABIP} and \textsf{ABIP-CG} are also more efficient than \textsf{SCS} and \textsf{SCS-CG} on the collection of problem instances that can be solved by all the solvers. Such promising performance of \textsf{ABIP} and \textsf{ABIP-CG} strongly supports the usage of the first-order interior-point method on very large LP problems.}
}

{
\begin{remark}
From the numerical results we have the following observations. (i) The second-order IPM is still the best option either when the dataset is of small or medium size, or when solution with high accuracy is desired. (ii) The first-order IPM (\textsf{ABIP}) is better than the vanilla first-order method (ADMM) when the dataset is extremely large. (iii) We only focused on LP in this paper; how to design general first-order IPM for other convex optimization problems remains a future research topic.
\end{remark}
}

\section{Conclusions}\label{Section6:Conclusion}
In this paper we present a novel implementation of the primal-dual interior point method to solve linear programs via self-dual embedding. In our approach, we use the ADMM to track the central path. Therefore, the new approach is an implementation of first-order interior point method (IPM). As such, it inherits intrinsic properties of the IPM. We present a theoretical analysis showing that the overall complexity of ADMM steps is $O\left(\frac{1}{\varepsilon}\log\left(\frac{1}{\varepsilon}\right)\right)$, and the extensive numerical experiments demonstrate that the new algorithm is stable in performance and scalable in size. For the future research, we are interested in improving the numerical stability by incorporating advanced iterative methods, and the numerical efficiency by incorporating distributed computing environment.

\bigskip

\noindent
{\bf Acknowledgements.}
We would like to thank two anonymous referees for their insightful and constructive suggestions that greatly improved the presentation of this paper. We would like to thank Zaiwen Wen for fruitful discussions regarding this project and many helps on the codes. The research of Shiqian Ma was supported in part by NSF grants DMS-1953210 and CCF-2007797, and UC Davis CeDAR (Center for Data Science and Artificial Intelligence Research) Innovative Data Science Seed Funding Program.


\newpage

\begin{table}[t]\small
\centering
\begin{tabular}{|r|r|r|r|r|r|r|} \hline
$(m, n)$ & solver & \textsf{obj} & \textsf{pres} & \textsf{dres} & \textsf{dgap} & \textsf{time} \\ \hline
\multirow{4}{*}{$(500, 5000)$}
& ABIP  & 1.01e+04   & 1.10e-04  & 6.82e-04  & 9.56e-04  & 1.06e+01 \\
& SCS   & 1.01e+04   & 5.98e-04  & 4.28e-04  & 3.06e-04  & 1.42e+01 \\
& SDPT3 & 1.01e+04   & 1.22e-08  & 2.72e-05  & 2.20e-05  & 4.74e+01 \\
& MOSEK & 1.01e+04   & 1.70e-12  & 4.51e-14  & 1.15e-12  & 1.22e+00 \\ \hline
\multirow{4}{*}{$(500, 10000)$}
& ABIP  & -2.65e+04  & 6.88e-05  & 3.84e-04  & 9.60e-04  & 2.57e+01 \\
& SCS   & -2.65e+04  & 8.71e-04  & 4.37e-04  & 3.47e-04  & 2.21e+02 \\
& SDPT3 & -2.65e+04  & 1.83e-08  & 4.81e-05  & 3.99e-05  & 8.22e+01 \\
& MOSEK & -2.65e+04  & 1.03e-09  & 1.31e-09  & 2.07e-12  & 2.36e+00 \\ \hline
\multirow{4}{*}{$(500, 20000)$}
& ABIP  & 3.41e+04   & 7.99e-05  & 2.28e-04  & 9.75e-04  & 6.16e+01 \\
& SCS   & 3.41e+04   & 7.91e-04  & 3.24e-04  & 3.10e-04  & 1.19e+02 \\
& SDPT3 & 3.41e+04   & 1.15e-08  & 9.55e-05  & 5.62e-05  & 1.73e+02 \\
& MOSEK & 3.41e+04   & 4.55e-12  & 2.76e-14  & 6.15e-12  & 5.97e+00 \\ \hline
\multirow{4}{*}{$(1000, 5000)$}
& ABIP  & -1.35e+04  & 7.12e-05  & 4.28e-04  & 8.83e-04  & 3.99e+01 \\
& SCS   & -1.35e+04  & 9.57e-04  & 9.11e-04  & 6.36e-04  & 1.51e+02 \\
& SDPT3 & -1.35e+04  & 4.02e-09  & 7.22e-06  & 6.31e-06  & 2.62e+02 \\
& MOSEK & -1.35e+04  & 6.00e-10  & 4.97e-11  & 8.72e-10  & 3.90e+00 \\ \hline
\multirow{4}{*}{$(1000, 10000)$}
& ABIP  & 1.57e+05   & 9.27e-05  & 6.73e-04  & 9.73e-04  & 5.60e+01 \\
& SCS   & 1.57e+05   & 9.97e-04  & 7.49e-04  & 2.15e-05  & 2.49e+02 \\
& SDPT3 & 1.57e+05   & 6.34e-08  & 3.80e-05  & 3.15e-05  & 3.33e+02 \\
& MOSEK & 1.57e+05   & 5.01e-10  & 2.05e-13  & 1.03e-10  & 7.04e+00 \\ \hline
\multirow{4}{*}{$(1000, 20000)$}
& ABIP  & 1.47e+05   & 5.96e-05  & 4.56e-04  & 9.97e-04  & 1.22e+02 \\
& SCS   & 1.46e+05   & 9.68e-04  & 5.10e-04  & 3.36e-05  & 5.01e+02 \\
& SDPT3 & 1.46e+05   & 5.91e-08  & 1.10e-04  & 6.97e-05  & 8.34e+02 \\
& MOSEK & 1.46e+05   & 2.54e-11  & 1.23e-13  & 6.39e-12  & 1.40e+01 \\ \hline
\multirow{4}{*}{$(2000, 10000)$}
& ABIP  & -5.32e+04  & 3.76e-05  & 2.64e-04  & 9.85e-04  & 8.89e+02 \\
& SCS   & -5.32e+04  & 9.02e-04  & 8.35e-04  & 1.62e-04  & 9.94e+02 \\
& SDPT3 & -5.34e+04  & 1.03e-08  & 5.81e-06  & 5.69e-06  & 2.29e+03 \\
& MOSEK & -5.34e+04  & 5.52e-11  & 5.74e-11  & 7.22e-12  & 2.14e+01 \\ \hline
\multirow{4}{*}{$(2000, 20000)$}
& ABIP  & -7.40e+04  & 3.36e-05  & 2.15e-04  & 9.92e-04  & 4.04e+03 \\
& SCS   & -7.40e+04  & 9.18e-04  & 7.06e-04  & 5.26e-04  & 5.47e+03 \\
& SDPT3 & -7.42e+04  & 4.93e-08  & 1.79e-04  & 2.22e-04  & 3.37e+03 \\
& MOSEK & -7.42e+04  & 1.97e-10  & 7.70e-12  & 1.14e-10  & 3.62e+01 \\ \hline
\multirow{4}{*}{$(4000, 20000)$}
& ABIP  & -3.50e+05  & 2.88e-05  & 2.54e-04  & 1.00e-03  & 5.67e+03 \\
& SCS   & -3.50e+05  & 9.55e-04  & 8.84e-04  & 4.09e-04  & 6.72e+03 \\
& SDPT3 & -3.50e+05  & 1.09e-08  & 9.06e-06  & 9.75e-06  & 1.64e+04 \\
& MOSEK & -3.50e+05  & 5.04e-10  & 2.92e-10  & 1.95e-11  & 9.06e+01 \\ \hline
\end{tabular}\caption{Performance of four direct solvers on randomly generated LPs. CPU times are in seconds. }\label{Table:random-LP-direct}
\end{table}

\begin{table}[t]\small
\centering
\begin{tabular}{|r|r|r|r|r|r|r|} \hline
$(m, n)$ & solver & \textsf{obj} & \textsf{pres} & \textsf{dres} & \textsf{dgap} & \textsf{time} \\ \hline
\multirow{3}{*}{$(500, 5000)$}
& ABIP-CG & -9.32e+02 & 1.75e-05 & 2.01e-04 & 9.66e-04 & 1.34e+03 \\
& SCS-CG & -9.33e+02 & 6.00e-04 & 5.13e-04 & 9.22e-04 & 3.45e+03 \\
& DSDP-CG & -9.32e+02 & 1.03e-10 & --- & 2.45e-04 & 6.53e+01 \\ \hline
\multirow{3}{*}{$(500, 10000)$}
& ABIP-CG & 1.08e+04  & 4.70e-05 & 3.21e-04 & 9.95e-04 & 8.73e+02 \\
& SCS-CG & 1.07e+04  & 8.38e-04 & 4.95e-04 & 3.85e-04 & 2.43e+03 \\
& DSDP-CG & 1.07e+04  & 9.88e-11 & --- & 1.25e-04 & 7.60e+01 \\ \hline
\multirow{3}{*}{$(500, 20000)$}
& ABIP-CG & 1.93e+04  & 2.82e-05 & 1.22e-04 & 9.94e-04 & 2.44e+03 \\
& SCS-CG & 1.93e+04  & 8.66e-04 & 3.62e-04 & 9.99e-04 & 3.05e+03 \\
& DSDP-CG & 1.93e+04  & 2.13e-11 & --- & 5.32e-05 & 1.52e+02 \\ \hline
\multirow{3}{*}{$(1000, 5000)$}
& ABIP-CG & 7.73e+04  & 1.23e-04 & 7.69e-04 & 9.99e-04 & 8.53e+02 \\
& SCS-CG & 7.71e+04  & 9.00e-04 & 8.12e-04 & 1.65e-04 & 4.59e+03 \\
& DSDP-CG & 7.71e+04  & 1.24e-09 & --- & 4.45e-04 & 3.28e+02 \\ \hline
\multirow{3}{*}{$(1000, 10000)$}
& ABIP-CG & -5.41e+04 & 6.38e-05 & 2.83e-04 & 9.93e-04 & 2.80e+03 \\
& SCS-CG & -5.44e+04 & 9.97e-04 & 8.01e-04 & 4.20e-05 & 8.82e+03 \\
& DSDP-CG & -5.43e+04 & 2.51e-11 & --- & 3.68e-05 & 4.72e+02 \\ \hline
\multirow{3}{*}{$(1000, 20000)$}
& ABIP-CG & -1.70e+05 & 7.15e-05 & 3.32e-04 & 9.77e-04 & 4.53e+03 \\
& SCS-CG & -1.71e+05 & 8.78e-04 & 7.34e-04 & 1.10e-04 & 3.34e+04 \\
& DSDP-CG & -1.71e+05 & 6.39e-12 & --- & 2.17e-04 & 5.86e+02 \\ \hline
\multirow{3}{*}{$(2000, 10000)$}
& ABIP-CG & 1.16e+05  & 4.91e-05 & 3.70e-04 & 9.86e-04 & 9.47e+03 \\
& SCS-CG & 1.15e+05  & 9.38e-04 & 8.16e-04 & 8.01e-05 & 3.19e+04 \\
& DSDP-CG & 1.15e+05  & 1.19e-09 & --- & 4.43e-04 & 3.01e+03 \\ \hline
\end{tabular}\caption{Performance of three indirect solvers on randomly generated LPs. CPU times are in seconds. Note: DSDP-CG does not provide the dual residuals.}\label{Table:random-LP-indirect}
\end{table}

\begin{table}[t]
\centering
\begin{tabular}{|r|r|r|r|r|r|r|r|}\hline
\multirow{2}{*}{\textbf{Name}} & \multicolumn{7}{|c|}{\textbf{Problem Statistics}} \\ \cline{2-8}
& Samples ($N$) & Variables ($d$) & Threshold ($\lambda$) & Rows ($m$)& Cols ($n$) & Nonzeros & Sparsity \\ \hline
gisette & 13500 & 5000 & 0.0377 & 10000 & 20000 & 50015000 & 0.2501 \\
isolet & 7797  & 618  & 0.0431 & 1236  &  2472 & 765702   & 0.2506 \\
sEMG & 1800  & 3000 & 0.1000 & 6000  & 12000 & 18009000 & 0.2501 \\
sEMGday & 3600  & 2500 & 0.0699 & 5000  & 10000 & 12507500 & 0.2501 \\
ucihapt & 10929 & 561  & 0.0361 & 1122  &  2244 & 631125   & 0.2507 \\
ucihar & 10299 & 561  & 0.0372 & 1122  &  2244 & 631125   & 0.2507 \\ \hline
\end{tabular}\caption{Statistics of six instances in \textsf{UCI} collection.} \label{Table:LP-UCI}
\end{table}

\begin{table}[!t]
\centering
\begin{tabular}{|r||r||r|r|r|r||r|}\hline
\textbf{Name} & \textbf{Method} & \textsf{obj} & \textsf{pres} & \textsf{dres} & \textsf{dgap}  & \textsf{time} \\ \hline
\multirow{4}{*}{gisette}
&  ABIP & 1.25e-05 & 1.00e-03 & 1.16e-04 & 6.47e-06 & 1.03e+03 \\
&  SCS  & 1.25e-05 & 8.89e-04 & 9.54e-04 & 4.13e-08 & 2.31e+03 \\
&  SDPT3 & ---      & ---      & ---      & ---      & ---      \\
&  MOSEK & 1.25e-05 & 7.11e-13 & 4.63e-15 & 2.89e-13 & 7.46e+01 \\ \hline
\multirow{4}{*}{isolet}
&  ABIP  & 6.37e+02 & 1.56e-04 & 9.99e-04 & 2.77e-08 & 1.89e+02 \\
&  SCS   & 6.37e+02 & 7.34e-04 & 8.85e-04 & 1.28e-05 & 8.68e+02 \\
&  SDPT3 & 6.37e+02 & 1.67e-08 & 1.45e-06 & 4.62e-08 & 2.89e+01 \\
&  MOSEK & 6.37e+02 & 2.89e-08 & 1.92e-10 & 2.43e-10 & 6.95e-01 \\ \hline
\multirow{4}{*}{sEMG}
&  ABIP  & 1.06e+02 & 9.68e-04 & 6.21e-04 & 3.19e-04 & 2.61e+02 \\
&  SCS   & 1.07e+02 & 9.48e-04 & 8.61e-04 & 5.27e-05 & 2.85e+02 \\
&  SDPT3 & 1.07e+02 & 2.35e-08 & 5.38e-06 & 4.06e-06 & 3.52e+03 \\
&  MOSEK & 1.07e+02 & 3.05e-09 & 5.76e-15 & 9.95e-10 & 2.99e+01 \\ \hline
\multirow{4}{*}{sEMGday}
&  ABIP  & 1.89e+02 & 4.04e-04 & 9.93e-04 & 4.55e-05 & 1.67e+02 \\
&  SCS   & 1.89e+02 & 4.18e-04 & 9.99e-04 & 1.70e-05 & 2.43e+02 \\
&  SDPT3 & 1.89e+02 & 1.23e-08 & 2.43e-06 & 1.21e-06 & 2.04e+03 \\
&  MOSEK & 1.89e+02 & 1.85e-09 & 1.10e-13 & 1.03e-09 & 1.76e+01 \\ \hline
\multirow{4}{*}{ucihapt}
&  ABIP  & 3.64e+03 & 2.23e-05 & 1.00e-03 & 4.65e-07 & 8.44e+02 \\
&  SCS   & ---      & ---      & ---      & ---      & ---      \\
&  SDPT3 & 3.63e+03 & 7.73e-07 & 3.79e-07 & 4.25e-08 & 4.20e+01 \\
&  MOSEK & 3.63e+03 & 1.06e-10 & 1.56e-11 & 3.17e-14 & 7.46e-01 \\ \hline
\multirow{4}{*}{ucihar}
&  ABIP  & 2.05e+03 & 1.51e-04 & 1.00e-03 & 6.30e-07 & 1.42e+02 \\
&  SCS   & 2.05e+03 & 3.55e-04 & 9.99e-04 & 1.88e-06 & 1.00e+03 \\
&  SDPT3 & 2.05e+03 & 3.99e-08 & 1.67e-09 & 1.08e-13 & 4.65e+01 \\
&  MOSEK & 2.05e+03 & 2.16e-08 & 6.57e-10 & 5.77e-11 & 5.87e-01 \\ \hline
\end{tabular} \caption{Performance of four direct solvers on UCI. CPU times are in seconds.} \label{Table:LP-UCI-Direct}
\end{table}

\begin{table}[!t]
\centering
\begin{tabular}{|r||r||r|r|r|r||r|}\hline
\textbf{Name} & \textbf{Method} & \textsf{obj} & \textsf{pres} & \textsf{dres} & \textsf{dgap}  & \textsf{time} \\ \hline
\multirow{3}{*}{gisette}
&  ABIP-CG & 1.23e-05 & 9.87e-04 & 1.21e-04 & 4.48e-06 & 1.42e+04 \\
&  SCS-CG  & 1.23e-05 & 9.93e-04 & 6.35e-04 & 1.85e-07 & 6.11e+04 \\
&  DSDP-CG & ---      & ---      & ---      & ---      & ---      \\ \hline
\multirow{3}{*}{isolet}
&  ABIP-CG & 6.37e+02 & 1.55e-04 & 9.98e-04 & 3.78e-07 & 2.18e+03 \\
&  SCS-CG  & 6.37e+02 & 9.36e-04 & 9.52e-04 & 1.30e-06 & 1.20e+04 \\
&  DSDP-CG & 6.37e+02 & 3.70e-10 & ---      & 1.65e-04 & 4.24e+01 \\ \hline
\multirow{3}{*}{sEMG}
&  ABIP-CG & 1.06e+02 & 9.95e-04 & 6.66e-04 & 2.41e-04 & 1.80e+03 \\
&  SCS-CG  & 1.07e+02 & 9.56e-04 & 7.92e-04 & 1.07e-04 & 2.81e+03 \\
&  DSDP-CG & 1.07e+02 & 2.38e-12 & ---      & 3.18e-04 & 6.15e+03 \\ \hline
\multirow{3}{*}{sEMGday}
&  ABIP-CG & 1.89e+02 & 4.07e-04 & 9.98e-04 & 3.76e-05 & 1.13e+03 \\
&  SCS-CG  & 1.89e+02 & 4.72e-04 & 9.62e-04 & 3.49e-05 & 2.10e+03 \\
&  DSDP-CG & 1.89e+02 & 6.63e-11 & ---      & 2.76e-04 & 3.32e+03 \\ \hline
\multirow{3}{*}{ucihapt}
&  ABIP-CG & 3.64e+03 & 2.36e-05 & 1.00e-03 & 4.73e-07 & 1.12e+04 \\
&  SCS-CG  & ---      & ---      & ---      & ---      & ---      \\
&  DSDP-CG & 3.63e+03 & 1.58e-07 & ---      & 1.82e-04 & 3.57e+01 \\ \hline
\multirow{3}{*}{ucihar}
&  ABIP-CG & 2.05e+03 & 1.46e-04 & 1.00e-03 & 5.68e-07 & 2.24e+03 \\
&  SCS-CG  & ---      & ---      & ---      & ---      & ---      \\
&  DSDP-CG & 2.05e+03 & 4.42e-08 & ---      & 3.95e-04 & 3.27e+01 \\ \hline
\end{tabular} \caption{Performance of three indirect solvers on UCI. CPU times are in seconds.} \label{Table:LP-UCI-Indirect}
\end{table}

\begin{table}[t] \footnotesize
\centering
\begin{tabular}{|r|r|r|r|r||r||r|r|r|r|} \hline
\multirow{2}{*}{Problem} & \multicolumn{4}{|c||}{\textsf{Before Presolving}} & & \multicolumn{4}{|c|}{\textsf{After Presolving}} \\ \cline{2-5} \cline{7-10}
& \textsf{Rows} & \textsf{Cols} & \textsf{Nonzeros} & \textsf{Sparsity} & & \textsf{Rows} & \textsf{Cols} & \textsf{Nonzeros} & \textsf{Sparsity}\\ \hline
25FV47      & 821    & 1876   & 10705  & 0.00695     & & 798    & 1854   & 10580   & 0.00715 \\
80BAU3B     & 2262   & 12061  & 23264   & 0.00085    & & 5221   & 14502  & 28620   & 0.00038       \\
ADLITTLE    & 56     & 138    & 424     & 0.05487    & & 55     & 137    & 417     & 0.05534 \\
AFIRO       & 27     & 51     & 102     & 0.07407    & & 27     & 51     & 102     & 0.07407 \\
AGG         & 488    & 615    & 2862    & 0.00954    & & 488    & 615    & 2862    & 0.00954 \\
AGG2        & 516    & 758    & 4740    & 0.01212    & & 516    & 758    & 4740    & 0.01212 \\
AGG3        & 516    & 758    & 4756    & 0.01216    & & 516    & 758    & 4756    & 0.01216 \\
BANDM       & 305    & 472    & 2494    & 0.01732    & & 269    & 436    & 2137    & 0.01822 \\
BEACONFD    & 173    & 295    & 3408    & 0.06678    & & 148    & 270    & 3105    & 0.07770 \\
BLEND       & 74     & 114    & 522     & 0.06188    & & 74     & 114    & 522     & 0.06188 \\
BNL1        & 643    & 1586   & 5532    & 0.00542    & & 632    & 1576   & 5522    & 0.00554 \\
BNL2        & 2324   & 4486   & 14996   & 0.00144    & & 2268   & 4430   & 14914   & 0.00148 \\
BOEING1     & 351    & 726    & 3827    & 0.01502    & & 592    & 967    & 4309    & 0.00753 \\
BOEING2     & 166    & 305    & 1358    & 0.02682    & & 213    & 352    & 1478    & 0.01971 \\
BORE3D      & 233    & 334    & 1448    & 0.01861    & & 210    & 311    &  1346    & 0.02061 \\
BRANDY      & 220    & 303    & 2202    & 0.03303    & & 149    & 259    &  2015    & 0.05221 \\
CAPRI       & 271    & 482    & 1896    & 0.01452    & & 397    & 605    & 2137    & 0.00890 \\
CRE{\_}A    & 3516   & 7248   & 18168   & 0.00071    & & 3428   & 7248   & 18168   & 0.00073 \\
CRE{\_}B    & 9648   & 77137  & 260785  & 0.00035    & & 7240   & 77137  & 260785  & 0.00047 \\
CRE{\_}C    & 3068   & 6411   & 15977   & 0.00081    & & 2986   & 6411   & 15977   & 0.00083 \\
CRE{\_}D    & 8926   & 73948  & 246614  & 0.00037    & & 6476   & 73948  & 246614  & 0.00051 \\
CYCLE       & 1903   & 3371   & 21234   & 0.00331    & & 1878   & 3381   & 20958   & 0.00330 \\
CZPROB      & 929    & 3562   & 10708   & 0.00324    & & 737    & 3141   &  9454    & 0.00408 \\
D2Q06C      & 2171   & 5831   & 33081   & 0.00261    & & 2171   & 5831   & 33081   & 0.00261 \\
D6CUBE      & 415    & 6184   & 37704   & 0.01469    & & 404    & 6184   & 37704   & 0.01509 \\
DEGEN2      & 444    & 757    & 4201    & 0.01250    & & 444    & 757    &  4201    & 0.01250 \\
DEGEN3      & 1503   & 2604   & 25432   & 0.00650    & & 1503   & 2604   & 25432   & 0.00650 \\
DFL001      & 6071   & 12230  & 35632   & 0.00048    & & 6084   & 12243  & 35658   & 0.00048 \\
E226        & 223    & 472    & 2768    & 0.02630    & & 220    & 469    &  2737    & 0.02653 \\
ETAMACRO    & 400    & 816    & 2537    & 0.00777    & & 488    & 823    &  2306    & 0.00574 \\
FFFFF800    & 524    & 1028   & 6401    & 0.01188    & & 501    & 1005   &  6283    & 0.01248 \\
FINNIS      & 497    & 1064   & 2760    & 0.00522    & & 528    & 1050   &  2599    & 0.00469 \\
FIT1D       & 24     & 1049   & 13427   & 0.53333    & & 1050   & 2075   & 15479   & 0.00710 \\
FIT1P       & 627    & 1677   & 9868    & 0.00938    & & 1026   & 2076   & 10666   & 0.00501 \\
FIT2D       & 25     & 10524  & 129042  & 0.49047    & & 10525  & 21024  & 150042  & 0.00068 \\
FIT2P       & 3000   & 13525  & 50284   & 0.00124    & & 10500  & 21025  & 65284   & 0.00030 \\
FORPLAN     & 161    & 492    & 4634    & 0.05850    & & 157    & 485    &  4583    & 0.06019 \\
GANGES      & 1309   & 1706   & 6937    & 0.00311    & & 1534   & 1931   &  7387    & 0.00249 \\
GFRD{\_}PNC & 616    & 1160   & 2445    & 0.00342    & & 858    & 1402   &  2929    & 0.00243 \\
GREENBEA    & 2392   & 5598   & 31070   & 0.00232    & & 2608   & 5714   & 31014   & 0.00208 \\
GREENBEB    & 2392   & 5598   & 31070   & 0.00232    & & 2608   & 5706   & 30966   & 0.00208 \\
GROW7       & 140    & 301    & 2612    & 0.06198    & & 420    & 581    &  3172    & 0.01300 \\
GROW15      & 300    & 645    & 5620    & 0.02904    & & 900    & 1245   &  6820    & 0.00609 \\
GROW22      & 440    & 946    & 8252    & 0.01983    & & 1320   & 1826   & 10012   & 0.00415 \\
ISRAEL      & 174    & 316    & 2443    & 0.04443    & & 174    & 316    &  2443    & 0.04443 \\
KB2         & 43     & 68     & 313     & 0.10705    & & 52     & 77     & 331     & 0.08267 \\
KEN{\_}7    & 2426   & 3602   & 8404    & 0.00096    & & 4558   & 5734   & 12374   & 0.00047 \\
KEN{\_}11   & 14694  & 21349  & 49058   & 0.00016    & & 29751  & 36406  & 78567   & 0.00007 \\
KEN{\_}13   & 28632  & 42659  & 97246   & 0.00008    & & 60813  & 74840  & 159749  & 0.00004 \\
KEN{\_}18   & 105127 & 154699 & 358171  & 0.00002    & & 231314 & 280886 & 606657  & 0.00001 \\
LOTFI       & 153    & 366    & 1136    & 0.02029    & & 151    & 364    & 1123    & 0.02043 \\
MAROS       & 846    & 1966   & 10137   & 0.00609    & & 835    & 1921   & 10060   & 0.00627 \\
MAROS{\_}R7 &  3136  & 9408   & 144848  & 0.00491    & & 3136   & 9408   & 144848  & 0.00491 \\
MODSZK1     & 687    & 1620   & 3168    & 0.00285    & & 686    & 1622   &  3170    & 0.00285 \\
NESM        & 662    & 3105   & 13470   & 0.00655    & & 2250   & 4518   & 16436   & 0.00162 \\
OSA{\_}07   & 1118   & 25067  & 144812  & 0.00517    & & 1118   & 25067  & 144812  & 0.00517 \\
OSA{\_}14   & 2337   & 54797  & 317097  & 0.00248    & & 2337   & 54797  & 317097  & 0.00248 \\ \hline
\end{tabular}\caption{Statistics of feasible LPs in \textsf{NETLIB} before and after presolving: part 1.}\label{Table:LP-NETLIB-Presolve-1}
\end{table}

\begin{table}[t] \footnotesize
\centering
\begin{tabular}{|r|r|r|r|r||r||r|r|r|r|} \hline
\multirow{2}{*}{Problem} & \multicolumn{4}{|c||}{\textsf{Before Presolving}} & &  \multicolumn{4}{|c|}{\textsf{After Presolving}} \\ \cline{2-5} \cline{7-10}
& \textsf{Rows} & \textsf{Cols} & \textsf{Nonzeros} & \textsf{Sparsity} & & \textsf{Rows} & \textsf{Cols} & \textsf{Nonzeros} & \textsf{Sparsity}\\ \hline
OSA{\_}30   & 4350  & 104374 & 604488  & 0.00133    & & 4350  & 104374 & 604488  & 0.00133 \\
OSA{\_}60   & 10280 & 243246 & 1408073 & 0.00056    & & 10280 & 243246 & 1408073 & 0.00056 \\
PDS{\_}02   & 2953  & 7716   & 16571   & 0.00073    & & 4768  & 9531   & 20190   & 0.00044 \\
PDS{\_}06   & 9881  & 29351  & 63220   & 0.00022    & & 18604 & 38074  & 80556   & 0.00011 \\
PDS{\_}10   & 16558 & 49932  & 107605  & 0.00013    & & 32079 & 65453  & 138482  & 0.00007 \\
PDS{\_}20   & 33874 & 108175 & 232647  & 0.00006    & & 67599 & 141976 & 299853  & 0.00003 \\
PEROLD      & 625   & 1506   & 6148    & 0.00653    & & 891   & 1796   & 7663    & 0.00479 \\
PILOT       & 1441  & 4860   & 44375   & 0.00634    & & 2481  & 5697   & 44380   & 0.00314 \\
PILOT4      & 410   & 1123   & 5264    & 0.01143    & & 649   & 1420   & 7720    & 0.00838 \\
PILOT87     & 2030  & 6680   & 74949   & 0.00553    & & 3608  & 8038   & 75635   & 0.00261 \\
PILOT{\_}JA & 940   & 2267   & 14977   & 0.00703    & & 1263  & 2383   & 14017   & 0.00466 \\
PILOT{\_}WE & 722   & 2928   & 9265    & 0.00438    & & 1016  & 3224   & 10125   & 0.00309 \\
PILOTNOV    & 975   & 2446   & 13331   & 0.00559    & & 1291  & 2582   & 13140   & 0.00394 \\
QAP8        & 912   & 1632   & 7296    & 0.00490    & & 912   & 1632   & 7296    & 0.00490 \\
QAP12       & 3192  & 8856   & 38304   & 0.00136    & & 3192  & 8856   & 38304   & 0.00136 \\
QAP15       & 6330  & 22275  & 94950   & 0.00067    & & 6330  & 22275  & 94950   & 0.00067 \\
RECIPE      & 91    & 204    & 687     & 0.03701    & & 153   & 245    & 785     & 0.02094 \\
SC50A       & 50    & 78     & 160     & 0.04103    & & 49    & 77     & 159     & 0.04214 \\
SC50B       & 50    & 78     & 148     & 0.03795    & & 48    & 76     & 146     & 0.04002 \\
SC105       & 105   & 163    & 340     & 0.01987    & & 105   & 163    & 340     & 0.01987 \\
SC205       & 205   & 317    & 665     & 0.01023    & & 205   & 317    & 665     & 0.01023 \\
SCAGR7      & 129   & 185    & 465     & 0.01948    & & 129   & 185    & 465     & 0.01948 \\
SCAGR25     & 471   & 671    & 1725    & 0.00546    & & 471   & 671    & 1725    & 0.00546 \\
SCFXM1      & 330   & 600    & 2732    & 0.01380    & & 322   & 592    & 2707    & 0.01420 \\
SCFXM2      & 660   & 1200   & 5469    & 0.00691    & & 644   & 1184   & 5419    & 0.00711 \\
SCFXM3      & 990   & 1800   & 8206    & 0.00460    & & 966   & 1776   & 8131    & 0.00474 \\
SCORPION    & 388   & 466    & 1534    & 0.00848    & & 375   & 453   & 1460    & 0.00859 \\
SCRS8       & 490   & 1275   & 3288    & 0.00526    & & 485   & 1270  & 3262    & 0.00530 \\
SCSD1       & 77    & 760    & 2388    & 0.04081    & & 77    & 760   & 2388    & 0.04081 \\
SCSD6       & 147   & 1350   & 4316    & 0.02175    & & 147   & 1350  & 4316    & 0.02175 \\
SCSD8       & 397   & 2750   & 8584    & 0.00786    & & 397   & 2750  & 8584    & 0.00786 \\
SCTAP1      & 300   & 660    & 1872    & 0.00945    & & 300   & 660   & 1872    & 0.00945 \\
SCTAP2      & 1090  & 2500   & 7334    & 0.00269    & & 1090  & 2500  & 7334    & 0.00269 \\
SCTAP3      & 1480  & 3340   & 9734    & 0.00197    & & 1480  & 3340  & 9734    & 0.00197 \\
SEBA        & 515   & 1036   & 4360    & 0.00817    & & 1029  & 1550  & 5388    & 0.00338 \\
SHARE1B     & 117   & 253    & 1179    & 0.03983    & & 112   & 248   & 1148    & 0.04133 \\
SHARE2B     & 96    & 162    & 777     & 0.04996    & & 96    & 162   &  777     & 0.04996 \\
SHELL       & 536   & 1777   & 3558    & 0.00374    & & 613   & 1604  & 3212    & 0.00327 \\
SHIP04L     & 402   & 2166   & 6380    & 0.00733    & & 356   & 2162  & 6368    & 0.00827 \\
SHIP04S     & 402   & 1506   & 4400    & 0.00727    & & 268   & 1414  & 4124    & 0.01088 \\
SHIP08L     & 778   & 4363   & 12882   & 0.00380    & & 688   & 4339  & 12810   & 0.00429 \\
SHIP08S     & 778   & 2467   & 7194    & 0.00375    & & 416   & 2171  & 6306    & 0.00698 \\
SHIP12L     & 1151  & 5533   & 16276   & 0.00256    & & 838   & 5329  & 15664   & 0.00351 \\
SHIP12S     & 1151  & 2869   & 8284    & 0.00251    & & 466   & 2293  & 6556    & 0.00614 \\
SIERRA      & 1227  & 2735   & 8001    & 0.00238    & & 3238  & 4731  & 11983   & 0.00078 \\
STAIR       & 356   & 614    & 4003    & 0.01831    & & 362   & 544   & 3843    & 0.01951 \\
STANDATA    & 359   & 1274   & 3230    & 0.00706    & & 463   & 1362  & 3381    & 0.00536 \\
STANDGUB    & 361   & 1383   & 3338    & 0.00669    & & 464   & 1470  & 3489    & 0.00512 \\
STANDMPS    & 467   & 1274   & 3878    & 0.00652    & & 571   & 1362  & 4029    & 0.00518 \\
STOCFOR1    & 117   & 165    & 501     & 0.02595    & & 109   & 157   &  471     & 0.02752 \\
STOCFOR2    & 2157  & 3045   & 9357    & 0.00142    & & 2157  & 3045  & 9357    & 0.00142 \\
STOCFOR3    & 16675 & 23541  & 72721   & 0.00019    & & 16675 & 23541 & 72721   & 0.00019 \\
TRUSS       & 1000  & 8806   & 27836   & 0.00316    & & 1000  & 8806  & 27836   & 0.00316 \\
TUFF        & 333   & 628    & 4561    & 0.02181    & & 317   & 642   & 4599    & 0.02260 \\
VTP{\_}BASE & 198   & 346    & 1051    & 0.01534    & & 258   & 389   & 1065    & 0.01061 \\
WOOD1P      & 244   & 2595   & 70216   & 0.11089    & & 244   & 2595  & 70216   & 0.11089 \\
WOODW       & 1098  & 8418   & 37487   & 0.00406    & & 1098  & 8418  & 37487   & 0.00406 \\ \hline
\end{tabular}\caption{Statistics of feasible LPs in \textsf{NETLIB} before and after presolving: part 2.} \label{Table:LP-NETLIB-Presolve-2}
\end{table}

\begin{table}[!ht]
\centering
\begin{tabular}{|r|r|r|r|r|}\hline
\multirow{2}{*}{\textbf{Solver}} & \multicolumn{2}{|c|}{\textbf{Num of Instances}} & \multicolumn{2}{|c|}{\textbf{Time(s)}} \\ \cline{2-5}
& \textsf{Solved} & \textsf{Unsolved} &  \textsf{Mean} & \textsf{Std} \\ \hline
ABIP  & 101  & 13  & 12.15  & 46.10  \\
SCS   & 89   & 25  & 36.39  & 236.26 \\
SDPT3 & 104  & 10  & 27.73  & 151.95 \\
MOSEK & 114  & 0   & 0.27   & 1.19   \\ \hline
\end{tabular}\caption{Summary of the performance of all four direct solvers ($\epsilon = 10^{-3}$) on \textsf{NETLIB}. The CPU times (in seconds) are summarized for 78 instances that can be solved by all four direct solvers.} \label{Table:LP-NETLIB-Main-Direct-1}
\end{table}

\begin{table}[!ht]
\centering
\begin{tabular}{|r|r|r|r|r|}\hline
\multirow{2}{*}{\textbf{Solver}} & \multicolumn{2}{|c|}{\textbf{Num of Instances}} & \multicolumn{2}{|c|}{\textbf{Time(s)}} \\ \cline{2-5}
& \textsf{Solved} & \textsf{Unsolved} &  \textsf{Mean} & \textsf{Std} \\ \hline
ABIP  & 85   & 29  & 31.96   & 94.22   \\
SCS   & 63   & 51  & 256.41  & 1300.60 \\
SDPT3 & 81   & 33  & 19.77   & 82.53   \\
MOSEK & 110  & 4   & 0.34    & 1.56    \\ \hline
\end{tabular}\caption{Summary of the performance of all four direct solvers ($\epsilon = 10^{-5}$) on \textsf{NETLIB}. The CPU times (in seconds) are summarized for 44 instances that can be solved by all four direct solvers.} \label{Table:LP-NETLIB-Main-Direct-2}
\end{table}

\begin{table}[!ht]
\centering
\begin{tabular}{|r|r|r|r|r|}\hline
\multirow{2}{*}{\textbf{Solver}} & \multicolumn{2}{|c|}{\textbf{Num of Instances}} & \multicolumn{2}{|c|}{\textbf{Time(s)}} \\ \cline{2-5}
& \textsf{Solved} & \textsf{Unsolved} &  \textsf{Mean} & \textsf{Std} \\ \hline
ABIP-CG & 101  & 13  & 66.61   & 313.25 \\
SCS-CG  & 86   & 28  & 79.99   & 316.45 \\
DSDP-CG & 105  & 9   & 158.47  & 895.59 \\ \hline
\end{tabular}\caption{Summary of the performance of all three indirect solvers ($\epsilon = 10^{-3}$) on \textsf{NETLIB}. The CPU times (in seconds) are summarized for 74 instances that can be solved by all five solvers.} \label{Table:LP-NETLIB-Main-Indirect}
\end{table}

\begin{table}[t] \scriptsize
\centering
\begin{tabular}{|r|r|r|r|r|r|}\hline
\textbf{Problem} & \textsf{ABIP} & \textsf{SCS} & \textsf{SDPT3} & \textsf{MOSEK} \\\hline
25FV47      & 5.51e+00  & 1.89e+01   & 6.92e-01   & 5.54e-02 \\ \hline
80BAU3B     & 4.12e+00  & 4.46e+00   & 6.64e+00   & 1.67e-01 \\ \hline
ADLITTLE    & 1.16e-02  & 3.33e-02   & 9.03e-02   & 5.50e-03 \\ \hline
AFIRO       & 1.58e-03  & 1.07e-03   & 6.38e-02   & 1.72e-03 \\ \hline
AGG         & 4.62e-01  & ---        & 2.95e-01   & 7.40e-03 \\ \hline
AGG2        & 7.35e-01  & 2.24e+01   & 2.56e-01   & 1.90e-02 \\ \hline
AGG3        & 8.83e-01  & 1.54e+01   & 2.55e-01   & 1.39e-02 \\ \hline
BANDM       & 2.89e-01  & 2.61e+00   & 1.87e-01   & 1.13e-02 \\ \hline
BEACONFD    & 8.48e-01  & 2.37e+00   & 9.60e-02   & 6.97e-03 \\ \hline
BLEND       & 3.64e-02  & 3.38e-02   & 6.53e-02   & 4.58e-03 \\ \hline
BNL1        & 5.12e-01  & 5.64e+00   & 3.21e-01   & 2.82e-02 \\ \hline
BNL2        & 9.70e-01  & 5.70e+00   & 2.51e+00   & 1.28e-01 \\ \hline
BOEING1     & 1.53e-01  & 1.35e+01   & 3.28e-01   & 1.60e-02 \\ \hline
BOEING2     & 2.15e-01  & 9.21e-01   & 1.92e-01   & 7.99e-03 \\ \hline
BORE3D      & 3.43e+00  & 3.61e+01   & 1.30e-01   & 4.38e-03 \\ \hline
BRANDY      & 6.05e-01  & 2.97e+00   & 1.95e-01   & 9.30e-03 \\ \hline
CAPRI       & 6.45e-01  & ---        & 4.33e-01   & 9.50e-03 \\ \hline
CRE{\_}A    & 2.86e-01  & 1.59e+00   & 1.37e+00   & 5.87e-02 \\ \hline
CRE{\_}B    & 8.58e+00  & 3.45e+01   & 7.05e+01   & 5.84e-01 \\ \hline
CRE{\_}C    & 4.33e-01  & 1.57e+00   & 1.41e+00   & 5.71e-02 \\ \hline
CRE{\_}D    & 1.18e+01  & 3.57e+01   & 7.20e+01   & 4.76e-01 \\ \hline
CYCLE       & 1.44e+02  & 1.17e+02   & 3.36e+00   & 6.06e-02 \\ \hline
CZPROB      & 9.47e-01  & 3.17e+00   & 4.27e-01   & 2.64e-02 \\ \hline
D2Q06C      & 4.38e+02  & ---        & 2.76e+00   & 2.45e-01 \\ \hline
D6CUBE      & 3.58e+00  & 2.35e+00   & 5.82e-01   & 9.54e-02 \\ \hline
DEGEN2      & 5.51e-02  & 5.36e-02   & 1.77e-01   & 2.05e-02 \\ \hline
DEGEN3      & 5.08e-01  & 5.71e-01   & 9.36e-01   & 1.17e-01 \\ \hline
DFL001      & ---       & 1.31e+01   & 1.15e+02   & 1.45e+00 \\ \hline
E226        & 1.40e+00  & 4.68e+00   & 2.68e-01   & 8.27e-03 \\ \hline
ETAMACRO    & 5.15e-01  & 3.53e+00   & 3.71e-01   & 2.95e-02 \\ \hline
FFFFF800    & 8.25e+01  & ---        & 5.04e-01   & 2.12e-02 \\ \hline
FINNIS      & 3.01e-01  & 5.67e-01   & 2.87e-01   & 1.42e-02 \\ \hline
FIT1D       & 3.89e+01  & ---        & 2.84e-01   & 1.40e-02 \\ \hline
FIT1P       & 2.41e-01  & 3.13e+00   & 9.33e-01   & 1.02e-01 \\ \hline
FIT2D       & 1.50e+03  & ---        & 2.15e+00   & 1.12e-01 \\ \hline
FIT2P       & 2.00e+00  & 8.17e+00   & 2.56e+00   & 9.22e-02 \\ \hline
FORPLAN     & ---       & 3.75e+00   & 2.60e-01   & 1.33e-02 \\ \hline
GANGES      & 1.85e-01  & 2.96e-01   & 2.81e-01   & 1.84e-02 \\ \hline
GFRD{\_}PNC & 2.65e-02  & 5.87e-02   & 3.27e-01   & 1.82e-02 \\ \hline
GREENBEA    & 3.64e+01  & 2.01e+02   & 3.52e+00   & 2.09e-01 \\ \hline
GREENBEB    & 1.83e+00  & 1.56e+01   & 3.62e+00   & 8.69e-02 \\ \hline
GROW7       & 3.80e-02  & 3.48e-02   & 1.67e-01   & 7.23e-03 \\ \hline
GROW15      & 8.57e-02  & 5.53e-02   & 1.63e-01   & 1.23e-02 \\ \hline
GROW22      & 1.37e-01  & 1.22e-01   & 2.14e-01   & 1.76e-02 \\ \hline
ISRAEL      & 2.10e+00  & ---        & 2.00e-01   & 1.65e-02 \\ \hline
KB2         & 1.18e+00  & 8.31e+00   & 7.55e-02   & 4.50e-03 \\ \hline
KEN{\_}7    & 3.44e-01  & 3.71e-01   & 3.95e-01   & 2.28e-02 \\ \hline
KEN{\_}11   & 6.04e+00  & 5.76e+00   & ---        & 1.48e-01 \\ \hline
KEN{\_}13   & 1.64e+01  & 1.73e+01   & ---        & 4.62e-01 \\ \hline
KEN{\_}18   & 1.84e+02  & 1.23e+02   & ---        & 2.26e+00 \\ \hline
LOTFI       & 4.12e+00  & ---        & 2.83e-01   & 6.61e-03 \\ \hline
MAROS       & ---       & ---        & 8.13e-01   & 2.93e-02 \\ \hline
MAROS{\_}R7 & 4.65e+00  & 1.96e+00   & 3.91e+00   & 3.47e-01 \\ \hline
MODSZK1     & ---       & 2.08e-01   & 5.06e-01   & 5.50e-02 \\ \hline
NESM        & 3.90e-01  & 7.96e-01   & 6.00e-01   & 4.95e-02 \\ \hline
OSA{\_}07   & 3.32e+00  & 8.50e-01   & 1.48e+00   & 1.12e-01 \\ \hline
OSA{\_}14   & 1.17e+01  & 2.58e+01   & 4.22e+00   & 2.38e-01 \\ \hline
\end{tabular}\caption{CPU time (in seconds) of all four direct solvers ($\epsilon = 10^{-3}$) on NETLIB: part 1.} \label{Table:LP-NETLIB-Direct-1}
\end{table}

\begin{table}[t] \scriptsize
\centering
\begin{tabular}{|r|r|r|r|r|r|}\hline
\textbf{Problem} & \textsf{ABIP} & \textsf{SCS} & \textsf{SDPT3} & \textsf{MOSEK} \\\hline
OSA{\_}30   & 2.42e+02    & 4.86e+01    & 8.06e+00   & 5.62e-01 \\ \hline
OSA{\_}60   & 1.19e+02    & 2.77e+02    & ---        & 1.47e+00 \\ \hline
PDS{\_}02   & 8.92e-01    & 3.84e-01    & 4.41e+00   & 4.52e-02 \\ \hline
PDS{\_}06   & 7.27e+00    & 2.19e+00    & 1.20e+02   & 4.24e-01 \\ \hline
PDS{\_}10   & 1.99e+01    & 7.69e+00    & 5.33e+02   & 8.41e-01 \\ \hline
PDS{\_}20   & 8.66e+01    & 4.45e+01    & 1.24e+03   & 2.93e+00 \\ \hline
PEROLD      & ---         & ---         & ---        & 6.05e-02 \\ \hline
PILOT       & 4.21e+00    & ---         & 3.47e+00   & 4.20e-01 \\ \hline
PILOT4      & 8.77e+01    & ---         & ---        & 5.32e-02 \\ \hline
PILOT87     & 2.92e+02    & 2.08e+03    & 1.31e+01   & 9.16e-01 \\ \hline
PILOT{\_}JA & ---         & ---         & ---        & 1.11e-01 \\ \hline
PILOT{\_}WE & 2.34e+00    & 5.59e+01    & ---        & 8.27e-02 \\ \hline
PILOTNOV    & ---         & ---         & 1.68e+00   & 1.03e-01 \\ \hline
QAP8        & 4.13e-01    & 1.30e-01    & 4.16e-01   & 9.38e-02 \\ \hline
QAP12       & 6.26e+00    & 5.66e+00    & 9.17e+00   & 1.47e+00 \\ \hline
QAP15       & 3.21e+01    & 2.95e+01    & 4.68e+01   & 1.01e+01 \\ \hline
RECIPE      & 1.87e-02    & 5.63e-03    & 9.02e-02   & 4.34e-03 \\ \hline
SC50A       & 3.20e-03    & 1.71e-03    & 5.89e-02   & 1.76e-03 \\ \hline
SC50B       & 8.43e-03    & 1.88e-03    & 4.88e-02   & 1.73e-03 \\ \hline
SC105       & 4.04e-03    & 6.89e-03    & 7.28e-02   & 3.85e-03 \\ \hline
SC205       & 2.47e-02    & 2.57e-02    & 1.14e-01   & 3.99e-03 \\ \hline
SCAGR7      & 1.89e-02    & 1.22e-02    & 1.23e-01   & 6.21e-03 \\ \hline
SCAGR25     & 5.70e-02    & 6.64e-02    & 2.58e-01   & 3.27e-02 \\ \hline
SCFXM1      & 1.80e+00    & ---         & 3.21e-01   & 1.15e-02 \\ \hline
SCFXM2      & 3.32e+00    & ---         & 3.84e-01   & 2.20e-02 \\ \hline
SCFXM3      & 4.32e+00    & ---         & 4.57e-01   & 3.21e-02 \\ \hline
SCORPION    & 3.88e-02    & 3.24e-02    & 1.31e-01   & 7.10e-03 \\ \hline
SCRS8       & 1.44e-01    & 7.68e-01    & 3.78e-01   & 1.41e-02 \\ \hline
SCSD1       & 2.90e-02    & 1.64e-02    & 6.34e-02   & 5.52e-03 \\ \hline
SCSD6       & 4.47e-02    & 6.39e-02    & 8.24e-02   & 1.14e-02 \\ \hline
SCSD8       & 7.86e-02    & 1.17e-01    & 1.38e-01   & 1.34e-02 \\ \hline
SCTAP1      & 1.81e-01    & 5.17e-01    & 1.64e-01   & 8.48e-03 \\ \hline
SCTAP2      & 1.02e+00    & 3.65e-01    & 2.32e-01   & 1.80e-02 \\ \hline
SCTAP3      & 9.71e-01    & 8.21e-01    & 3.46e-01   & 2.43e-02 \\ \hline
SEBA        & 2.91e+00    & ---         & 5.56e-01   & 1.71e-02 \\ \hline
SHARE1B     & 3.27e+00    & ---         & 1.45e-01   & 8.70e-03 \\ \hline
SHARE2B     & 1.67e+00    & ---         & 7.52e-02   & 5.70e-03 \\ \hline
SHELL       & 7.82e-02    & 8.08e-02    & 3.09e-01   & 1.74e-02 \\ \hline
SHIP04L     & 8.16e-02    & 8.88e-02    & 1.99e-01   & 1.37e-02 \\ \hline
SHIP04S     & 7.66e-02    & 9.15e-02    & 1.68e-01   & 1.31e-02 \\ \hline
SHIP08L     & 1.98e-01    & 2.34e-01    & 2.57e-01   & 1.94e-02 \\ \hline
SHIP08S     & 1.11e-01    & 1.27e-01    & 2.24e-01   & 1.45e-02 \\ \hline
SHIP12L     & 4.34e-01    & 1.07e+00    & 2.86e-01   & 2.47e-02 \\ \hline
SHIP12S     & 1.63e-01    & 2.86e-01    & 2.60e-01   & 1.73e-02 \\ \hline
SIERRA      & 1.25e+02    & ---         & 1.88e+00   & 2.84e-02 \\ \hline
STAIR       & 1.79e-01    & 2.31e+00    & 2.86e-01   & 1.78e-02 \\ \hline
STANDATA    & 2.55e-01    & 1.41e-01    & 2.95e-01   & 9.95e-03 \\ \hline
STANDGUB    & 3.65e-01    & 2.41e-01    & 3.36e-01   & 9.76e-03 \\ \hline
STANDMPS    & 3.88e-01    & 2.02e-01    & 3.17e-01   & 1.28e-02 \\ \hline
STOCFOR1    & ---         & 9.82e-03    & 8.98e-02   & 2.76e-03 \\ \hline
STOCFOR2    & ---         & 2.47e+00    & 3.44e-01   & 3.31e-02 \\ \hline
STOCFOR3    & ---         & ---         & 2.92e+00   & 2.32e-01 \\ \hline
TRUSS       & 2.19e+00    & 7.54e+00    & 4.43e-01   & 8.81e-02 \\ \hline
TUFF        & 6.15e-02    & ---         & ---        & 1.29e-02 \\ \hline
VTP{\_}BASE & ---         & ---         & ---        & 5.56e-03 \\ \hline
WOOD1P      & ---         & 7.79e-01    & 6.95e-01   & 4.19e-02 \\ \hline
WOODW       & ---         & ---         & 9.28e-01   & 5.35e-02 \\ \hline
\end{tabular} \caption{CPU time (in seconds) of all four direct solvers ($\epsilon = 10^{-3}$) on NETLIB: part 2.} \label{Table:LP-NETLIB-Direct-2}
\end{table}

\begin{table}[t] \scriptsize
\centering
\begin{tabular}{|r|r|r|r|r|r|}\hline
\textbf{Problem} & \textsf{ABIP} & \textsf{SCS} & \textsf{SDPT3} & \textsf{MOSEK} \\\hline
25FV47      & 2.19e+02  & ---        & 1.46e+00   & 8.56e-02 \\ \hline
80BAU3B     & 5.84e+02  & 3.19e+02   & 6.74e+00   & 1.62e-01 \\ \hline
ADLITTLE    & 5.63e-02  & 6.93e-01   & 2.62e-01   & 4.10e-03 \\ \hline
AFIRO       & 2.98e-03  & 1.09e-03   & 1.26e-01   & 1.62e-03 \\ \hline
AGG         & ---       & ---        & 3.52e-01   & 8.35e-03 \\ \hline
AGG2        & 1.01e+02  & ---        & ---        & 1.42e-02 \\ \hline
AGG3        & 1.00e+02  & ---        & ---        & ---      \\ \hline
BANDM       & 2.61e+01  & ---        & 2.21e-01   & 9.45e-03 \\ \hline
BEACONFD    & ---       & ---        & 1.14e-01   & 4.22e-03 \\ \hline
BLEND       & 2.50e-01  & 4.84e-02   & 8.76e-02   & 3.86e-03 \\ \hline
BNL1        & 4.17e+01  & ---        & 3.36e-01   & 2.93e-02 \\ \hline
BNL2        & 1.59e+02  & ---        & 2.59e+00   & 1.29e-01 \\ \hline
BOEING1     & 2.05e+01  & ---        & ---        & 1.52e-02 \\ \hline
BOEING2     & 6.64e+00  & ---        & 2.11e-01   & 9.14e-03 \\ \hline
BORE3D      & ---       & ---        & 1.60e-01   & 4.57e-03 \\ \hline
BRANDY      & 7.94e+00  & 5.33e+00   & 2.31e-01   & 1.11e-02 \\ \hline
CAPRI       & 1.00e+01  & ---        & ---        & 9.36e-03 \\ \hline
CRE{\_}A    & 7.11e+00  & 8.00e+01   & 1.39e+00   & 6.34e-02 \\ \hline
CRE{\_}B    & 1.63e+02  & 1.35e+03   & 7.04e+01   & 5.97e-01 \\ \hline
CRE{\_}C    & 1.26e+01  & 3.25e+01   & 1.37e+00   & 6.07e-02 \\ \hline
CRE{\_}D    & 1.34e+02  & 8.57e+03   & 7.03e+01   & 5.00e-01 \\ \hline
CYCLE       & ---       & ---        & 3.39e+00   & 6.12e-02 \\ \hline
CZPROB      & 1.81e+00  & 1.21e+02   & 3.81e-01   & 2.52e-02 \\ \hline
D2Q06C      & ---       & ---        & 2.59e+00   & 2.37e-01 \\ \hline
D6CUBE      & 1.63e+01  & ---        & 5.77e-01   & 1.08e-01 \\ \hline
DEGEN2      & 5.78e-01  & 1.97e-01   & 1.86e-01   & 2.53e-02 \\ \hline
DEGEN3      & 4.89e+00  & 7.24e+00   & 9.26e-01   & 1.19e-01 \\ \hline
DFL001      & ---       & 5.90e+01   & ---        & 1.46e+00 \\ \hline
E226        & 1.51e+01  & ---        & 2.85e-01   & 9.62e-03 \\ \hline
ETAMACRO    & 1.97e+01  & ---        & 3.64e-01   & 2.64e-02 \\ \hline
FFFFF800    & ---       & ---        & 4.98e-01   & 2.16e-02 \\ \hline
FINNIS      & 2.26e+01  & ---        & 2.83e-01   & 1.36e-02 \\ \hline
FIT1D       & ---       & ---        & 2.89e-01   & 1.74e-02 \\ \hline
FIT1P       & 2.88e+00  & 5.22e+00   & ---        & 1.01e-01 \\ \hline
FIT2D       & ---       & ---        & 2.11e+00   & 1.23e-01 \\ \hline
FIT2P       & 3.22e+01  & 1.30e+02   & ---        & 9.46e-02 \\ \hline
FORPLAN     & ---       & 7.17e+01   & ---        & 1.26e-02 \\ \hline
GANGES      & 3.01e+00  & 1.60e+00   & 2.61e-01   & 1.79e-02 \\ \hline
GFRD{\_}PNC & 3.62e+00  & ---        & ---        & 1.51e-02 \\ \hline
GREENBEA    & 4.11e+02  & ---        & ---        & 2.14e-01 \\ \hline
GREENBEB    & ---       & ---        & ---        & 8.82e-02 \\ \hline
GROW7       & 4.94e-01  & 6.21e-02   & 1.69e-01   & 6.94e-03 \\ \hline
GROW15      & 5.21e-01  & 1.26e-01   & 1.57e-01   & 1.28e-02 \\ \hline
GROW22      & 4.53e-01  & 2.00e-01   & 2.00e-01   & ---      \\ \hline
ISRAEL      & 1.29e+01  & ---        & 2.05e-01   & 1.35e-02 \\ \hline
KB2         & ---       & ---        & 8.14e-02   & 2.88e-03 \\ \hline
KEN{\_}7    & 2.59e+00  & 1.38e+01   & ---        & 2.24e-02 \\ \hline
KEN{\_}11   & 5.13e+01  & 3.63e+02   & ---        & 1.51e-01 \\ \hline
KEN{\_}13   & 2.90e+02  & 1.25e+03   & ---        & 4.55e-01 \\ \hline
KEN{\_}18   & 3.60e+03  & 6.35e+03   & ---        & 2.28e+00 \\ \hline
LOTFI       & 1.90e+01  & ---        & ---        & 4.08e-03 \\ \hline
MAROS       & ---       & ---        & 7.89e-01   & 2.89e-02 \\ \hline
MAROS{\_}R7 & 4.04e+01  & ---        & ---        & 3.22e-01 \\ \hline
MODSZK1     & ---       & 4.71e+01   & 4.74e-01   & ---      \\ \hline
NESM        & 5.58e+00  & 1.60e+02   & 5.75e-01   & 4.88e-02 \\ \hline
OSA{\_}07   & 5.71e+01  & 1.25e+02   & ---        & 1.11e-01 \\ \hline
OSA{\_}14   & 7.51e+01  & 2.66e+02   & ---        & 2.35e-01 \\ \hline
\end{tabular}\caption{CPU time (in seconds) of all four direct solvers ($\epsilon = 10^{-5}$) on NETLIB: part 1.} \label{Table:LP-NETLIB-Direct-3}
\end{table}

\begin{table}[t] \scriptsize
\centering
\begin{tabular}{|r|r|r|r|r|r|}\hline
\textbf{Problem} & \textsf{ABIP} & \textsf{SCS} & \textsf{SDPT3} & \textsf{MOSEK} \\\hline
OSA{\_}30   & 2.27e+02    & 6.02e+02    & ---        & 5.49e-01 \\ \hline
OSA{\_}60   & 8.68e+02    & 1.20e+04    & ---        & 1.45e+00 \\ \hline
PDS{\_}02   & 7.49e+00    & 1.30e+00    & 4.42e+00   & 4.33e-02 \\ \hline
PDS{\_}06   & 5.71e+01    & 1.85e+01    & 1.20e+02   & 4.01e-01 \\ \hline
PDS{\_}10   & 1.58e+02    & 2.17e+01    & 5.32e+02   & 8.49e-01 \\ \hline
PDS{\_}20   & 5.77e+02    & 6.24e+02    & ---        & 3.03e+00 \\ \hline
PEROLD      & ---         & ---         & ---        & 5.18e-02 \\ \hline
PILOT       & 6.51e+02    & ---         & 3.45e+00   & 4.10e-01 \\ \hline
PILOT4      & ---         & ---         & ---        & 4.67e-02 \\ \hline
PILOT87     & ---         & ---         & 1.30e+01   & 9.46e-01 \\ \hline
PILOT{\_}JA & ---         & ---         & ---        & 1.13e-01 \\ \hline
PILOT{\_}WE & 1.15e+02    & ---         & ---        & ---      \\ \hline
PILOTNOV    & ---         & ---         & 1.59e+00   & 1.08e-01 \\ \hline
QAP8        & 1.39e+00    & 1.87e-01    & ---        & 9.47e-02 \\ \hline
QAP12       & 2.64e+01    & 1.20e+02    & 8.93e+00   & 1.47e+00 \\ \hline
QAP15       & 8.33e+01    & 3.70e+02    & 4.59e+01   & 1.03e+01 \\ \hline
RECIPE      & 1.07e+00    & 6.88e-03    & ---        & 2.99e-03 \\ \hline
SC50A       & 1.86e-02    & 2.34e-03    & 6.37e-02   & 1.75e-03 \\ \hline
SC50B       & 3.21e-02    & 2.91e-03    & 6.03e-02   & 1.28e-03 \\ \hline
SC105       & 3.73e-02    & 7.72e-03    & 7.71e-02   & 4.42e-03 \\ \hline
SC205       & 1.00e-01    & 4.14e-02    & 1.06e-01   & 4.52e-03 \\ \hline
SCAGR7      & 4.74e-02    & 9.24e-02    & 1.06e-01   & 6.45e-03 \\ \hline
SCAGR25     & 4.11e-01    & 1.49e+00    & 2.16e-01   & 2.69e-02 \\ \hline
SCFXM1      & 1.69e+01    & ---         & 3.15e-01   & 1.19e-02 \\ \hline
SCFXM2      & 2.02e+01    & ---         & 4.93e-01   & 2.21e-02 \\ \hline
SCFXM3      & 4.09e+01    & ---         & 4.43e-01   & 3.14e-02 \\ \hline
SCORPION    & 1.10e+01    & 7.91e-02    & 1.34e-01   & 5.84e-03 \\ \hline
SCRS8       & ---         & ---         & 3.73e-01   & 1.42e-02 \\ \hline
SCSD1       & 5.17e-02    & 1.54e-02    & 6.80e-02   & 4.73e-03 \\ \hline
SCSD6       & 1.06e-01    & 7.37e-02    & 8.53e-02   & 1.10e-02 \\ \hline
SCSD8       & 1.00e+00    & 1.34e-01    & ---        & 1.71e-02 \\ \hline
SCTAP1      & 3.95e+00    & 2.34e+01    & 1.68e-01   & 8.55e-03 \\ \hline
SCTAP2      & 5.22e+01    & 6.81e+00    & 2.14e-01   & 1.99e-02 \\ \hline
SCTAP3      & 1.32e+01    & 1.08e+00    & 3.19e-01   & 2.53e-02 \\ \hline
SEBA        & 5.25e+01    & ---         & 5.30e-01   & 1.49e-02 \\ \hline
SHARE1B     & ---         & ---         & 1.44e-01   & 7.44e-03 \\ \hline
SHARE2B     & ---         & ---         & 8.04e-02   & 4.55e-03 \\ \hline
SHELL       & 8.33e-01    & 5.78e-01    & 2.91e-01   & 1.63e-02 \\ \hline
SHIP04L     & 9.41e+00    & 1.39e-01    & 1.93e-01   & 1.34e-02 \\ \hline
SHIP04S     & 2.33e-01    & 6.65e-01    & 1.54e-01   & 1.25e-02 \\ \hline
SHIP08L     & 5.75e-01    & 1.16e+00    & 2.30e-01   & 1.69e-02 \\ \hline
SHIP08S     & 3.60e+00    & 3.57e-01    & 2.11e-01   & 1.07e-02 \\ \hline
SHIP12L     & 4.20e+00    & 1.06e+01    & 2.70e-01   & 2.51e-02 \\ \hline
SHIP12S     & 1.21e+00    & 1.38e+00    & 2.47e-01   & 1.28e-02 \\ \hline
SIERRA      & ---         & ---         & ---        & 2.74e-02 \\ \hline
STAIR       & 1.29e+01    & ---         & 2.54e-01   & 1.74e-02 \\ \hline
STANDATA    & 2.70e+00    & 2.51e-01    & 2.78e-01   & 9.22e-03 \\ \hline
STANDGUB    & 2.60e+00    & 2.75e-01    & 2.87e-01   & 8.69e-03 \\ \hline
STANDMPS    & 2.48e+00    & 1.99e-01    & 2.62e-01   & 1.24e-02 \\ \hline
STOCFOR1    & ---         & ---         & ---        & 3.51e-03 \\ \hline
STOCFOR2    & ---         & ---         & 3.08e-01   & 2.99e-02 \\ \hline
STOCFOR3    & ---         & ---         & 2.86e+00   & 2.29e-01 \\ \hline
TRUSS       & 3.90e+01    & 5.04e+01    & 4.30e-01   & 8.60e-02 \\ \hline
TUFF        & 5.48e+00    & ---         & ---        & 1.17e-02 \\ \hline
VTP{\_}BASE & ---         & ---         & ---        & 4.69e-03 \\ \hline
WOOD1P      & ---         & 2.90e+02    & 6.85e-01   & 4.32e-02 \\ \hline
WOODW       & ---         & ---         & 8.57e-01   & 5.21e-02 \\ \hline
\end{tabular} \caption{CPU time (in seconds) of all four direct solvers ($\epsilon = 10^{-5}$) on NETLIB: part 2.} \label{Table:LP-NETLIB-Direct-4}
\end{table}

\begin{table}[t] \scriptsize
\centering
\begin{tabular}{|r|r|r|r|}\hline
\textbf{Problem} & \textsf{ABIP-CG} & \textsf{SCS-CG} & \textsf{DSDP-CG} \\\hline
25FV47          & 2.73e+01 & 1.13e+02 & 1.01e+00 \\ \hline
80BAU3B         & 1.27e+01 & 5.47e+00 & 3.21e+01 \\ \hline
ADLITTLE        & 3.43e-02 & 8.65e-02 & 1.28e-02 \\ \hline
AFIRO           & 4.17e-03 & 1.43e-03 & 2.42e-03 \\ \hline
AGG             & 1.85e+00 & ---      & 1.69e-01 \\ \hline
AGG2            & 2.49e+00 & 1.48e+02 & 2.24e-01 \\ \hline
AGG3            & 2.03e+00 & 2.88e+01 & 2.36e-01 \\ \hline
BANDM           & 2.33e+00 & 1.67e+01 & 7.53e-02 \\ \hline
BEACONFD        & 5.06e+00 & 1.33e+01 & 5.46e-02 \\ \hline
BLEND           & 2.51e-01 & 1.48e-01 & 1.18e-02 \\ \hline
BNL1            & 3.60e+00 & 5.02e+01 & 5.01e-01 \\ \hline
BNL2            & 6.31e+00 & 3.76e+01 & 6.22e+00 \\ \hline
BOEING1         & 7.22e-01 & 6.45e+01 & 3.25e-01 \\ \hline
BOEING2         & 9.02e-01 & 8.02e-01 & 8.74e-02 \\ \hline
BORE3D          & 5.17e+01 & 2.31e+02 & 4.62e-02 \\ \hline
BRANDY          & 1.06e+01 & 5.95e+01 & ---      \\ \hline
CAPRI           & 7.20e+00 & ---      & 1.31e-01 \\ \hline
CRE{\_}A        & 1.33e+00 & 5.01e+00 & 1.18e+01 \\ \hline
CRE{\_}B        & 4.94e+01 & 2.50e+02 & ---      \\ \hline
CRE{\_}C        & 2.31e+00 & 6.09e+00 & 1.13e+01 \\ \hline
CRE{\_}D        & 7.20e+01 & 1.63e+02 & ---      \\ \hline
CYCLE           & 2.18e+03 & 2.30e+03 & 5.01e+00 \\ \hline
CZPROB          & 3.37e+00 & 7.70e+00 & 1.09e+00 \\ \hline
D2Q06C          & 4.79e+03 & ---      & 1.09e+01 \\ \hline
D6CUBE          & 2.35e+01 & 1.42e+01 & 3.01e+00 \\ \hline
DEGEN2          & 2.01e-01 & 2.19e-01 & 1.56e-01 \\ \hline
DEGEN3          & 3.58e+00 & 4.26e+00 & 3.26e+00 \\ \hline
DFL001          & ---      & 3.54e+01 & 8.35e+01 \\ \hline
E226            & 9.34e+00 & 9.64e+01 & 1.21e-01 \\ \hline
ETAMACRO        & 1.87e+00 & 1.64e+01 & 2.01e-01 \\ \hline
FFFFF800        & 5.65e+02 & ---      & 4.29e-01 \\ \hline
FINNIS          & 7.00e-01 & 1.18e+00 & 2.14e-01 \\ \hline
FIT1D           & 1.29e+02 & ---      & 1.35e+00 \\ \hline
FIT1P           & 5.82e-01 & 2.86e+00 & 1.61e+00 \\ \hline
FIT2D           & 6.94e+03 & ---      & 1.42e+02 \\ \hline
FIT2P           & 4.81e+00 & 1.95e+01 & 1.86e+03 \\ \hline
FORPLAN         & ---      & 1.17e+01 & 1.01e-01 \\ \hline
GANGES          & 2.20e+00 & 4.53e+00 & 8.97e-01 \\ \hline
GFRD{\_}PNC     & 2.31e-01 & 4.38e-01 & 3.10e-01 \\ \hline
GREENBEA        & 3.55e+02 & 1.44e+03 & 1.25e+01 \\ \hline
GREENBEB        & 2.15e+01 & 1.82e+02 & 1.27e+01 \\ \hline
GROW7           & 1.60e-01 & 1.06e-01 & 7.32e-02 \\ \hline
GROW15          & 4.02e-01 & 2.01e-01 & 2.95e-01 \\ \hline
GROW22          & 6.03e-01 & 2.94e-01 & ---      \\ \hline
ISRAEL          & 5.21e+00 & ---      & 8.51e-02 \\ \hline
KB2             & 3.58e+00 & 2.89e+01 & 6.18e-03 \\ \hline
KEN{\_}7        & 3.58e+00 & 1.38e+00 & 5.70e+00 \\ \hline
KEN{\_}11       & 8.73e+01 & 1.88e+01 & ---      \\ \hline
KEN{\_}13       & 2.37e+02 & 7.82e+01 & ---      \\ \hline
KEN{\_}18       & 4.21e+03 & 4.32e+02 & ---      \\ \hline
LOTFI           & 1.21e+01 & ---      & 3.91e-02 \\ \hline
MAROS           & ---      & ---      & 1.14e+00 \\ \hline
MAROS{\_}R7     & 2.71e+00 & 1.28e+00 & 1.91e+01 \\ \hline
MODSZK1         & ---      & 5.58e-01 & 3.85e-01 \\ \hline
NESM            & 3.23e+00 & 4.97e+00 & 3.99e+00 \\ \hline
OSA{\_}07       & 3.40e+01 & 3.23e+00 & 2.43e+01 \\ \hline
OSA{\_}14       & 1.21e+02 & 8.80e+01 & 1.35e+02 \\ \hline
\end{tabular}\caption{CPU time (in seconds) of all three indirect solvers ($\epsilon = 10^{-3}$) on NETLIB: part 1.} \label{Table:LP-NETLIB-Indirect-1}
\end{table}

\begin{table}[t] \scriptsize
\centering
\begin{tabular}{|r|r|r|r|r|r|}\hline
\textbf{Problem} & \textsf{ABIP-CG} & \textsf{SCS-CG} & \textsf{DSDP-CG} \\\hline
OSA{\_}30       & 1.58e+03 & 1.36e+02 & 5.47e+02 \\ \hline
OSA{\_}60       & 1.36e+03 & 4.31e+02 & ---      \\ \hline
PDS{\_}02       & 4.40e+00 & 1.08e+00 & 1.35e+01 \\ \hline
PDS{\_}06       & 3.15e+01 & 5.61e+00 & 2.90e+02 \\ \hline
PDS{\_}10       & 7.21e+01 & 1.53e+01 & 1.13e+03 \\ \hline
PDS{\_}20       & 2.55e+02 & 1.18e+02 & 7.43e+03 \\ \hline
PEROLD          & ---      & ---      & 1.06e+00 \\ \hline
PILOT           & 3.27e+01 & ---      & 1.67e+01 \\ \hline
PILOT4          & 4.49e+02 & ---      & 6.65e-01 \\ \hline
PILOT87         & 1.21e+03 & ---      & 3.92e+01 \\ \hline
PILOT{\_}JA     & ---      & ---      & 2.88e+00 \\ \hline
PILOT{\_}WE     & 3.23e+01 & 4.74e+02 & 1.72e+00 \\ \hline
PILOTNOV        & ---      & ---      & 2.82e+00 \\ \hline
QAP8            & 1.63e-01 & 4.10e-02 & 4.52e-01 \\ \hline
QAP12           & 6.09e-01 & 1.19e+00 & 1.66e+01 \\ \hline
QAP15           & 1.97e+00 & 3.39e+00 & 1.27e+02 \\ \hline
RECIPE          & 6.89e-02 & 1.16e-02 & 1.89e-02 \\ \hline
SC50A           & 1.48e-02 & 3.01e-03 & 4.60e-03 \\ \hline
SC50B           & 4.39e-02 & 2.81e-03 & 4.60e-03 \\ \hline
SC105           & 2.17e-02 & 1.82e-02 & 1.77e-02 \\ \hline
SC205           & 2.25e-01 & 1.49e-01 & 2.97e-02 \\ \hline
SCAGR7          & 6.17e-02 & 3.25e-02 & 1.61e-02 \\ \hline
SCAGR25         & 3.56e-01 & 2.49e-01 & 8.67e-02 \\ \hline
SCFXM1          & 2.64e+01 & ---      & 1.19e-01 \\ \hline
SCFXM2          & 4.73e+01 & ---      & 4.90e-01 \\ \hline
SCFXM3          & 6.52e+01 & ---      & 8.91e-01 \\ \hline
SCORPION        & 3.85e-01 & 2.29e-01 & 7.42e-02 \\ \hline
SCRS8           & 1.17e+00 & 8.02e+00 & 2.38e-01 \\ \hline
SCSD1           & 1.12e-01 & 3.59e-02 & 1.87e-02 \\ \hline
SCSD6           & 2.77e-01 & 3.68e-01 & 5.96e-02 \\ \hline
SCSD8           & 9.59e-01 & 1.48e+00 & 1.55e-01 \\ \hline
SCTAP1          & 6.77e-01 & 3.57e+00 & 6.83e-02 \\ \hline
SCTAP2          & 3.69e+00 & 2.14e+00 & 8.08e-01 \\ \hline
SCTAP3          & 3.60e+00 & 2.58e+00 & 1.55e+00 \\ \hline
SEBA            & 1.45e+01 & ---      & 1.16e+00 \\ \hline
SHARE1B         & 2.20e+01 & ---      & 3.45e-02 \\ \hline
SHARE2B         & 8.54e+00 & ---      & 1.67e-02 \\ \hline
SHELL           & 4.35e-01 & 1.72e-01 & 2.80e-01 \\ \hline
SHIP04L         & 5.14e-01 & 3.35e-01 & 2.92e-01 \\ \hline
SHIP04S         & 5.02e-01 & 2.34e-01 & 1.43e-01 \\ \hline
SHIP08L         & 1.45e+00 & 9.66e-01 & 1.08e+00 \\ \hline
SHIP08S         & 6.14e-01 & 8.80e-01 & 3.22e-01 \\ \hline
SHIP12L         & 2.83e+00 & 3.47e+00 & 1.33e+00 \\ \hline
SHIP12S         & 9.11e-01 & 9.49e-01 & 2.99e-01 \\ \hline
SIERRA          & 1.48e+03 & ---      & 4.09e+00 \\ \hline
STAIR           & 9.46e-01 & 1.45e+01 & 1.43e-01 \\ \hline
STANDATA        & 1.80e+00 & 8.71e-01 & 2.52e-01 \\ \hline
STANDGUB        & 2.30e+00 & 7.72e-01 & 2.69e-01 \\ \hline
STANDMPS        & 2.85e+00 & 1.07e+00 & 3.84e-01 \\ \hline
STOCFOR1        & ---      & ---      & 1.27e-02 \\ \hline
STOCFOR2        & ---      & ---      & 3.06e+00 \\ \hline
STOCFOR3        & ---      & ---      & 1.94e+02 \\ \hline
TRUSS           & 2.32e+01 & 4.53e+01 & 1.43e+00 \\ \hline
TUFF            & 1.49e+00 & ---      & 2.77e-01 \\ \hline
VTP{\_}BASE     & ---      & ---      & 5.42e-02 \\ \hline
WOOD1P          & ---      & 4.05e+00 & ---      \\ \hline
WOODW           & ---      & ---      & 7.22e+00 \\ \hline
\end{tabular} \caption{CPU time (in seconds) of all three indirect solvers ($\epsilon = 10^{-3}$) on NETLIB: part 2.} \label{Table:LP-NETLIB-Indirect-2}
\end{table}

\end{document}